\documentclass[a4paper,reqno]{amsart}
\usepackage[T1]{fontenc}
\usepackage[utf8x]{inputenc}
\usepackage[english]{babel}
\usepackage{amsmath,amsfonts,amssymb,upgreek,comment,mathtools}
\usepackage[11pt,curve,matrix,arrow,frame,graph]{xy}
\usepackage{graphicx}
\usepackage{hyperref}
\usepackage{enumerate}
\usepackage{xcolor}
\usepackage{esint}
\usepackage{etoolbox}

\let\oldtocsection=\tocsection
\let\oldtocsubsection=\tocsubsection
\renewcommand{\tocsection}[2]{\hspace{0em}\oldtocsection{#1}{#2}}
\renewcommand{\tocsubsection}[2]{\hspace{1.8em}\oldtocsubsection{#1}{#2}}

\usepackage{amsthm}

\theoremstyle{plain}
\newtheorem*{theorem*}{Theorem}

\newtheorem{theorem}{Theorem}[section]
\theoremstyle{definition}
\newtheorem{D}[theorem]{Definition}
\newtheorem{lemma}[theorem]{Lemma}
\newtheorem{cor}[theorem]{Corollary}
\newtheorem{prop}[theorem]{Proposition}

\theoremstyle{definition}
\newtheorem{rem}[theorem]{Remark}
\newtheorem{rems}[theorem]{Remarks}
\newcommand{\scaling}{\ensuremath{\mathrm{sc}}}
\newcommand{\R}{\ensuremath{\mathbb R}}

\newcommand{\eps}{\ensuremath{\varepsilon}}

\newcommand{\Ric}{\ensuremath{\mbox{Ric}}}

\newcommand{\setN}{\mathbb{N}}

\newcommand{\setR}{\mathbb{R}}

\newcommand{\cD}{\mathcal{D}}

\newcommand{\cS}{\mathcal{S}}

\DeclareMathOperator{\tr}{tr}

\newcommand{\di}{\mathop{}\!\mathrm{d}}

\DeclareMathOperator{\Lip}{Lip}

\newcommand{\paral}{\mathbin{\!/\mkern-5mu/\!}}

\newcommand{\Gr}{\mathcal{G}}
\newcommand{\Mix}{\mathrm{Mix}}
\newcommand{\Diag}{\mathrm{Diag}}

\newfont{\tmpf}{cmsy10 scaled 2500}

\newcommand{\uX}{\underline{X}}

\newcommand{\uF}{\underline{F}}

\def\R{\mathbb R}

\def\cC{\mathcal C}
\def\cD{\mathcal D}

\def\cG{\mathcal G}

\def\cM{\mathcal M}

\def\cO{\mathcal O}
\def\bO{\mathbf O}
\def\cS{\mathcal S}

\def\cR{\mathcal R}
\def\cQ{\mathcal Q}

\def\cU{\mathcal U}

\newcommand{\df}{\coloneqq}

\newcommand{\conv}{\mathrm{co} \, }

\newcommand{\uY}{\underline{Y}}
\newcommand{\uE}{\underline{E}}

\newcommand{\fk}{\mathfrak{F}}

\usepackage{enumitem}

\title[Critical metrics of eigenvalue functionals]{Critical metrics of eigenvalue functionals via Clarke subdifferential}

\author{Romain Petrides}
\address{Romain Petrides, Universit\'e Paris Cit\'e, Institut de Math\'ematiques de Jussieu - Paris Rive Gauche, b\^atiment Sophie Germain, 75205 PARIS Cedex 13, France}
\email{romain.petrides@imj-prg.fr}

\author{David Tewodrose}
\address{D. Tewodrose, Vrije Universiteit Brussel, Department of Mathematics and Data Science, Pleinlaan 2,
B-1050 Brussel, Belgium.}
\email{david.tewodrose@vub.be}
\date{}

\begin{document}

\maketitle

\begin{abstract}
We set up a new framework to study critical points of functionals defined as combinations of eigenvalues of operators with respect to a given set of parameters: Riemannian metrics, potentials, etc.  Our setting builds upon Clarke's differentiation theory to provide a novel understanding of critical metrics.  In particular, we unify and refine previous research carried out on Laplace and Steklov eigenvalues.  We also use our theory to tackle original examples such as the conformal GJMS operators, the conformal Laplacian, and the Laplacian with mixed boundary conditions. 
\end{abstract}

\tableofcontents
\newpage

\section{Introduction}

In Riemannian geometry, critical points of normalized eigenvalue functionals are strongly related to meaningful mappings of a manifold. This was first observed by Nadirashvili in his pioneering work \cite{Nadi} where he established that any $\overline{\lambda}_k$-critical Riemannian metric on a closed surface yields a minimal isometric immersion of the surface into some sphere.  Here $\overline{\lambda}_k$ denotes the volume normalized $k$-th non-zero eigenvalue of the Laplace-Beltrami operator $\Delta_g$ depending on a Riemannian metric $g$, where repetition is made according to the multiplicity.  Since then,  critical metrics involving various normalizations have been given a geometric interpretation in the following contexts:

\begin{itemize}
\item \textit{One Laplace eigenvalue.} El Soufi and Ilias \cite{EI1} extended Nadirashvili's theorem from surfaces to manifolds of dimension greater than or equal to three. In \cite{ElSouIlias2003,ColboisElSoufi}, they also proved that critical metrics in a conformal class correspond to harmonic mappings into spheres.
\item \textit{One Steklov eigenvalue.} Fraser and Schoen \cite{FraserSchoen} showed that on a compact surface with non-empty boundary,  a smooth Riemannian metric maximizing the normalized $k$-th Steklov eigenvalue yields a free boundary minimal immersion of the surface into a solid ball, and that maximality in a conformal class yields free boundary harmonic mappings into a ball. This was recently extended by Karpukhin and Métras \cite{KarpukhinMetras} to critical metrics in dimension $n\ge 3$ via a suitably weighted version of the Steklov eigenvalue problem. 
\item \textit{Second eigenvalue of the conformal Laplacian.} Ammann and Humbert \cite{AmmannHumbert} studied nodal Yamabe solutions associated with any generalized metric which minimizes, within a given conformal class, the normalized second eigenvalue of the conformal Laplacian on manifolds with non-negative Yamabe invariant. 
More recently, Gursky and Pérez--Ayala \cite{GPA} provided a related analytico-geometric interpretation --- involving both nodal Yamabe solutions and harmonic mappings --- for those \textit{maximizing} metrics under \textit{negative} Yamabe invariant.
\item \textit{One Paneitz eigenvalue in dimension four.} Pérez--Ayala \cite{Perez} studied the properties of critical metrics for one Paneitz eigenvalue in dimension $4$,  both within a given conformal class and on the whole set of Riemannian metrics.
\item \textit{One Dirac eigenvalue.} Ammann \cite{Ammann} showed that conformally minimal metrics for the first Dirac eigenvalue of a spin surface correspond to harmonic maps into $\mathbb{S}^2$ that can be interpreted as Gauss maps of CMC immersions into $\mathbb{R}^3$.  Karpukhin, Métras and Polterovich \cite{KMP} recently extended this result to Dirac eigenvalues of any higher index, showing that in that case conformally critical metrics are related to harmonic mappings into $\mathbb{CP}^n$.
\item \textit{Finite combinations of eigenvalues.} In \cite{PetridesEllipsoids}, the first-named author proved that critical metrics for combinations of Laplace eigenvalues are related to minimal immersions into Pseudo-Euclidean ellipsoids (Euclidean if the combination has positive partial derivatives). Analogously, he proved in \cite{PetridesSteklovellipsoids} that critical metrics for combinations of Steklov eigenvalues are related to free boundary minimal immersions into (Pseudo)-Euclidean ellipsoids.
\item \textit{Combination of Robin eigenvalues} In the very recent \cite{LimaMenezes,Medvedev}, the authors have investigated the mapping properties of metrics that are critical for some specific combination of Robin eigenvalues.
\end{itemize}

The solutions proposed for all these problems share some similarities, hence it is natural to look for a suitable common framework that could treat them all at once. This is the main goal of this paper. We also expect this new framework to apply in new situations and to yield natural tools to establish the existence of critical metrics in a general setting. Indeed, after Nadirashvili proved the existence of a maximal first Laplace eigenvalue on tori \cite{Nadi}, various variational methods have been developed to treat the case of one Laplace eigenvalue \cite{PetridesGFA} \cite{NadirashviliSire} \cite{Petrides} \cite{KNPP} \cite{KarpukhinStern}, one Steklov eigenvalue \cite{FraserSchoenInv} \cite{PetridesJDG}, the second eigenvalue of the conformal Laplacian \cite{AmmannHumbert} \cite{GPA}, the first eigenvalue of the Dirac operator \cite{Ammann} and finite combinations of eigenvalues \cite{PetridesEllipsoids} \cite{PetridesSteklovellipsoids}  \cite{PetridesVaria} \cite{Petridesnonplanarsphere} \cite{Petridesnonplanardisk}, but a common method is still to be found.\\

The starting point of our investigation is to question the notion of criticality. As well-known, branching phenomena caused by multiplicity prevent eigenvalue functionals from being differentiable. This is the reason why Nadirashvili proposed the following seminal definition:
a Riemannian metric $g$ is called $\bar{\lambda}_k$-critical if either
\begin{equation}\label{eq:criti_Nadi}
\bar{\lambda}_k(g_t) \le \bar{\lambda}_k(g_0) + o(t) \quad \text{or} \quad \bar{\lambda}_k(g_t) \ge \bar{\lambda}_k(g_0) + o(t) 
\end{equation}
is satisfied as $t \to 0$, for any smooth, area preserving, perturbation $\{g_t\}$ of $g=g_0$.  
This definition is \textit{ad hoc} to handle one eigenvalue, in particular when the critical metric $g$ comes as a minimizer/maximizer of a variational problem. However, as we will show, it does not encompass all the metrics that we would expect to be critical, especially in the case of combinations of eigenvalues
\begin{equation}\label{eq:combination}
\fk = F \circ (\bar{\lambda}_1,\ldots,\bar{\lambda}_N) \qquad \text{with } F \in \cC^1(\setR^N).
\end{equation}

We wish to interpret Nadirasvhili's definition in terms of sub/superdifferentials. Assume for a moment that $\bar{\lambda}_k$ were convex with respect to $g$. Then criticality of $g_0$ would translate into $0$ belonging to the subdifferential of $\bar{\lambda}_k$ at $g_0$. Now $\bar{\lambda}_k$ is obviously not convex in $g$, but one can still give a meaning to the sub/superdifferentials $\partial^{-}\bar{\lambda}_k$ and $\partial^{+}\bar{\lambda}_k$, see \ref{subsubsec:sub/super}. If $\mathcal{P}(g_0)$ denotes the set of perturbations of $g_0$ considered in \eqref{eq:criti_Nadi}, then:
\[
0 \in \partial^{+}\bar{\lambda}_k(g_0) \qquad \iff \qquad  \bar{\lambda}_k(g_t) \le \bar{\lambda}_k(g_0) + o(t) \quad \forall \{g_t\} \in \mathcal{P}(g_0),
\]
\[
0 \in \partial^{-}\bar{\lambda}_k(g_0) \qquad \iff \qquad  \bar{\lambda}_k(g_t) \ge \bar{\lambda}_k(g_0) + o(t) \quad \forall \{g_t\} \in \mathcal{P}(g_0).
\]

In some simple situations, however, the sub/superdifferentials of an eigenvalue may be empty; this is the case, for instance, for the second Laplace eigenvalue of the two-dimensional round sphere, see Proposition \ref{remempty}. Therefore, we rather work with the variants of sub/superdifferentials introduced by Clarke in the seventies \cite{Clarke1975}.  Up to our best knowledge, except in an interesting paper by Cox \cite{Cox} which found applications in shape optimization, this is the first time that the non-smooth analysis developed by Clarke  \cite{Clarke,RockafellerWets} is used to tackle eigenvalue differentiation problems.

Let $\Omega$ be an open subset of a Banach space $X$ and $\mathfrak{F} : \Omega \to \setR$ a function.  The generalized upper/lower directional derivatives of $\mathfrak{F}$ at $x \in \Omega$ in a direction $h \in X$ are defined as
\begin{equation} \label{eqgenderintro}
\mathfrak{F}_{+}^{\circ} (x \, ;h)\coloneqq  \limsup_{\substack{x' \to x\\t \downarrow 0}} \frac{\mathfrak{F}(x'+th)-\mathfrak{F}(x')}{t}
\end{equation}
\begin{equation} \label{eqgenderintro2}
\mathfrak{F}_{-}^{\circ} (x \, ;h)\coloneqq \liminf_{\substack{x' \to x\\t \downarrow 0}} \frac{\mathfrak{F}(x'+th)-\mathfrak{F}(x')}{t}
\end{equation}
and the Clarke sub/superdifferentials of $\mathfrak{F}$ at $x$ are
\begin{equation}
\partial_C^- \mathfrak{F}(x) \coloneqq \{ \zeta \in X^* : \mathfrak{F}_+^{\circ} (x \, ;h) \ge \langle \zeta , h \rangle \,\, \text{for any $h \in X$}\},
\end{equation}
\begin{equation}
\partial_C^+ \mathfrak{F}(x) \coloneqq \{ \zeta \in X^* : \mathfrak{F}_-^{\circ} (x \, ;h) \le \langle \zeta , h \rangle \,\, \text{for any $h \in X$}\}.
\end{equation}
We propose the following definition.
\begin{D}\label{def:critical}
We say that $x \in \Omega$ is critical for $\mathfrak{F}$ if and only if
\begin{equation}\label{eq:criti_nous}
0 \in \partial_C^- \mathfrak{F}(x) \cup \partial_C^+ \mathfrak{F}(x).
\end{equation}
\end{D}

The advantage of this definition is threefold. 

First,  when $\mathfrak{F}$ is a combination of eigenvalues like in \eqref{eq:combination}, we are able to show that $\partial_C^- \mathfrak{F}(x) \cup \partial_C^+ \mathfrak{F}(x)$ belongs to some concrete convex hull. Then the assumption \eqref{eq:criti_nous} easily provides a general Euler--Lagrange equation which can be interpreted in situations where $\Omega$ is a set of Riemannian metrics or conformal factors to deduce mapping properties of a critical point.  In this way, we obtain various geometrically or analytically meaningful mappings by eigenfunctions of manifolds into quadrics.  In addition, thanks to a so-called Mixing Lemma from linear algebra (Lemma \ref{lem:mixing}), we show that the coordinates of these mappings can be made \textit{orthogonal} with masses depending explicitly on the combination $F$: see Theorem \ref{th:main_cri}. All this relies on the fact that the Clarke sub/superdifferentials behave well under combinations. Indeed,  they satisfy a chain rule which writes as
\begin{equation} \label{eq:chainruleintro}
\partial_C^{\pm} \mathfrak{F}(x) \subset \sum_{i=1}^N \partial_i F\left(\bar{\lambda}_1(x),\cdots,\bar{\lambda}_N(x)\right) \partial_C^{\pm}\bar{\lambda}_i(x)
\end{equation}
when $\partial_i F\ge 0$ (see \eqref{eq:chain_rule_Clarke} for the general statement), so that we need only investigating the inclusion properties of each single $\partial_C^{\pm}\bar{\lambda}_i(x)$.

Secondly, the Clarke sub/superdifferentials are never empty, and they always contain the classical ones, see \ref{subsubsec:Clarkesub/super}.  In this regard, our study reveals that in the presence of a spectral gap $\bar{\lambda}_{k+1} - \bar{\lambda}_{k-1} >0$, even if convexity may still fail,  the classical sub/superdifferential of $\bar{\lambda}_k$ coincides with the Clarke one.  This yields a substantial improvement of our results: since the classical sub/superdifferentials satisfy the opposite chain rule inclusion (see \eqref{eq:chain_rule}),  the inclusion \eqref{eq:chainruleintro} becomes an equality and we obtain a reciprocal of the previous property: if there exists one of the aforementioned mappings into a quadric associated to a point $x$, then $x$ is critical for some well chosen combination $F$ (see Theorem \ref{th:converse} and all its applications).

Thirdly,  Definition \ref{def:critical} paves the way towards a definition of Palais--Smale sequence which would be the first step in building a general existence theory for critical metrics initiated by \cite{PetridesVaria}. Indeed,  for $\mathcal{C}^1$ functionals $E:X\to \R$, Palais--Smale sequences $(x_n)$ are classically defined as those for which $DE(x_n) \to 0$ in $X^*$ as $n\to+\infty$. The sub/superdifferential point of view gives us the opportunity to generalize this definition to non-differentiable spectral functionals $\mathfrak{F}$ by demanding that $\mathrm{dist}_{X^*}(0, \partial_C^{\pm}\mathfrak{F}(x_n)) \to 0 $ as $n\to +\infty$.

\medskip

Our results are based on a new explicit computation (Theorem \ref{th:main}) of generalized limits of suitable difference quotients for eigenvalues $\lambda_k$ defined in a broad setting.  These eigenvalues are obtained from an abstract Rayleigh quotient $\cR$ by usual min-max formula,  and we notably assume a compact embedding between the natural Hilbert spaces associated with $\cR$.  In this very general context, we prove that eigenvalues always have left and right classical directional derivatives, and we also compute their Clarke generalized derivatives thanks to a key result of Ioffe \cite{Ioffe2}.

We also give an explicit computation of the classical sub/superdifferential that clearly suggests that these sets are empty in numerous situations, see Proposition \ref{propclassicalsubdifferential}. This advocates in favor of Clarke's sub/superdifferentials to define criticality.  However, if one knows that the classical sub/superdifferential is non-empty (for instance if $x$ is an extremal point),  then we obtain much more information. For instance, we reinterpret the following property in terms of sub/superdifferentials: if $\bar{\lambda}_j(x) > \bar{\lambda}_i(x)$ for $j>i$, where $x$ is a local minimum of $\bar{\lambda}_j$, then one can take at most $j-i$ eigenfunctions in the mapping that appears in the Euler-Lagrange equation.  This property was used in the context of minimization of the first Dirac eigenvalues \cite{Ammann} and of the second eigenvalue of the conformal Laplacian \cite{AmmannHumbert}.

\medskip

The last sections of the paper are devoted to concrete applications:

 In Section \ref{sec:surfaces}, we reinterpret and complete the characterization of critical metrics of \cite{PetridesEllipsoids} \cite{PetridesSteklovellipsoids} for combinations of normalized Steklov and Laplace eigenvalues on surfaces. In particular, we prove a converse : any minimal immersion (resp. with free boundary) into some Pseudo-Euclidean quadric corresponds to a metric which is critical for appropriate choices of combination $F$ of rescaled Laplace (resp. Steklov) eigenvalues. 
 
 In Section \ref{sec:high} we give a result for combinations of Steklov and Laplace eigenvalues in higher dimension in the same spirit as \cite{KarpukhinMetras}. 
 
In Section \ref{sec:conf}, we assume that the numerator $G$ of the Rayleigh quotient $\cR$ does not depend on the set of variations $\beta$ while the denominator $Q$ only depends linearly on it, i.e.~$ Q(\beta,\phi) = \int_M \phi^2 \beta$.  With a natural choice of scaling $\Vert \beta \Vert_{L^p}$, we obtain that critical points correspond to $L$-harmonic mappings into quadrics, where $L$ is the operator naturally associated with $G$. For instance, $G(\phi) = \int_M \left\vert \nabla \phi \right\vert^2$ gives harmonic maps, $G(\phi) = \int_M \left\vert \Delta \phi \right\vert^2$ gives bi-harmonic maps, etc.  Thanks to this simple general result, we characterize the critical metrics in a conformal class of numerous conformally covariant operators,  like the conformal Laplacian,  the Paneitz operator, and more generally the GJMS operators of any order. 

In Section \ref{sec:conformal}, we exploit the explicit expression of the conformal Laplacian to investigate the properties of critical points in the whole set of Riemannian metrics for any combination of normalized eigenvalues.

Lastly, Section \ref{sec:mixed} provides an illustration of the flexibility of our approach, as we investigate Laplace functionals depending on different boundary conditions on a surface $\Sigma$ with a non-empty boundary $\partial\Sigma$ written as a disjoint union of finitely many connected curves: $\partial \Sigma = \bigsqcup_{k=1}^l A_i$. We define $\bar{\lambda}_k^{A_i}$ as the $k$-th normalized eigenvalue with Dirichlet and Neumann boundary condition on $A_i$ and $\partial \Sigma \setminus A_i$ respectively. We also denote by $\bar{\lambda}_k^{\mathcal{N}}:= \bar{\lambda}_k^{\emptyset}$ and $\bar{\lambda}_k^{\mathcal{D}}:= \lambda_k^{\partial\Sigma}$ the normalized Neumann and Dirichlet $k$-th non-zero eigenvalue. Then a critical metric for a combination of $\bar{\lambda}_1^{A_i}$ for $i=1,\cdots,l$ and of $\bar{\lambda}_1^{\mathcal{N}},\cdots,\bar{\lambda}_k^{\mathcal{N}}$ gives a proper free boundary minimal immersion into a $2^{-i}$-th piece of an ellipsoid.  Our Theorem \ref{th:main8} unifies Hersch's celebrated ``four properties'' \cite{hersch} seen as minimization problems into one single criticality property:
\begin{itemize}
\item If $g$ is a metric on a sphere $\mathbb{S}^2$, then
$$ \frac{1}{\bar{\lambda}_1(g)} + \frac{1}{\bar{\lambda}_2(g)} + \frac{1}{\bar{\lambda}_3(g)} \geq \frac{3}{8\pi} $$
with equality if and only if $(\mathbb{S}^2,g)$ is isometric to a round sphere $\mathbb{S}^2$.  This corresponds to the minimal immersions into ellipsoids associated to critical metrics of Laplace eigenvalues in the closed case (first observed in \cite{PetridesEllipsoids}).
\item If $g$ is a metric on a disk $\mathbb{D}$, then
$$ \frac{1}{\bar{\lambda}_1^{\mathcal{D}}(g)} + \frac{1}{\bar{\lambda}_1^{\mathcal{N}}(g)} + \frac{1}{\bar{\lambda}_2^{\mathcal{N}}(g)} \geq \frac{3}{4\pi} $$
with equality if and only if $(\mathbb{D},g)$ is isometric to a half-sphere $\mathbb{S}^2_+$. This corresponds to free boundary minimal immersions into half ellipsoids associated to critical metrics for combinations of $\bar{\lambda}_1^{\mathcal{D}}$ and $\bar{\lambda}_1^{\mathcal{N}},\cdots,\bar{\lambda}_k^{\mathcal{N}}$ (with $k=2$).
\item If $g$ is a metric on a disk $\mathbb{D}$ and $\partial\mathbb{D} = A_1 \sqcup A_2$, then
$$ \frac{1}{\bar{\lambda}_1^{A_1}(g)} + \frac{1}{\bar{\lambda}_1^{A_2}(g)} + \frac{1}{\bar{\lambda}_1^{\mathcal{N}}(g)} \geq \frac{3}{2\pi} $$
with equality if and only if $(\mathbb{D},g)$ is isometric to a quadrant of the unit sphere.  This corresponds to free boundary minimal immersions into quadrants of ellipsoids associated to critical metrics for combinations of $\bar{\lambda}_1^{A_1},\bar{\lambda}_1^{A_2}$ and $\bar{\lambda}_1^{\mathcal{N}},\cdots,\bar{\lambda}_k^{\mathcal{N}}$ (with $k=1$).
\item If $g$ is a metric on a disk $\mathbb{D}$ and $\partial\mathbb{D} = A_1 \sqcup A_2 \sqcup A_3$ then
$$ \frac{1}{\bar{\lambda}_1^{A_1}(g)} + \frac{1}{\bar{\lambda}_1^{A_2}(g)} + \frac{1}{\bar{\lambda}_1^{A_3}(g)} \geq \frac{3}{\pi} $$
with equality if and only if $(\mathbb{D},g)$ is isometric to an octant of the unit sphere. This corresponds to free boundary minimal immersions into octants of ellipsoids associated to critical metrics for combinations of $\bar{\lambda}_1^{A_1},\bar{\lambda}_1^{A_2}, \bar{\lambda}_1^{A_3}$ and $\bar{\lambda}_1^{\mathcal{N}},\cdots,\bar{\lambda}_k^{\mathcal{N}}$ (with $k=0$).
\end{itemize}

\subsubsection*{Acknowledgments} Both authors are grateful to Asma Hassannezhad and Bruno Premoselli for their interest in this work. They would also like to thank the anonymous referee for the suggestions for  improvement. The second author was supported by Laboratoire de Mathématiques Jean Leray via the project Centre Henri Lebesgue, by Fédération de recherche Mathématiques de Pays de Loire via the project Ambition Lebesgue Loire, and by the Research Foundation – Flanders (FWO) via the Odysseus programme Geometric and analytic properties of metric measure spaces with spectral curvature constraints, with applications to manifold learning (G0DBZ23N).

\section{Preliminaries}\label{sec:prelim}

The vector spaces considered in this paper are all over the field of real numbers.  Whenever we use a letter, say $X$, to refer to a vector space,  we underline this letter to denote the set obtained upon removing the zero vector:\[\uX\coloneqq X \backslash \{0\}.\] 
We set $\setN^* \df \setN \backslash \{0\}$ and $\overline{\setR} \df \setR \bigcup \{\pm \infty\}$.  We use the following convention on the notations $\pm$, $\mp$ and $\le_\pm$:
\begin{itemize}
\item If $\pm$ appears in a formula, it means that the formula holds with all $\pm$ replaced by $+$ and with all $\pm$ replaced by $-$. 
\item If $\pm$ and $\mp$ appears in a formula, then $\mp$ must be read as $-$ if $\pm$ is $+$ and as $+$ if $\pm$ is $-$. 
\item If $\pm$ and $\le_\pm$ appears in a formula,  then $\le_\pm$ must be read as $\le$ when $\pm$ is $+$ and as $\ge$ when $\pm$ is $-$.
\end{itemize}

If $M$ is a smooth manifold, then we write $\mathcal{R}(M)$ for the space of smooth Riemannian metrics on $M$. For any integer $m\ge 2$,  we let $\mathcal{R}^m(M)$ (resp.~$\mathcal{S}^m(M)$) be the space of Riemannian metrics (resp.~of symmetric $(2,0)$ tensors) of class $\cC^m$  on $M$.  

Assume that $M$ is compact and connected. Then any $g,g' \in \cR^2(M)$ are such that $c^{-1}g \le g' \le cg$ for some $c \ge 1$.  This has the following consequences for section and function spaces on $M$.  The usual $\cC^m$ norms on $\cS^m(M)$ (resp.~$\cC^m(M)$), which depend on the choice of a Riemannian metric,  are all equivalent one to another, and the space $\cS^m(M)$ (resp.~$\cC^m(M)$) endowed with any such a norm is a Banach space of which $\cR^m(M)$ (resp.~$\cC_{>0}^m(M)$) is an open subset.   Moreover,  neither the space $L^2(M)$ of square integrable (equivalent classes of) functions nor the Sobolev space $H^1(M)$ of square integrable weakly differentiable (equivalent classes of)  functions with square integrable differential norm depends on the choice of a Riemannian metric on $M$.  For any $g \in \mathcal{R}^2(M)$, the associated $L^2$ norm of $u \in L^2(M)$ is defined as
\[
\|u\|_{L^2(g)} \df \left(\int_M u^2 \, \di v_g\right)^{1/2}
\]
and the $H^1$ norm of $v \in H^1(M)$ is 
\[
\|v\|_{H^1(g)} \df \left(\|v\|_{L^2(g)} + \||dv|_g\|_{L^2(g)} \right)^{1/2}.
\]
For any $g \in \cR^m(M)$ and $h,h' \in \cS^m(M)$, we let $\langle h,h'\rangle_g$ denote the $\cC^m$ function expressed as $h_{ik} h'_{\ell m} g^{il} g^{km}$ in any local coordinate system.  Recall that the conformal class of $g \in \cR(M)$ is defined as
\[
[g] \df \{f g \, : \,  f \in \cC_{>0}^\infty(M)\}.
\]

\subsection{Linear algebra} Let us consider a vector space $X$ that we keep fixed until the end of this subsection.  For any $E \subset X$, we let $\conv  E$ be the convex hull of $E$. We recall that Minkowski linear combination behaves well with respect to taking convex hulls, that is to say, for any $E,E' \subset X$ and $\lambda \in \setR$,
\begin{equation}\label{eq:Minkowski}
\conv ( E + \lambda E' ) = \conv E + \lambda \conv  E'.
\end{equation}
For any positive integer $k$, we let $\Gr_{k}(X)$ be the $k$-th grassmannian of $X$ which is the space of all the $k$-dimensional linear subspaces of $X$.  We adopt the convention that $\Gr_0(X)=\{0\}$. 

Let $A$ be a map from $X$ to another vector space $Y$.  We say that $A$ is quadratic if the map $B : X \times X \ni (u,v) \mapsto (A(u+v)-A(u)-A(v))/2$ is bilinear and $A(\lambda u)=\lambda^2 A(u)$ for any $u \in X$ and $\lambda \in \setR$.  We say that a family $(u_k)$ of elements in $X$ is $A$-orthogonal if $B(u_k,u_{k'})=0$ for any $k\neq k'$. We say that $Q:X \to [0,+\infty]$ is a non-negative definite quadratic form if its domain $\cD(Q) = \{ Q < +\infty\}$ is a linear subspace of $X$ and if the restriction of $Q$ to $\cD(Q)$ is a quadratic map. We also say that a family $(u_k)$ of elements in $E$ is $Q$-orthonormal if it is $Q$-orthogonal and such that $Q(u_k)=1$ for any $k$.\\

Assume that $X$ is a Banach space with norm $\|\cdot\|_X$.  For any $x \in X$ and $\eps >0$ we let $B_X(x,\eps) \df \{y \in X  \, : \, \|x-y\|_X < \eps\}$ be the open ball centered at $x$ with radius $\eps$. For any subspace $E\subset X$, we set \[S(E) \coloneqq \{u \in E : \|u\|_X = 1\}.\]
We let $X^*$ be the topological dual of $X$. We denote by $\langle \cdot, \cdot \rangle : X^* \times X \to \setR$ the duality pairing and by $\|\cdot\|_{X^*}$ the dual norm defined by
\[
\|\zeta\|_{X^*} \df \sup_{u \in S(X)} \langle \zeta, u \rangle
\]
for any $\zeta \in X^*$. We recall that this norm defines the strong topology of $X^*$. We say that a quadratic form $Q$ on $X$ is densely defined if its domain $\cD(Q)$ is dense in $X$. For any subspace $E\subset X$,  we set
 \[S_Q(E)\coloneqq \{u \in E : Q(u) = 1\}.\]
If $\cD(Q)=X$ and $Q$ is positive definite, then the associated bilinear form $B$ is a scalar product on $X$. In particular,  $Q^{1/2}$ defines a norm on $X$. We say that a linear subspace $E\subset X$ is $Q^{1/2}$-dense in $X$ if for any $x \in X$ there exists a sequence $\{x_i\} \subset E$ such that $Q^{1/2}(x-x_i) \to 0$. We let $\|\cdot\|_{Q^{*}}$ be the norm on $X^*$ defined by
\[
\|\zeta\|_{Q^{*}} \df \sup_{u \in S_Q(X)} \langle \zeta , u \rangle
\]
for any $\zeta \in X^*$.\\  

Assume now that $X$ is a Hilbert space.  We denote by $\langle \cdot, \cdot \rangle_X$  the scalar product associated with $\|\cdot \|_X$, by $\cO(X)$ the set of orthogonal bases of $X$, and by $\bO(X)$ the set of orthonormal bases.  By $\bO(\setR^m)$ we mean the set of orthonormal bases with respect to the classical Euclidean scalar product. We also denote by $\cM_p(\setR)$ the space of $p\times p$ matrices with real entries.  If $A \in \cM_p(\setR)$ then $A^T$ denotes its transpose. If $d=(d_1,\ldots, d_m) \in \R^m$, we let $\Diag(d)$ be the diagonal matrix with $i$-th diagonal entry equal to $d_i$.  We use the notation $\mathfrak{S}_m$ to refer to the group of permutations of $\{1,\ldots,m\}$. If $\sigma \in \mathfrak{S}_m$ then we set $d_\sigma = (d_{\sigma(1)},\ldots,d_{\sigma(m)})$.  Lastly, we set
\begin{equation}\label{eq:mix}
\Mix(d) \df \conv\{ d_\sigma : \sigma \in \mathfrak{S}_m\} .
\end{equation}

The following elementary lemma is crucial for our work.

\begin{lemma}[Mixing lemma]\label{lem:mixing} Let $F$ be a Euclidean vector space of dimension $m$ and $A$ a quadratic map from $F$ to another vector space.  Then for any $d=(d_1,\ldots, d_m) \in \R^m$, 
\begin{align}\label{eq:mixing}
&  \phantom{=} \phantom{=} \mathrm{co} \left\{ \sum_{k=1}^{m} d_k A(u_k): (u_k) \in \bO(F)\right\} \\
& = \left\{ \sum_{\ell=1}^m \tilde{d}_\ell A(u_\ell) : (u_\ell) \in \bO(F) \text{ and } \tilde{d} \in  \Mix(d) \right\} \cdot
\end{align}
\end{lemma}

\begin{proof} We need only proving the direct inclusion because the converse is trivial. Consider
\[
v = \sum_{\alpha} t_\alpha \sum_{k=1}^m d_k A(u_k^\alpha) \in \mathrm{co} \left\{ \sum_{k=1}^{m} d_k A(u_k): (u_k) \in \ \bO(F) \right\}
\]
and $(e_1,\ldots,e_m) \in \bO(F)$. For any $\alpha$ let $O^\alpha \in \bO(\setR^m)$ be such that
\begin{equation}\label{eq:1}
u_k^\alpha= \sum_{j=1}^m O^\alpha_{jk} e_j
\end{equation}
for any $k \in \{1,\ldots,m\}$.  The spectral theorem applied to the symmetric matrix 
\[\sum_{\alpha} t_\alpha O^\alpha\, \mathrm{Diag}(d)\, (O^\alpha)^T\]
implies that there exist $O \in \bO(\setR^m)$ and $\tilde{d} \in \mathbb{R}^m$ such that
\[
\Diag(\tilde{d})  = \sum_{\alpha} t_\alpha P^\alpha\, \mathrm{Diag}(d)\, (P^\alpha)^T
\]
where we have set $P^\alpha \df O^T O^\alpha$ for any $\alpha$.  Then
\begin{equation}\label{eq:3}
\tilde{d}_j = \sum_\alpha t_\alpha \sum_{k=1}^m d_k (P_{jk}^\alpha)^2
\end{equation}
for any $j \in \{1,\ldots,m\}$. For any $\alpha$, let $Q^\alpha$ be the matrix whose $(i,j)$-th entry is $(P_{i,j}^{\alpha})^2$. Then
\[
\tilde{d} = Q d
\]
with $Q \df \sum t_\alpha Q^{\alpha}$. As a convex combination of doubly stochastic matrices, $Q$ is doubly stochastic, hence $\tilde{d} \in \conv\{ d_\sigma : \sigma \in \mathfrak{S}_m\}$ by Birkhoff's theorem (see e.g.~\cite[Theorem 1]{Horn}). Moreover,  for any $\ell,\ell' \in \{1,\ldots,m\}$ such that $\ell \neq \ell'$,
\begin{equation}\label{eq:4}
0 = \sum_\alpha t_\alpha \sum_{k=1}^m d_k P_{\ell k}^\alpha P_{\ell' k}^\alpha.
\end{equation}
Let $(u_1,\ldots,u_m) \in \bO(F)$ be defined by
\[
u_\ell = \sum_{j=1}^m O_{j \ell} e_j
\]
for any $\ell \in \{1,\ldots,m\}$.  For any $\alpha$, since $OP^\alpha = O^\alpha$, we know that
\begin{equation}\label{eq:2}
O^\alpha_{jk} = \sum_{\ell = 1}^m O_{j \ell} P^\alpha_{\ell k}
\end{equation}
for any $j,k \in \{1,\ldots,m\}$. Then we have the following, where $B$ is the bilinear map associated to $A$:
\begin{align*}
v & = \sum_{\alpha} t_\alpha \sum_{k=1}^m d_k \, A(u_k^\alpha) & \\
& = \sum_{\alpha} t_\alpha \sum_{k=1}^m d_k \, A \left(  \sum_{j=1}^m O^\alpha_{jk} e_j \right) & \text{by \eqref{eq:1}}\\
& = \sum_{\alpha} t_\alpha \sum_{k=1}^m d_k \, A \left(  \sum_{j,\ell =1}^m O_{j \ell} P^\alpha_{\ell k} e_j \right) & \text{by \eqref{eq:2}} \\
& = \sum_{\alpha} t_\alpha \sum_{k=1}^m d_k \, B \left(  \sum_{j,\ell =1}^m O_{j \ell} P^\alpha_{\ell k} e_j,  \sum_{j',\ell' =1}^m O_{j' \ell'} P^\alpha_{\ell' k} e_{j'}\right) \\
& = \sum_{\alpha} t_\alpha \sum_{k=1}^m d_k \sum_{j,j', \ell,\ell'=1}^m P_{\ell k}^\alpha P_{\ell' k}^\alpha O_{j\ell} O_{j'\ell'}B (e_j, e_{j'})\\
& = \sum_{j,j', \ell,\ell'=1}^m \left( \sum_{\alpha} t_\alpha \sum_{k=1}^m d_kP_{\ell k}^\alpha P_{\ell' k}^\alpha \right)O_{j\ell} O_{j'\ell'}B (e_j, e_{j'})\\
& = \sum_{j,j', \ell = 1}^m \tilde{d}_\ell O_{j\ell}O_{j\ell'} B (e_j, e_{j'}) & \text{by \eqref{eq:3} and \eqref{eq:4}}\\
& = \sum_{\ell = 1}^m \tilde{d}_\ell B\left( \sum_{j=1}^m O_{j \ell} e_j, \sum_{j'=1}^m O_{j' \ell} e_j\right)\\
& = \sum_{\ell = 1}^m \tilde{d}_\ell A(u_\ell).
\end{align*}
\end{proof}

\subsection{Non-smooth analysis}

Let us provide some elements from non-smooth analysis.  Our presentation is inspired by \cite{Clarke}.  We consider a Banach space $X$, an open set $\Omega \subset X$ and a locally Lipschitz function \[\mathfrak{F} : \Omega \to \setR.\] We recall that the local Lipschitz constant of $\fk$ at $x \in \Omega$ is defined as
\[
\mathrm{Lip}_x(\fk) \df \limsup_{x' \to x} \frac{|\fk(x)-\fk(x')|}{\|x'-x\|_X} \, \cdot
\]

\subsubsection{Fréchet differentiability} We begin with recalling the following classical definition.

\begin{D}
We say that $\mathfrak{F}$ is Fréchet differentiable at $x \in \Omega$ if there exists $\mathfrak{F}_x(x) \in X^*$ such that
\[
\lim\limits_{h \to 0} \frac{\mathfrak{F}(x+h)-\mathfrak{F}(x) - \langle \mathfrak{F}_x(x), h\rangle }{\|h\|_{X}}  = 0.
\]
We say that $\mathfrak{F}$ is $\mathcal{C}^1$-Fréchet on an open set $V \subset \Omega$ if it is Fréchet differentiable at any $x \in V$ and if the restriction of the map $\mathfrak{F}'$ to $V$ is continuous with respect to the strong topology of $X^*$.
\end{D}

If $\mathfrak{F}$ is $\mathcal{C}^1$-Fréchet in a neighborhood $V$ of a point $x \in \Omega$, then the following version of the fundamental theorem of calculus holds: 
\begin{align}\label{eq:DL1}
\mathfrak{F}(x+th) & = \mathfrak{F}(x)  + t \int_0^1 \langle \mathfrak{F}_x(x + sth), h\rangle \di s
\end{align}
for any $x\in V$, $h \in X$ and $t>0$ such that $x+th \in V$, and
\begin{align}\label{eq:reste_integral}
\int_0^1 \langle \mathfrak{F}_x(y_t + sth), h\rangle \di s =   \langle \mathfrak{F}_x(x), h\rangle  + o(1)
\end{align}
as $t \downarrow 0$,  for any $\{y_t\}_{t>0} \subset V$ such that $y_t \to x$ as $t \downarrow 0$.

\begin{rem}
If $E$ is a vector space and $\mathcal{R} : \Omega \times E \to \setR$ is such that for any $u \in E$, the map $\mathcal{R}(\cdot,u) : \Omega \to \setR$ is Fréchet différentiable at any point in $\Omega$, then we use a slight abuse of notation by letting
\[
\mathcal{R}_x :  \Omega \times E \to X^*
\]
be such that $\mathcal{R}_x(x,u)$ denotes the Fréchet derivative of $\mathcal{R}(\cdot,u)$ at $x$.
\end{rem}

\subsubsection{Directional derivatives and classical sub/super differentials}\label{subsubsec:sub/super} The next definition is classical in convex analysis.  Note that here, however, we do not assume any convexity on $\mathfrak{F}$.

\begin{D}
We say that $\mathfrak{F}$ has a right directional derivative denoted $\mathfrak{F}_{r}'(x \, ;h)$ at $x \in \Omega$ in the direction $h \in X$ if the following limit exists
\[
\mathfrak{F}_r'(x;h) := \lim\limits_{t \downarrow 0} \frac{\mathfrak{F}(x+th)-\mathfrak{F}(x)}{t} \, \cdot
\]
Likewise, we say that $\mathfrak{F}$ has a left directional derivative  at $x$ in the direction $h$ if
\[
\mathfrak{F}_\ell'(x;h) := \lim\limits_{t \uparrow 0} \frac{\mathfrak{F}(x+th)-\mathfrak{F}(x)}{t}
\]
exists.  If both $\mathfrak{F}_\ell'(x;h)$ and $\mathfrak{F}_r'(x;h)$ exist and are equal, then we say that $\mathfrak{F}$ is directionally differentiable at $x$ in the direction $h$.
\end{D}

\begin{rems}
\hfill

\begin{enumerate}
\item A simple change of variable $s=-t$ shows that $\mathfrak{F}_\ell'(x;h)$ exists if and only if $\mathfrak{F}_r'(x;-h)$ does, in which case
\begin{equation}\label{eq:left_and_right}
\mathfrak{F}_\ell'(x;h) = - \mathfrak{F}_r'(x;-h).
\end{equation}

\item For any $x \in \Omega$ and $h \in X$, if $\mathfrak{F}$ admits a right/left derivative at $x$ in the direction $h$, then
\[
 \left| \mathfrak{F}_r'(x;h) \right| \le \Lip_x(\fk) \quad \text{and} \quad  \left| \mathfrak{F}_\ell'(x;h) \right| \le \Lip_x(\fk).
\]

\end{enumerate}
\end{rems}

From the notion of right-directional derivative, one can define the sub/superdifferential, as ensured by the following elementary lemma of which we provide a proof for the reader's convenience.

\begin{lemma}\label{lem:elem}
For any bounded $\phi : X \to \mathbb{R}$, the sets $$D^\pm_\phi \coloneqq \{ \zeta \in X^* \, : \, \langle \zeta,h \rangle \le_\pm \phi(h) \, \, \text{for any $h \in X$}\}$$ are convex and weak* compact. Moreover, if $D^\pm_\phi$ is non-empty, then $\pm\phi$ is its support function, that is to say, 
\[
\pm \phi(h) = \sup_{\zeta \in D^\pm_\phi} \langle \zeta, h \rangle
\]
for any $h \in X$; as a conseqeunce, $D^\pm_\phi$ is uniquely determined by $\phi$.
\end{lemma}

\begin{proof}
We prove the result only for $\pm=+$. The other case is obtained in a similar way.  Convexity is a simple consequence of linearity of the duality pairing. Indeed, consider $\zeta, \zeta' \in D_\phi$ and $t \in [0,1]$.  Then for any $h \in X$, 
\[
\langle t \,\zeta + (1-t) \, \zeta',h \rangle = t \langle \zeta,h \rangle + (1-t)\langle \zeta',h \rangle \le  t \phi(h) + (1-t)\phi(h) = \phi(h).
\]
Let us prove that $D_\phi$ is weak* closedness. If $\zeta_n \stackrel{*}{\rightharpoonup} \zeta$ for $(\zeta_n) \subset D_\phi$ and $\zeta \in X$, then $\phi(h) \ge \langle \zeta_n , h \rangle \to \langle \zeta , h \rangle$ for any $h \in X$, so that $\zeta \in D^+_\phi$. Lastly, if $\phi$ is bounded, then  $\langle \zeta, h \rangle \le \phi(h) \le \sup_X |\phi|$ and
$- \langle \zeta, h \rangle = \langle \zeta,  -h \rangle  \le \phi(-h) \le \sup_X |\phi|$ for any $\zeta \in D_\phi^+$, so that
\[
\|\zeta\|_{X^*} = \sup_{h \in SX} |\langle \zeta, h \rangle| \le \sup_X |\phi|.
\]
Thus $D^+_\phi$ is bounded in the strong topology of $X^*$, hence it is weakly* compact by the Banach--Alaoglu--Bourbaki theorem. The last statement follows from \cite[Proposition 1.3 in Chapter 2]{Clarke}. 
\end{proof}

Applying the previous lemma to a right-directional derivative function $\fk_r'(x,\cdot)$, we can define the following.

\begin{D}
If $\mathfrak{F}$ admits a right directional derivative at $x$ in any direction $h$, then the subdifferential of $\mathfrak{F}$ at $x$ is
\[
\partial^- \mathfrak{F} (x) \coloneqq \{ \zeta \in X^* \, : \, \langle \zeta,h \rangle \le \mathfrak{F}_r'(x,h) \, \, \text{for any $h \in X$}\}
\]
and the superdifferential of $\mathfrak{F}$ at $x$ is
\[
\partial^+ \mathfrak{F} (x) \coloneqq \{ \zeta \in X^* \, : \, \langle \zeta,h \rangle \ge \mathfrak{F}_r'(x,h) \, \, \text{for any $h \in X$}\}.
\]
\end{D}

\begin{rem}
Even if $\mathfrak{F}$ admits a right directional derivative at $x$ in any direction $h$, it may happen that $\partial^- \mathfrak{F} (x)$ and $\partial^+ \mathfrak{F} (x)$ are both empty: we provide an example in Proposition \ref{remempty}.
\end{rem}

The following proposition collects several properties of the right directional derivative and the corresponding sub/superdifferential. 

\begin{prop}\label{prop:collection}
Assume that $\fk$ admits a right-directional derivative at $x$ in any direction $h$.

\begin{enumerate}
\item If $\mathfrak{F}$ is $\mathcal{C}^1$-Fréchet on a neighborhood of $x$, then
\begin{equation}\label{eq:direct_Fréchet_rule}
\fk'_r(x,h) = \langle \mathfrak{F}_x(x), h \rangle
\end{equation}
for any $h \in X$, and
\begin{equation}\label{eq:Fréchet_rule}
\partial^\pm \mathfrak{F}(x) = \{\mathfrak{F}_x(x)\}.
\end{equation}

\item For any $\lambda\in \setR$,
\[
[\lambda \fk]'_r(x,h) = \lambda \fk'_r(x,h) 
\]
for any $h \in X$, and:
\begin{align*}\label{eq:homo_rule}
\lambda >0 & \quad \Rightarrow \quad \partial^{\pm} [\lambda \mathfrak{F}](x) = \lambda \partial^{\pm} \mathfrak{F}(x)\\
 \lambda < 0 & \quad \Rightarrow \quad \partial^{\pm} [\lambda \mathfrak{F}](x) = \lambda \partial^{\mp} \mathfrak{F}(x).
\end{align*}

\item If $\mathfrak{G} : \Omega \to \setR$ is another locally Lipschitz function which admits a right-directional derivative at $x$ in any direction $h$, then
\begin{equation}\label{eq:direct_sum_rule}
[\mathfrak{F} + \mathfrak{G}]_r'(x,h) = \mathfrak{F}_r'(x,h) +  \mathfrak{G}_r'(x,h),
\end{equation}
\begin{equation}\label{eq:direct_product_rule}
[\mathfrak{F} \mathfrak{G}]_r'(x,h) = \mathfrak{F}(x) \mathfrak{G}_r'(x,h) +  \mathfrak{G}(x) \mathfrak{F}_r'(x,h)
\end{equation}
for any $h \in X$, and
\begin{equation}\label{eq:sum_rule}
\partial^{\pm} (\mathfrak{F} + \mathfrak{G})(x) \supset  \partial^{\pm} \mathfrak{F} (x) + \partial^{\pm} \mathfrak{G}(x),
\end{equation}
\begin{equation}\label{eq:product_rule}
\partial^{\pm} (\mathfrak{F} \mathfrak{G})(x) \supset \mathfrak{F}(x) \partial^{\pm} \mathfrak{G} (x) + \mathfrak{G}(x) \partial^{\pm} \mathfrak{F}(x).
\end{equation}

\item If $\fk = F \circ (\lambda_1,\ldots, \lambda_N)$ where $F \in \cC^1(\setR^N)$ and $\lambda_1,\ldots, \lambda_N : \Omega \to \setR$ all admit a directional derivative at $x$ in a given direction $h \in X$, then so does $\fk$, and
\begin{equation}\label{eq:direct_chain_rule}
\fk_r'(x,h) = \sum_{k=1}^N d_k(x) [\lambda_k]_r'(x,h)
\end{equation}
where
\[
d_k(x) \df \partial_k F(\lambda_1(x),\ldots, \lambda_N(x)).
\]
In addition, if $\lambda_1,\ldots, \lambda_N : \Omega \to \setR$ all admit a directional derivative at $x$ in any direction $h \in X$, then
\begin{equation}\label{eq:chain_rule}
\partial^\pm \fk(x) \supset \sum_{\substack{k=1\\d_k(x)>0}}^N d_k(x)\partial^\pm \lambda_k(x) +  \sum_{\substack{k=1\\d_k(x)<0}}^N d_k(x)\partial^\mp \lambda_k(x)
\end{equation}
\end{enumerate}
\end{prop}

\begin{proof}
We do not provide a proof for (1) and (2) since they easily follow from the definitions. Let us prove (3) and (4). Observe first that \eqref{eq:direct_sum_rule} is a direct consequence of the fact that for any $t>0$ and $h \in X$,
\[
\frac{[\fk + \mathfrak{G}](x+th)-[\fk + \mathfrak{G}](x)}{t}  = \frac{\fk(x+th)-\fk(x)}{t} + \frac{\mathfrak{G}(x+th)-\mathfrak{G}(x)}{t}\, \cdot
\]
We now prove \eqref{eq:sum_rule} in the case $\pm=-$. Consider $\zeta = \zeta_1 + \zeta_2$ where $\zeta_1 \in \partial^{-}\fk(x)$ and $\zeta_2 \in \partial^{-}\mathfrak{G}(x)$. Then for any $h \in X$,
\[
\langle \zeta, h \rangle = \langle \zeta_1, h \rangle + \langle \zeta_2, h \rangle \le \fk'_r(x,h) + \mathfrak{G}'_r(x,h) = [\fk + \mathfrak{G}]'_r(x,h).
\]
This implies that $\zeta \in \partial^- [\fk + \mathfrak{G}](x)$. The case $\pm=-$ is handled in a similar way.   In the same way as \eqref{eq:direct_sum_rule}, \eqref{eq:direct_product_rule} and \eqref{eq:direct_chain_rule} are obtained by writing a suitable formula for any $t>0$ and $h \in X$ and letting $t \to 0$. Then \eqref{eq:product_rule} and \eqref{eq:chain_rule} follow from these two properties in the same lines as \eqref{eq:sum_rule}  follows from \eqref{eq:direct_sum_rule}.
\end{proof}

\subsubsection{Generalized directional derivatives and Clarke sub/super differentials}\label{subsubsec:Clarkesub/super}

In \cite{Clarke1975}, Clarke introduced the following variants of the directional derivatives.

\begin{D} The generalized upper directional derivative of $\mathfrak{F}$ at $x  \in \Omega$ in the direction $h \in X$ is defined as
\begin{equation}\label{eq:gen_diff_+}
\mathfrak{F}_+^{\circ} (x\, ;h)\coloneqq \limsup_{\substack{y \to x\\t \downarrow 0}} \frac{\mathfrak{F}(y+th)-\mathfrak{F}(y)}{t}  \,  \cdot
\end{equation}  Likewise, the generalized lower  directional derivative of $\mathfrak{F}$ at $x $ in the direction $h$ is
\begin{equation}\label{eq:gen_diff_-}
\mathfrak{F}_-^{\circ} (x\, ;h)\coloneqq \liminf_{\substack{y \to x\\t \downarrow 0}} \frac{\mathfrak{F}(y+th)-\mathfrak{F}(y)}{t} \,  \cdot
\end{equation}
\end{D}

\begin{rems}
\hfill
\begin{enumerate}
\item For any $h \in X$,  \begin{equation}\label{eq:locLip}
\mathfrak{F}_{\pm}^{\circ} (x\, ;h) \le_{\pm} \pm \mathrm{Lip}_x(\mathfrak{F})\|h\|_X.
\end{equation}
\item A change of variable $z=y+th$ in \eqref{eq:gen_diff_+} and \eqref{eq:gen_diff_-} shows that
\begin{equation}\label{eq:change}
\mathfrak{F}_\pm^{\circ} (x\, ; h) = - \mathfrak{F}_\mp^{\circ} (x\, ;-h).
\end{equation}
\item A change of variable $s=-t$ and \eqref{eq:change} show that replacing $t \downarrow 0$ by $t \uparrow 0$ in \eqref{eq:gen_diff_+} and \eqref{eq:gen_diff_-} makes no change in the definitions of $\mathfrak{F}_+^{\circ} (x\, ;h)$ and $\mathfrak{F}_-^{\circ} (x\, ;h)$.
 \item If $\mathfrak{F}$ admits a right-directional derivative at $x$ in the direction $h$, then
\begin{equation}\label{eq:-'+}
\mathfrak{F}^\circ_{-}(x,h) \le  \mathfrak{F}_r'(x,h) \le  \mathfrak{F}^\circ_{+}(x,h).
\end{equation}
\end{enumerate}
\end{rems}

Applying Lemma \ref{lem:elem} to the generalized directional derivative functions $\mathfrak{F}_{+}^\circ(x,\cdot)$ and $\mathfrak{F}_{-}^\circ(x,\cdot)$, one can define the following sets.

\begin{D}
The Clarke subdifferential of $\mathfrak{F}$ at $x$ is 
\[
\partial_C^- \mathfrak{F} (x) \coloneqq \{ \zeta \in X^* \, : \, \langle \zeta,h \rangle \le \mathfrak{F}_{+}^\circ(x,h) \, \, \text{for any $h \in X$}\}
\]
and the Clarke  superdifferential of $\mathfrak{F}$ at $x$ is
\[
\partial_C^+ \mathfrak{F} (x) \coloneqq \{ \zeta \in X^* \, : \, \langle \zeta,h \rangle \ge \mathfrak{F}_{-}^\circ(x,h) \, \, \text{for any $h \in X$}\}.
\]
\end{D}

\begin{rem} It is crucial to notice that the sets $\partial_C^\pm \mathfrak{F} (x)$ are never empty.  This is because the functions $h \mapsto \mathfrak{F}_{\pm}^\circ(x,h)$ are positively homogeneous, super/sub-additive and bounded on the unit ball: see \cite[Proposition 1.3 (d) in Chapter 2]{Clarke}. Moreover, it follows from \eqref{eq:-'+} that 
\[
\partial^\pm \mathfrak{F} (x) \subset \partial_C^\pm \mathfrak{F} (x).
\]
\end{rem}

The next proposition collects several results about the generalized directional derivatives and the Clarke sub/superdifferentials.  It is worth pointing out that, compared to Proposition \ref{prop:collection},  the Clarke sub/superdifferentials satisfy the opposite inclusion in the sum, product and chain rules.

\begin{prop}\label{prop:collectionClarke}

\hfill
\begin{enumerate}
\item If $\mathfrak{F}$ is $\mathcal{C}^1$-Fréchet on a neighborhood of $x$, then
\begin{equation}\label{eq:direct_Fréchet_rule_Clarke}
\fk_\pm^\circ(x,h) = \fk'_r(x,h) = \langle \mathfrak{F}_x(x), h \rangle
\end{equation}
for any $h \in X$, and
$$\partial_C^\pm \mathfrak{F}(x)  = \partial^\pm \mathfrak{F}(x) = \{\mathfrak{F}_x(x)\}.$$

\item For any $\lambda\ge 0$,
\[
[\lambda \fk]_\pm^\circ(x,h) = \lambda \fk_\pm^\circ(x,h) \quad \text{and} \quad [-\lambda \fk]_\pm^\circ(x,h) = -\lambda \fk_\mp^\circ(x,h)
\]
for any $h \in X$, and
\begin{equation}\label{eq:homo_rule_Clarke}
\partial_C^{\pm} [\lambda \mathfrak{F}](x) = \lambda \partial_C^{\pm} \mathfrak{F}(x) \quad \text{and} \quad \partial_C^{\pm} [-\lambda \mathfrak{F}](x) = -\lambda \partial_C^{\mp} \mathfrak{F}(x).
\end{equation}

\item If $\mathfrak{G} : \Omega \to \setR$ is another locally Lipschitz function, then
\begin{equation}\label{eq:direct_sum_rule_Clarke}
[\mathfrak{F} + \mathfrak{G}]_\pm^\circ(x,h) \le_\pm \mathfrak{F}_\pm^\circ(x,h) +  \mathfrak{G}_\pm^\circ(x,h),
\end{equation}
\begin{equation}\label{eq:direct_product_rule_Clarke}
[\mathfrak{F} \mathfrak{G}]_\pm^\circ(x,h) \le_\pm \mathfrak{F}(x) \mathfrak{G}_\pm^\circ(x,h) +  \mathfrak{G}(x) \mathfrak{F}_\pm^\circ(x,h)
\end{equation}
for any $h \in X$, and
\begin{equation}\label{eq:sum_rule_Clarke}
\partial_C^{\pm} (\mathfrak{F} + \mathfrak{G})(x) \subset  \partial_C^{\pm} \mathfrak{F} (x) + \partial_C^{\pm} \mathfrak{G}(x),
\end{equation}
\begin{equation}\label{eq:product_rule_Clarke}
\partial_C^{\pm} (\mathfrak{F} \mathfrak{G})(x) \subset \mathfrak{F}(x) \partial_C^{\pm} \mathfrak{G} (x) + \mathfrak{G}(x) \partial_C^{\pm} \mathfrak{F}(x).
\end{equation}

\item If $\fk = F \circ (\lambda_1,\ldots, \lambda_N)$ where $F \in \cC^1(\setR^N)$ and $\lambda_1,\ldots, \lambda_N : \Omega \to \setR$, then for any $h \in X$,
\[
\fk^\circ_{\pm}(x,h)  \le \sum_{\substack{k=1\\d_k(x)>0}}^N d_k(x) \lambda^\circ_{\pm}(x,h) +  \sum_{\substack{k=1\\d_k(x)<0}}^N d_k(x) \lambda^\circ_{\mp}(x,h)
\]
where
\[
d_k(x) \df \partial_k F(\lambda_1(x),\ldots, \lambda_N(x)),
\]
and
\begin{equation}\label{eq:chain_rule_Clarke}
\partial_C^\pm \fk(x) \subset \sum_{\substack{k=1\\d_k(x)>0}}^N d_k(x)\partial_C^\pm \lambda_k(x) +  \sum_{\substack{k=1\\d_k(x)<0}}^N d_k(x)\partial_C^\mp \lambda_k(x).
\end{equation}
\end{enumerate}

\end{prop}

\begin{proof}
Like in the proof of Proposition \ref{prop:collection}, we only provide details on how to prove the sum rules \eqref{eq:direct_sum_rule_Clarke} and \eqref{eq:sum_rule_Clarke}.  For $h \in X$, consider $y_n \to x$ and $t_n \downarrow 0$ which realize the limit in the definition of $[\fk+ \mathfrak{G}]_+^\circ(x,h)$.  Then
\begin{align*}
[\fk+ \mathfrak{G}]_+^\circ(x,h) & = \lim\limits_{n \to +\infty} \frac{[\fk+ \mathfrak{G}](y_n+ t_n h) - [\fk+ \mathfrak{G}](y_n)}{t_n} \\
& \le \limsup\limits_{n \to +\infty} \frac{\fk(y_n+ t_n h) - \fk(y_n)}{t_n} + \limsup\limits_{n \to +\infty} \frac{\mathfrak{G}(y_n+ t_n h) - \mathfrak{G}(y_n)}{t_n}\\
& \le \fk_+^\circ(x,h) + \mathfrak{G}_+^\circ(x,h).
\end{align*}
This proves \eqref{eq:direct_sum_rule_Clarke}.  The latter implies that the support functions of both sets involved in \eqref{eq:sum_rule_Clarke} are equal, hence the unique determination granted by Lemma \ref{lem:elem} implies that \eqref{eq:sum_rule_Clarke} holds true.
\end{proof}

\subsubsection{Upper and lower regularity} In general, the right directional derivatives and the generalized ones may  behave differently; this yields different properties between the classical sub/superdifferential and the Clarke ones.  However, the next definition provides a setting where both sets of notions agree.

\begin{D}
We say that $\fk$ is upper regular at $x \in \Omega$ if for any $h \in X$,
\begin{equation}\label{eq:upper_regular}
\fk_r'(x\, ;h) = \fk_+^{\circ} (x\, ;h).
\end{equation} 
Likewise, we say that $\fk$ is lower regular at $x$ if for any $h \in X$,
\begin{equation}\label{eq:lower_regular}
\fk_r'(x\, ;h) = \fk_-^{\circ} (x\, ;h).
\end{equation} 
\end{D}

Upper and lower regularity have the following obvious consequences.

\begin{prop}\label{prop:upper_reg}
Assume that $\fk$ is upper regular at $x \in \Omega$.

\begin{enumerate}
\item For any $h \in X$,
\[
\fk_\ell'(x,h) =  \fk_-^\circ(x,-h)
\]
and
\begin{equation}\label{eq:equal_regular}
\partial_C^- \fk(x) =  \partial^- \fk(x).
\end{equation}
\item If $\mathfrak{G}:\Omega \to \setR$ is another locally Lipschitz function upper regular at $x$, then
\begin{align}\label{eq:sum_rule_equal}
\partial_C^{-} (\mathfrak{F} + \mathfrak{G})(x) =  \partial_C^{-} \mathfrak{F} (x) + \partial_C^{-} \mathfrak{G}(x),
\end{align}
\begin{align}\label{eq:product_rule_equal}
 & \partial_C^{-} (\mathfrak{F} \mathfrak{G})(x)  = \mathfrak{F}(x) \partial_C^{-} \mathfrak{G} (x) + \mathfrak{G}(x) \partial_C^{-} \mathfrak{F}(x).
\end{align}
\item If $\fk = F \circ (\lambda_1,\ldots, \lambda_N)$ where $F \in \cC^1(\setR^N)$ and $\lambda_1,\ldots, \lambda_N : \Omega \to \setR$, then for any $h \in X$,
\begin{align}\label{eq:chain_rule_equal}
\partial_C^- \fk(x) = \sum_{\substack{k=1\\d_k(x)>0}}^N d_k(x)\partial_C^- \lambda_k(x) +  \sum_{\substack{k=1\\d_k(x)<0}}^N d_k(x)\partial_C^+ \lambda_k(x).
\end{align}
\end{enumerate}
The same results hold true with ``upper'' and ``$-$''  replaced by ``lower'' and ``$+$'' respectively.
\end{prop}

\subsubsection{Critical points} With the tools of non-smooth analysis recalled above,  one possible definition of a critical point $x$ for an eigenvalue functional $\mathfrak{F}$ is
\begin{equation}\label{eq:deriv}\tag{C}
\sup_{h \in X} \mathfrak{F}_\ell'(x\, ; h) \mathfrak{F}_r'(x\, ; h) \le 0.
\end{equation}
It corresponds to the seminal definition \eqref{eq:criti_Nadi}. The next proposition provides some relation between this condition and classical sub/superdifferentials.

\begin{prop}\label{prop:C}
For any $x \in \Omega$:
\begin{align*}
0 \in \partial^{-}\mathfrak{F}(x) \qquad \iff & \qquad \inf_{h\in X}\mathfrak{F}_r'(x\, ; h)  \ge 0 \qquad  \Longrightarrow \qquad \eqref{eq:deriv},\\
0 \in \partial^{+}\mathfrak{F}(x) \qquad \iff & \qquad \sup_{h \in X}\mathfrak{F}_r'(x\, ; h)  \le 0 \qquad  \Longrightarrow \qquad \eqref{eq:deriv},\\
0 \in \partial^{-}\mathfrak{F}(x) \cup \partial^{+}\mathfrak{F}(x) \quad & \Longrightarrow \quad \eqref{eq:deriv} \quad  \Longrightarrow \quad 0 \in \overline{\conv} \left\{ \partial^{-}\mathfrak{F}(x) \cup \partial^{+}\mathfrak{F}(x) \right\}.
\end{align*}
\end{prop}

\begin{proof}
The equivalences in the first two lines are obvious. The implication of the first line is proved in the following way: for any $h \in X$,
\[
\mathfrak{F}_\ell'(x\, ; h) \mathfrak{F}_r'(x\, ; h) = - \underbrace{\mathfrak{F}_r'(x\, ; -h)}_{\ge 0} \underbrace{\mathfrak{F}_r'(x\, ; h)}_{\ge 0} \le 0.
\]
The implication in the second line is proved in a similar way. The first implication in the third line is obvious because of the two first lines. Let us prove the contrapositive of the second implication in the third line. Assume that $0$ does not belong to $\overline{\conv} \left\{ \partial^{-}\mathfrak{F}(x) \cup \partial^{+}\mathfrak{F}(x) \right\}$. By the Hahn--Banach theorem, there exists $h \in X$ such that $\langle \xi, h \rangle>0$ for any $\xi \in \overline{\conv} \left\{ \partial^{-}\mathfrak{F}(x) \cup \partial^{+}\mathfrak{F}(x) \right\}$.  In particular, if $\xi \in  \partial^{-}\mathfrak{F}(x)$, then
\[
\mathfrak{F}_r'(x,h) \ge \langle \xi,  h \rangle >0,
\]
while if $\xi \in \partial^{+}\mathfrak{F}(x)$, then
\[
\mathfrak{F}_r'(x,-h) \ge \langle \xi,  -h \rangle = - \langle \xi,  h \rangle  < 0,
\]
hence we get
\[
\mathfrak{F}_\ell'(x,h)  \mathfrak{F}_r'(x,h)  =  \underbrace{-\mathfrak{F}_r'(x,-h)}_{>0} \underbrace{\mathfrak{F}_r'(x,h)}_{>0} > 0
\]
which means that \eqref{eq:deriv} does not hold.
\end{proof}

With the generalized directional derivatives at our disposal,  instead of \eqref{eq:deriv} we can consider the following condition which does not require the existence of right directional derivatives:
\begin{equation}\label{eq:deriv+}\tag{CC}
\sup_{h \in X} \mathfrak{F}_+^\circ(x\, ; h) \mathfrak{F}_-^\circ(x\, ; h) \le 0.
\end{equation}

Then we can prove a statement analogous to Proposition \ref{prop:C}.

\begin{prop}\label{prop:CC}
For any $x \in \Omega$:
\begin{align*}
0 \in \partial^{-}_C\mathfrak{F}(x) \qquad \iff & \qquad \inf_{h\in X} \mathfrak{F}_+^\circ(x\, ; h)  \ge 0 \qquad  \Longrightarrow \qquad \eqref{eq:deriv+},\\
0 \in \partial^{+}_C\mathfrak{F}(x) \qquad \iff & \qquad \sup_{h \in X}\mathfrak{F}_-^\circ(x\, ; h)  \le 0 \qquad  \Longrightarrow \qquad \eqref{eq:deriv+},\\
0 \in \partial_C^{-}\mathfrak{F}(x) \cup \partial_C^{+}\mathfrak{F}(x) \quad & \Longrightarrow \quad \eqref{eq:deriv+} \quad  \Longrightarrow \quad 0 \in \overline{\conv} \left\{ \partial_C^{-}\mathfrak{F}(x) \cup \partial_C^{+}\mathfrak{F}(x) \right\}.
\end{align*}
\end{prop}

\begin{proof}
The equivalences in the first two lines are obvious. The implication in the first line follows from noting that for any $h \in X$,
\[
\mathfrak{F}_+^\circ (x\, ; h) \mathfrak{F}_-^\circ(x\, ; h)  = - \underbrace{\mathfrak{F}_+^\circ(x\, ; h)}_{\ge 0} \underbrace{\mathfrak{F}_+^\circ(x\, ; -h)}_{\ge 0} \le 0.
\]
The implication in the second line follows from noting that for any $h \in X$,
\[
\mathfrak{F}_+^\circ (x\, ; h) \mathfrak{F}_-^\circ(x\, ; h)  = - \underbrace{\mathfrak{F}_-^\circ(x\, ; - h)}_{\le 0} \underbrace{\mathfrak{F}_-^\circ(x\, ; h)}_{\le 0} \le 0.
\]
The implication in the third line is obvious. Let us prove the contrapositive of the second implication in the third line. Assume that $0 \notin \overline{\conv} \left\{ \partial_C^{-}\mathfrak{F}(x) \cup \partial_C^{+}\mathfrak{F}(x) \right\}$. Then the Hahn-Banach theorem implies that there exists $h \in X$ such that $\langle \xi, h\rangle > 0$ for any $\xi \in \overline{\conv} \left\{ \partial_C^{-}\mathfrak{F}(x) \cup \partial_C^{+}\mathfrak{F}(x) \right\}$.  In particular, if $\xi \in  \partial_C^{-}\mathfrak{F}(x)$, then
\[
\mathfrak{F}^\circ_+(x \, ; h) \ge \langle \xi, h \rangle > 0,
\]
while if $\xi \in  \partial_C^{+}\mathfrak{F}(x)$, then
\[
\mathfrak{F}^\circ_+(x \, ; - h) \ge \langle \xi,  - h \rangle = - \langle \xi,  h \rangle < 0,
\]
so that
\[
\mathfrak{F}^\circ_+(x \, ; h) \mathfrak{F}^\circ_-(x \, ; h) = - \underbrace{\mathfrak{F}^\circ_+(x \, ; h)}_{>0} \underbrace{\mathfrak{F}^\circ_+(x \, ; - h)}_{<0} >0
\]
which means that \eqref{eq:deriv+} does not hold.
\end{proof}

The relationship between  \eqref{eq:deriv} and \eqref{eq:deriv+} is easy to describe under upper/lower regularity.

\begin{prop}
Assume that  $\mathfrak{F}$ is upper regular or lower regular at $x$. Then
\[\eqref{eq:deriv} \qquad \iff \qquad \eqref{eq:deriv+}.\] As a consequence:
\[
\begin{array}{ccccc}
0 \in \partial^{-}\mathfrak{F}(x) \cup \partial^{+}\mathfrak{F}(x) &  \Longrightarrow & \eqref{eq:deriv} &  \Longrightarrow &   0 \in \overline{\conv} \left\{ \partial^{-}\mathfrak{F}(x) \cup \partial^{+}\mathfrak{F}(x) \right\}\\
\hfill\\
 \Downarrow & &   \Updownarrow & & \Downarrow \\
 \hfill \\
 0 \in \partial_C^{-}\mathfrak{F}(x) \cup \partial_C^{+}\mathfrak{F}(x) & \Longrightarrow &  \eqref{eq:deriv+} &  \Longrightarrow & 0 \in \overline{\conv} \left\{ \partial_C^{-}\mathfrak{F}(x) \cup \partial_C^{+}\mathfrak{F}(x) \right\}.
\end{array}
\]
In particular:
\[
0 \in \partial_C^{-}\mathfrak{F}(x) \cup \partial_C^{+}\mathfrak{F}(x)  \qquad \Rightarrow \qquad \partial^{-}\mathfrak{F}(x) \cup \partial^{+}\mathfrak{F}(x) \neq \emptyset.
\]
\end{prop}

\begin{proof}
Assume that $\mathfrak{F}$ is upper regular at $x$. For any $h \in X$,
\begin{align*}
\mathfrak{F}_+^\circ(x\, ; h) \mathfrak{F}_-^\circ(x\, ; h) & = - \mathfrak{F}_+^\circ(x\, ; h) \mathfrak{F}_+^\circ(x\, ; -h)\\
& = - \mathfrak{F}'_r(x\, ; h) \mathfrak{F}_r'(x\, ; -h)\\
& =  \mathfrak{F}'_r(x\, ; h) \mathfrak{F}_\ell'(x\, ; h).
\end{align*}
Then \eqref{eq:deriv} and \eqref{eq:deriv+} are obviously equivalent. In the same way, if $\mathfrak{F}$ is lower regular at $x$, then for any $h \in X$,
\begin{align*}
\mathfrak{F}_+^\circ(x\, ; h) \mathfrak{F}_-^\circ(x\, ; h) & = - \mathfrak{F}_-^\circ(x\, ; -h) \mathfrak{F}_-^\circ(x\, ; h)\\
& = - \mathfrak{F}'_r(x\, ; -h) \mathfrak{F}_r'(x\, ; h)\\
& =  \mathfrak{F}'_\ell(x\, ; h) \mathfrak{F}_r'(x\, ; h),
\end{align*}
so that \eqref{eq:deriv} and \eqref{eq:deriv+} are obviously equivalent. The table of equivalences and implications follows from the previous propositions and the inclusions $\partial^{\pm}\mathfrak{F}(x) \subset \partial_C^{\pm}\mathfrak{F}(x)$.  This table notably implies that if $0 \in \partial_C^{-}\mathfrak{F}(x) \cup \partial_C^{+}\mathfrak{F}(x)$, then $0 \in \overline{\conv} \left\{ \partial^{-}\mathfrak{F}(x) \cup \partial^{+}\mathfrak{F}(x) \right\}$ so that $\partial^{-}\mathfrak{F}(x) \cup \partial^{+}\mathfrak{F}(x)$ cannot be empty.
\end{proof}

In the rest of the paper, we adopt the following definition.

\begin{D}\label{def:crit}
We say that $x$ is critical for $\mathfrak{F}$ if $0 \in \partial_C^- \mathfrak{F}(x) \cup  \partial_C^+ \mathfrak{F}(x).$
\end{D}

\subsubsection{Approximate sub/superdifferential} 
Later in this paper, we shall be interested in fine differentiability properties of the locally Lipschitz function $\fk$. In particular, we will need a notion to relate classical sub/superdifferentials with Clarke's ones. This is provided by the approximate sub/superdifferentials introduced by Ioffe in the more general context of a locally convex space \cite{Ioffe2}, see also \cite{Ioffe1, Ioffe3}. 

Let us begin with a preliminary definition.

\begin{D}
Let $\{A(y)\}_{y \in \Omega}$ be a family of sets belonging to $X^*$.  For any $x \in \Omega$, we define
$$ \limsup_{y \to x} A(y) \df \bigcap_{\delta>0} \overline{ \bigcup_{\Vert x-y \Vert < \delta} A(y)} $$
Note that $\zeta \in \limsup_{y \to x} A(y)$ if and only if there exist $y_n \to x$ and $\zeta_{n} \in A(y_{n})$ for any $n$ such that $\zeta_n \to \zeta$ in $X^*$ as $n\to +\infty$. 
\end{D}

We now introduce a useful notation. Let $V$ be a linear subspace of $X$. For any $y \in \Omega$, letting $r = \mathrm{dist}(y,X \backslash \Omega)$, we define
\[
\fk_{\vert y+V} \df \fk : (y+V) \cap B(y,r) \to \setR.
\]
If $\mathfrak{F}$ has left and right derivatives at $y$ in any direction $h\in X$, then the sub/superdifferentials of $\fk_{\vert y+V}$ at $y$ are well-defined subsets of $V^*$ and are denoted $\partial^{\pm}\fk_{\vert y+V}(y)$. To be consistent with Ioffe's definitions, we denote
$$\widetilde{\partial^{\pm}\fk_{\vert y+V}}(y) = \{ \zeta \in X^* : \zeta_{\vert V} \in \partial^{\pm}\fk_{\vert y+V}(y) \}.$$


We are now in a position to define the approximate sub/superdifferentials in the sense of Ioffe.

\begin{D}
Assume that $\mathfrak{F}$ has left and right derivatives in a neighborhood of $x$ in any direction $h\in X$. The approximate sub/superdifferential of $\mathfrak{F}$ at $x$ is defined as
$$ \partial_A^{\pm}\mathfrak{F}(x) = \bigcap_{V \in \mathcal{G}(X)} \limsup_{y \to x}  \widetilde{\partial^{\pm}\fk_{\vert y+V}}(y).$$
where $\mathcal{G}(X) = \bigcup_{k\in \mathbb{N}} \mathcal{G}_k(X)$.
\end{D}

\begin{rem}
In his papers, Ioffe worked with slightly different definitions.  Firstly, he took as a support function for $\partial^{-}\fk$ the Dini derivative, which is always well-defined, while we use the left and right directional derivatives : this is because the latter ones exist in all of our applications. Secondly, Ioffe defined $\fk_{S}(x)=x$ if $x \in S$, $\fk_{S}(x)=+\infty$ otherwise, for any set $S \subset X$.  In this way, the sub/superdifferentials of $\fk_{S}$ are subsets of $X^*$, so there is no need to extend elements of $V^*$ to $X^*$ as we do. We found our variant more convenient in our proofs.
\end{rem}

The next Proposition is inspired by \cite[Proposition 3.3]{Ioffe2}. We provide a self-contained proof for the reader's convenience.
\begin{prop}\label{prop:IOFFE}  Assume that $\mathfrak{F}$ is a locally Lipschitz function with left and right derivatives in a neighborhood of $x \in \Omega$ in any direction $h\in X$. Then
$$ \partial_C^{\pm} \mathfrak{F}(x) = \overline{ \text{co}\left(  \partial_A^{\pm}\mathfrak{F}(x)\right)}.$$
\end{prop}
\begin{proof} We shall give a proof based on \cite[Proposition 2.4, Proposition 3.3]{Ioffe2} that we adapt to our context : $X$ is a Banach space and $\mathfrak{F}$ is a Lipschitz function with left and right derivatives in a neighborhood of $x \in \Omega$ in any direction $h\in X$. We only prove the proposition for subdifferentials.

\noindent\textbf{Step 1:} We have that $ \overline{ \text{co}\left(  \partial_A^-\mathfrak{F}(x)\right)} \subset \partial_C^- \mathfrak{F}(x) $.

\noindent\textbf{Proof of Step 1:} It suffices to prove that $\partial_A^-\mathfrak{F}(x) \subset \partial_C^- \mathfrak{F}(x)$. Let $\xi \in \partial_A^-\mathfrak{F}(x)$.
Let $h \in X$ and $V \in \mathcal{G}(X)$ containing $h$. By definition of $\partial_A^-  \mathfrak{F}(x)$, there exist $x_n \to x$ and $\xi_n \to \xi$ such that $\xi_n \in \widetilde{\partial^{-}\fk_{\vert x_n+V}}(x_n)$. In particular, we have that
$$ \langle \xi_n , h \rangle \leq \left(\fk_{\vert x_n+V} \right)'_r(x_n;h). $$
Then
\begin{align*}
\langle \xi,h \rangle & = \limsup_{n\to +\infty} \langle \xi_n,h \rangle  \\
& \le \limsup_{n\to +\infty} \left(\fk_{\vert x_n+V} \right)'_r(x_n;h)  \\
& = \limsup_{n\to  +\infty}  \fk'_r(x_n;h)  \\
& =  \limsup_{n\to +\infty}  \lim_{\substack{t \downarrow 0}} \frac{\mathfrak{F}(x_n+th)-\mathfrak{F}(x_n)}{t}\\
& \le  \limsup_{\substack{u \to x\\ t \downarrow 0}} \frac{\mathfrak{F}(u+th)-\mathfrak{F}(u)}{t} =  \mathfrak{F}^\circ_+(x,h).
\end{align*}
Then $\xi \in \partial_C^{-}\mathfrak{F}(x)$
and we proved Step 1.

\medskip

\noindent\textbf{Step 2:} We have that $ \partial_C^- \mathfrak{F}(x)  \subset \overline{ \text{co}\left(  \partial_A^-\mathfrak{F}(x)\right)} $.

\noindent\textbf{Proof of Step 2:} We let $h \in X$. It suffices to prove that
\begin{equation} \label{eq:goalioffe} \mathfrak{F}^\circ_+(x,h) \leq \sup\{ \langle \xi,h \rangle ; \xi \in \partial_A^-\fk(x) \}.
\end{equation}
Indeed, the right-hand side is the supporting function of $\overline{ \text{co}\left(  \partial_A^{\pm}\mathfrak{F}(x)\right)}$. Let $V \in \mathcal{G}(X)$. We denote 
\begin{equation*} Q_V = \left\{ \xi \in V^* : \begin{split}& \text{There is } (x_n) \in \Omega^\mathbb{N} \text{ such that } \\
 & x_n \to x \text{ as } n\to +\infty, \\ 
& \mathfrak{F}_{\vert x_n + V} \text{ is differentiable at } x_n, \\ & D \mathfrak{F}_{\vert x_n + V}(x_n) \to \xi \text{ as } n\to +\infty. \end{split} \right\} \end{equation*}
and
$$ \widetilde{Q_V} = \{ \xi \in X^* ; \xi_{\vert V} \in Q_V \}. $$
By definition of the $\limsup$ of sequences of sets, we have that
$$ \widetilde{Q_V} \subset \limsup_{y \to x}  \widetilde{\partial^{\pm}\fk_{\vert y+V}}(y) =: S_V $$
so that
$$ \bigcap_{V \in \mathcal{G}(X)} \widetilde{Q_V} \subset \bigcap_{V \in \mathcal{G}(X)} S_V = \partial_A^- \fk(x) \subset  \partial_C^- \fk(x). $$
In Step 3, we will prove that
\begin{equation} \label{eq:goalstep2ioffe} \forall V \in \mathcal{G}(X) \text{ s.t } h \in V, \{ \xi \in \widetilde{Q_V} \cap \partial_C^- \fk(x) : \langle \xi,h \rangle = \mathfrak{F}^\circ_+(x,h)  \} \neq \emptyset. \end{equation}
If \eqref{eq:goalstep2ioffe} holds true, since $\partial_A \mathfrak{F}(x) = \cap_{V \in \mathcal{G}(X)} S_V$ and $\partial_C\mathfrak{F}(x)$ is weak$*$ compact and $S_V$ is weak$*$ closed, we can take an increasing sequence $\{V_\ell\} \subset \mathcal{G}(X)$ such that $h \in V_\ell$ yielding elements $\xi_\ell \in \partial_C\mathfrak{F}(x) \cap \bigcap_{j \le \ell}S_{V_j}$ which shall subconverge to $\xi \in \partial_C\mathfrak{F}(x) \cap \partial_A\mathfrak{F}(x)$ satisfying
\[
\langle \xi , h \rangle  =  \mathfrak{F}_+^\circ(x,h).
\]
This yields \eqref{eq:goalioffe}.

\medskip

\noindent\textbf{Step 3 :} We prove \eqref{eq:goalstep2ioffe}.

\noindent\textbf{Proof of Step 3:} Let $h \in X$ and $V \in \mathcal{G}(X)$ containing $h$. First we prove that 
\begin{equation} \label{eq:goalstep2ioffe2} \{ \xi \in Q_V : \langle \xi,h \rangle = \mathfrak{F}^\circ_+(x,h)  \} \neq \emptyset. \end{equation}
The proof of \eqref{eq:goalstep2ioffe2} is inspired by Clarke \cite[Proposition (1.4)]{Clarke1975}. We set
$$ \alpha = \limsup\{ \langle D\fk_{\vert y+V}(y),h \rangle ; y \to x \text{ and } D\fk_{\vert y+V}(y) \text{ exists} \} $$
and we aim at proving that $\fk^\circ(x;h) \leq \alpha$ (the other inequality is straightforward). Let $\eps>0$. Let $\delta >0$ be such that for any $y \in B_X(x,\delta)$ such that $D\fk_{\vert y+V}(y)$ exists, we have
$$ \langle D\fk_{y+V}(y),h\rangle \leq \alpha +\eps. $$
We fix $y_0 \in B_X(x,\frac{\delta}{2})$. By Rademacher's theorem, $\fk_{\vert y_0 + V}$ is differentiable almost everywhere on $(y_0 + V) \cap B_X(x,\delta)$. We set for $y \in y_0 + V$
the line
$$ L_y = (y + \R h) \cap B_X(x,\delta). $$
By Fubini's theorem, for almost every $y \in y_0 + V$ (with respect to the Lebesgue measure on $y_0+V$), the set of points $z \in L_y$ such that $D\fk_{\vert y_0 +V}$ is not differentiable at $z$ has zero measure in $L_y$. Let $y \in B_X(x,\frac{\delta}{2})\cap (y_0+V)$ be such that this property holds true. Then, we can write
$$ \fk(y+th) - \fk(y) = t \int_0^1 \langle D\fk_{\vert y_0 + V}(y+sh),h \rangle ds \leq t (\alpha + \eps) $$
for any $0 \leq t \leq \frac{\delta}{2 \Vert h \Vert} $. Then we have that for almost every $y \in B_X(x,\frac{\delta}{2})\cap (y_0+V)$,
$$ \frac{\fk(y+th) - \fk(y)}{t} \leq (\alpha + \eps). $$
By continuity of $\fk$, this property holds everywhere in $B_X(x,\frac{\delta}{2})\cap (y_0+V)$. Since $y_0$ is arbitrary, this property holds everywhere in $B_X(x,\frac{\delta}{2})$. Passing to the $\limsup$ as $y \to x$ and $t \downarrow 0$ and then letting $\eps \to 0$, we obtain $\fk_+^\circ(x;h) \leq \alpha$ so that $\fk_+^\circ(x;h) = \alpha $. We can pick a sequence $y_i \to x$ such that $D\fk_{\vert y_i+V}(y_i)$ exists and 
$$ \langle D\fk_{\vert y_i+V}(y_i),h \rangle \to \alpha $$
as $i\to +\infty$. $V^*$ is a finite dimensional set and $\fk$ is Lipschitz : up to taking a subsequence, $D\fk_{\vert y_i+V}(y_i) \to \xi$ and $\xi$ belongs to the set in \eqref{eq:goalstep2ioffe2}: we obtain \eqref{eq:goalstep2ioffe2}.

Now we conclude the proof of \eqref{eq:goalstep2ioffe}. Since $\xi \in Q_V$, we have
$$ \forall v \in V, \langle \xi,v \rangle \leq  \mathfrak{F}^\circ_+(x,v). $$
We also have equality if $v = h$ from \eqref{eq:goalstep2ioffe2}.
The Hahn-Banach dominated extension theorem with the continuous sublinear function $ \mathfrak{F}^\circ_+(x,\cdot) : X \to \setR$ gives $\zeta \in X^*$ coinciding with $\xi$ on $V$ such that,
\[
\forall v \in X, \langle \zeta , v \rangle  \le \mathfrak{F}^\circ_+(x,v).
\]
In particular, $\zeta\in \partial_C^-f(x)$ and by construction, we have that $\zeta \in \widetilde{Q_V}$ and $\langle \zeta, h \rangle = \mathfrak{F}^\circ_+(x,h)$. This completes the proof of \eqref{eq:goalstep2ioffe} and the proof of the proposition.
\end{proof}

\subsubsection{Local minima and maxima}

We conclude this section with the following elementary lemma which implies that any local minimum or local maximum is critical in the sense of Definition \ref{def:crit}.

\begin{lemma}\label{lem:minmax} Consider $x \in \Omega$.
\begin{enumerate}
\item If $\fk$ admits a local minimum at $x$, then $0 \in \partial_C^- \fk(x)$. If $\fk$ admits a local maximum at $x$, then $0 \in \partial_C^+ \fk(x)$. 

\item Assume that $\fk$ admits right directional derivatives at $x$ in any direction $h \in X$. If $\fk$ admits a local minimum at $x$, then $0 \in \partial^- \fk(x)$. If $\fk$ admits a local maximum at $x$, then $0 \in \partial^+ \fk(x)$. 
\end{enumerate}
\end{lemma}

\begin{proof}
Let us prove (1). If $\fk$ admits a local minimum at $x$, then for any $h \in X$, 
\[
\fk_+^\circ(x , h) \ge \limsup_{t \downarrow 0} \frac{\fk(x+th)-\fk(x)}{t} \ge 0,
\]
thus $0 \in  \partial_C^- \fk(x)$.  If $\fk$ admits a local maximum for $x$, then for any $h \in X$,
\[
\fk_-^\circ(x , h) \le \liminf_{t \downarrow 0} \frac{\fk(x+th)-\fk(x)}{t} \le 0,
\]
thus $0 \in  \partial_C^+ \fk(x)$.  The proof of (2) is immediate, therefore we omit it.
\end{proof}

\subsection{Mappings to pseudo-Euclidean quadrics}\label{subsec:mappings}

In this subsection, we provide some background information about mappings of smooth manifolds to pseudo-Euclidean quadrics. We follow the presentation of \cite{PetridesEllipsoids}.

\subsubsection{Pseudo-Euclidean quadrics} For given parameters $\lambda=(\lambda_1, \ldots, \lambda_N) \in \setR^N$, $\eps=(\eps_k) \in \{-1,0,1\}^{N}$ and $c \in \setR$,  we consider the (possibly degenerate) quadric of $\setR^N$ defined by
\begin{align}\label{eq:def_quadric}
\cQ  = \cQ(\lambda,\eps,c) & \df \left\{ \xi = (\xi_1,\ldots,\xi_n) \in \setR^N \, : \, \sum_{k=1}^N \eps_k \lambda_k \xi_k^2 = c \right\}.
\end{align}
Note that $\cQ(\lambda,\eps,c) = \cQ(\lambda,-\eps,-c)$. We endow $\cQ$ with the pseudo-Euclidean metric induced by \[h_\eps \df \sum_{k=1}^N \eps_k \, d \xi_k \otimes d \xi_k\] on the space $\setR^N(\eps)$ which we define as the image of $\setR^N$ under the projection over the subspace corresponding to those coordinates $k$ for which $\eps_k \neq 0$.  Observe that $\cQ$ is a level set of the quadratic form
\[
q: \xi \mapsto \sum_{k=1}^{N} \eps_k\lambda_k \xi_k^2
\]
and that  the differential of $q$ at $\xi \in \cQ$ in the direction $v \in \setR^N$ is
\[
\di_\xi q (v) = 2 \sum_{k=1}^N \eps_k \lambda_k \xi_k v_k = h_{\eps}(\Lambda\xi, v)
\]
where we have set $\Lambda = \mathrm{Diag}(\lambda_1,\ldots,\lambda_N)$. Then the tangent space of $\cQ$ at $\xi$ is
\begin{equation}\label{eq:tangent}
T_\xi \cQ = (\Lambda \xi)^{\perp_{h_\eps}} \subset \setR^N(\eps).
\end{equation}

\subsubsection{$L$-harmonic mappings}\label{subsubsec:L_harmonic_mappings} Let us consider a smooth, compact,  connected Riemannian manifold $(M,g)$ endowed with a $\|\cdot\|_{L^2(g)}$-densely defined elliptic self-adjoint operator $L$ of order $m \ge 2$ acting on $L^2(M)$. We assume that $\cC^\infty(M)$ belongs to the domain of $L$.  
For a mapping
\[
\Phi = (\phi_1,\ldots,\phi_N)  \in \cC^m(M,\setR^N)
\]
we set
\[
L\Phi \df (L\phi_1,\ldots,L\phi_N)
\]
and
\begin{align*}
E_{L,\eps}(\Phi) & \df  \frac{1}{2} \int_M \langle \Phi,  L\Phi \rangle_{h_\eps} \di v_g = \frac{1}{2} \sum_{k=1}^N \eps_k \int_M \phi_k L\phi_k \di v_g.
\end{align*}
A classical computation shows that for any $\Phi, \Psi \in \cC^m(M,\setR^N)$, 
\begin{equation}\label{eq:classical}
\left. \frac{\di}{\di t} \right|_{t=0} E_{L,\eps}(\Phi + t \Psi) = \int_M \langle L\Phi, \Psi \rangle_{h_\eps} \di v_g.
\end{equation}
Assume that $\partial M = \emptyset$.  From \eqref{eq:classical}, we obtain that the Euler--Lagrange equation  characterizing a critical point $\Phi$ of the restriction of $E_{L,\eps}$ to the set
\[
\mathcal{M}_m(\cQ) \df \left\{ \Phi   \in \cC^m(M,\setR^N) \, : \, \Phi(M) \subset \cQ \right\}
\]
 is 
\begin{equation}\label{eq:crit}
L \Phi \perp_{h_\eps} T_\Phi \cQ.
\end{equation}

\begin{D}
We say that $\Phi  \in \cM_m(\cQ)$ is an $L$-harmonic mapping from $(M,g)$ to $(\cQ,h_\eps)$ whenever \eqref{eq:crit} is satisfied.
\end{D}

\begin{rem}\label{rem:obvious}
Condition \eqref{eq:crit} is obviously equivalent to $L \Phi \perp_{h_{-\eps}} T_\Phi \cQ$. 
\end{rem}

\begin{rem}\label{rem:obvious2}
By \eqref{eq:tangent},  the Euler--Lagrange characterization \eqref{eq:crit} rewrites as
\[
L \Phi  \paral \Lambda \Phi
\]
that is to say, $L \Phi  = c \, \Lambda \Phi$ for some $c : M \to \setR$.   In terms of coordinates, this amounts to asking that for any $1 \le k \le N$, 
$$L \phi_k = c \lambda_k \phi_k.$$
As a trivial consequence,  if $\lambda_1,\ldots,\lambda_N$ are the $N$ lowest eigenvalues of $L$, then any mapping into $\cQ$ whose $k$-th coordinate is an eigenfunction associated with $\lambda_k$ is an $L$-harmonic mapping.
\end{rem}

\subsubsection{Minimal immersions}
We recall that a $\mathcal{C}^2$ immersion 
\[
\Phi = (\phi_1,\ldots,\phi_N) : (M,g) \to(\setR^N,h_\eps)
\]
defines a pull-back metric $\Phi^* h_\eps \df \sum_{k=1}^N \eps_k d\phi_k \otimes d \phi_k$ on $M$ and is called:
\begin{itemize}
\item conformal if and only if $f g = \Phi^* h_\eps$ for some $f \in \cC^\infty_{>0}(M)$,
\item homothetic if and only if $\gamma g = \Phi^* h_\eps$ for some $\gamma>0$,
\item isometric if and only if
$g = \Phi^* h_\eps$
\end{itemize}
Moreover, $\Phi$ is minimal if and only if it is both $\Delta_g$-harmonic and conformal.

 \subsubsection{$p$-harmonic mappings} Let $n\ge 2$ be the dimension of $M$ and $p \ge 2$ a given number. We consider the set of mappings
\begin{equation}\label{eq:mappings}
\mathcal{M}(\eps) \df  \left\{ \Phi \in \cC^2(M,\setR^N) \,  : \, \langle d \Phi,  d \Phi \rangle_{g,h_\eps}   > 0  \text{ on $M$} \right\}
\end{equation}
where
\begin{equation}\label{eq:prod}
\langle d \Phi,  d \Phi \rangle_{g,h_\eps}  \coloneqq \sum_{k=1}^N \eps_k |d \phi_k|_g^2.
\end{equation}
For any $\Phi \in \mathcal{M}(\eps)$, define
\[
E_{p,g,\eps}(\Phi) = \int_M |d \Phi|_{g,h_\eps}^p \di v_g
\]
where 
\[
|d \Phi|_{g,h_\eps} \df \left( \sum_{k=1}^N \eps_k |d \phi_k|_g^2 \right)^{1/2}.
\]
A classical computation shows that for any $\Phi\in \mathcal{M}(\eps)$ and $\Psi \in \cC^2(M,\setR^N)$ such that $\Phi + t \Psi \in \mathcal{M}(\eps)$ for any $t$ close enough to $0$,
\begin{equation}\label{eq:comp2}
\left. \frac{\di}{\di t} \right|_{t=0} E_{p,g,\eps}(\Phi + t \Psi) = \frac{p}{2} \int_M  \left\langle  \delta_g \left( |d \Phi|_{g,h_\eps}^{p-2} d \Phi \right) , \Psi\right\rangle_{h_\eps} \di v_g,
\end{equation}
where $\delta_g$ is the dual of $d$. 

Assume that $\partial M = \emptyset$.  From \eqref{eq:comp2}, we obtain that the Euler--Lagrange equation  characterizing a critical point $\Phi$ for the restriction of $E_{p,g,\eps}$ to the set
\[
\mathcal{M}(\eps,\cQ) \df  \left\{ \Phi \in \mathcal{M}(\eps) : \Phi(M) \subset \cQ \right\}
\]
 is 
\begin{equation}\label{eq:EL_n_harm}
 \delta_{g} \left( |d \Phi|_{g,h_\eps}^{p-2} d \Phi \right) \perp_{h_\eps} T_\Phi \cQ.
\end{equation}

\begin{D}
We say that $\Phi  \in \cM(\eps,\cQ)$ is a $p$-harmonic mapping from $(M,g)$ to $(\cQ,h_\eps)$ whenever \eqref{eq:EL_n_harm} is satisfied.
\end{D}

\begin{rems}

\hfill
\begin{enumerate}
\item The condition for being $2$-harmonic coincides with the one from the previous subsection for being $\Delta_g$-harmonic.
\item If $p>2$, then \eqref{eq:EL_n_harm} is not equivalent to $\delta_{g} \left( |d \Phi|_{g,h_{-\eps}}^{p-2} d \Phi \right) \perp_{h_{-\eps}} T_\Phi \cQ$.  Actually, $\mathcal{M}(\eps)$ and $\mathcal{M}(-\eps)$ do not even coincide.
\item By \eqref{eq:tangent}, the Euler--Lagrange equation \eqref{eq:EL_n_harm} rewrites as
\begin{equation}\label{eq:EL_n_harm_modif}
 \delta_{g} \left( |d \Phi|_{g,h_\eps}^{p-2} d \Phi \right)  \paral \Lambda \Phi.
\end{equation}
\item It is easily seen that $E_{n,g,\eps}$ is invariant under a conformal change of the metric $g$, i.e.~under replacing $g$ by $fg$ for some $f \in \cC_{>0}^\infty(M)$. 
\end{enumerate}
\end{rems}

 \subsubsection{Free boundary $p$-harmonic mappings}

Assume now that $\partial M \neq \emptyset$.  We consider the set of mappings
\[
\mathcal{M}(\eps,V,W) \df  \left\{ \Phi \in \mathcal{M}(\eps) : \Phi(M) \subset V \text{ and } \Phi(\partial M) \subset W \right\}
\] 
where $V$ is a smooth submanifold of $\setR^N$ and $W$ is a smooth submanifold of $V$.  From \eqref{eq:comp2}, we obtain that a critical point $\Phi$ of the restriction of $E_{p,g,\eps}$ to $\mathcal{M}(\eps,V,W)$ is characterized by the Euler--Lagrange equation
\begin{equation}\label{eq:EL_n_harm_free}
\begin{cases}
 \delta_{g} \left( |d \Phi|_{g,h_\eps}^{p-2} d \Phi \right) \perp_{h_\eps} T_\Phi V & \text{on $M$}\\
  |d \Phi|_{g,h_\eps}^{p-2}  \partial_g^{\nu}  \Phi  \perp_{h_\eps} T_\Phi W  & \text{on $\partial M$},
 \end{cases}
\end{equation}
where $\partial_g^{\nu}$ is the normal derivative with respect to $g$ at the boundary $\partial M$ and 
\[
\partial_g^{\nu}  \Phi  \df (\partial_g^{\nu}  \phi_1, \ldots,\partial_g^{\nu}  \phi_N).
\]

\begin{D}
We say that $\Phi  \in \cM(\eps,V,W)$ is a free boundary $p$-harmonic mapping from $(M,\partial M, g)$ to $(V,W,h_\eps)$ whenever \eqref{eq:EL_n_harm_free} is satisfied.
\end{D}

\begin{rem}
\hfill
\begin{enumerate}
\item When $p=2$ we obtain the classical notion of free boundary $\Delta_g$-harmonic maps.
\item In this paper, we are interested in the two following cases only.
\begin{itemize}
\item $(V,W)=(\setR^N, \cQ)$, in which case \eqref{eq:EL_n_harm_free} becomes
\begin{equation}\label{eq:EL_free_boundary_modif}
\begin{cases}
 \delta_{g} \left( |d \Phi|_{g,h_\eps}^{p-2} d \Phi \right) = 0 & \text{on $M$},\\
  \partial_g^{\nu} \Phi \paral \Lambda \Phi  & \text{on $\partial M$}.
 \end{cases}
\end{equation}
If $\cQ$ is the boundary of its convex hull $\conv \cQ$, then a simple argument based on the maximum principle shows that a free boundary $p$-harmonic mapping $\Phi$ of $(M,\partial M,g)$ into $(\setR^N, \cQ,h_\eps)$ maps $M$ into  $\conv \cQ$. Moreover, if $\Phi$ is additionally minimal, then it is a proper immersion of $M$ into $\conv \cQ$, i.e.~$\Phi(\mathring{M}) \cap \cQ = \emptyset$.

\item $(V,W) = (\cQ,\partial \mathcal{D}\cQ)$ where $\mathcal{D}\cQ$ is the positive piece of some dyadic subdivision of $\cQ$, i.e.~$\mathcal{D}\cQ \df \left\{x \in \cQ \, : \,   \text{$x_i > 0$ if $i \in I$}\right\}$ for some $I \subset \{1,\ldots,N\}$, and $\partial \mathcal{D}\cQ)\df \left\{x \in \cQ \, : \,   \text{$x_i = 0$ if $i \in I$}\right\}$ is its boundary.
\end{itemize}
\end{enumerate}
\end{rem}

\section{Theoretical part}\label{sec:theory}

\subsection{Abstract Rayleigh quotient}\label{subsec:abstract_Rayleigh} From now on,  we consider a Banach space $X$, an open set $\Omega \subset X$,  a vector space $Y$,  and a map
\[ \cR : \Omega \times \uY\to \setR \cup\{+\infty\}\]
such that the following hold at any $x \in \Omega$.

\subsubsection*{Standing assumptions}

\begin{enumerate}
\item[(A)] There exist a positive definite quadratic form $Q(x,\cdot)$ on $Y$ and another quadratic form $G(x,\cdot)$ with $Q^{1/2}(x,\cdot)$-dense domain $H(x)\subset Y$ such that for any $u \in \underline{Y}$,
\[
\cR(x,u)  = \frac{G(x,u)}{Q(x,u)}\, \cdot
\]
We let $B(x,\cdot,\cdot)$ and $\Gamma(x,\cdot,\cdot)$ be the bilinear forms associated to $Q(x,\cdot)$ and $G(x,\cdot)$, respectively. Note that $B(x,\cdot,\cdot)$ is a scalar product on $Y$.
\end{enumerate}

\begin{enumerate}
\item[(B)] There exists a neighborhood $V \subset \Omega$ of $x$ and a bounded function $\theta : V \to (0,+\infty)$ such that for any $x' \in V$, \[N(x',\cdot) \df [G(x'\, ,\cdot)+\theta(x') Q(x',\cdot)]^{1/2}\] defines a Hilbert norm on $H(x')$.  We let $P(x',\cdot,\cdot)$ be the associated scalar product. We set $\Theta(x') \df \inf \{\theta \ge  0 : G(x',u) + \theta Q(x',u) \ge 0 \,\, \text{for any $u \in \underline{Y}$}\}$.
\end{enumerate}

\begin{enumerate}
\item[(C)] There exist $C \ge 1$ and a neighborhood $V \subset \Omega$ of $x$ such that for any $x' \in V$,
\begin{align}\label{eq:equivN}
C^{-1} N(x'\, ,\cdot) & \le N(x\, ,\cdot) \le C N(x'\, ,\cdot). 
\end{align}
In particular, the spaces $H(x)$ and $H(x')$ are all isomorphic one to another.  Therefore, from now on we use the notation $H$ to denote the space $H(x)$. 
\end{enumerate}

\begin{enumerate}
\item[(D)] Let $\{x_n\} \subset \Omega$  be such that $\|x_n - x\|_X\to 0$. For any $(u_n) \subset H$ such that
\[
\sup_n N(x_n,u_n) < +\infty,
\]
there exists $u \in H$ such that, up to extracting a common subsequence from $\{x_n\}$ and $\{u_n\}$,
\begin{align}\label{eq:compactness}
\begin{cases}
Q(x_n,u_n - u) \to 0\\
\Gamma(x_n,u_{n},h) \to \Gamma(x,u,h) \qquad \forall h \in H.
\end{cases}
\end{align}

\item[(E)] The maps $Q(\cdot,u)$ and $G(\cdot,v)$ are Fréchet differentiable at $x$ for any $u \in \uY$ and $v \in H$.

\item[(F)] For any $(x_n) \subset \Omega$,  $(u_n)\subset Y$, $u\in Y$ and $(v_n) \subset H$, $v_n\in H$,
\begin{equation}\label{eq:contQ}
\|x_n-x\|_X + Q(x_n,u_n-u) \to 0  \quad \Rightarrow \quad \begin{cases}Q(x_n,u_n) \to Q(x,u),\\
 Q_x(x_n,u_n) \to Q_x(x,u),\end{cases}
\end{equation}
\begin{equation}\label{eq:contG}
\|x_n-x\|_X + N(x_n,v_n-v) \to 0  \quad \Rightarrow \quad \begin{cases}G(x_n,v_n) \to G(x,u),\\ G_x(x_n,v_n) \to G_x(x,v). \end{cases}
\end{equation}
\item[(G)] There exist $C_0>0$ and a neighborhood $V \subset \Omega$ of $x$ such that for any $x' \in V$,  $u_1,u_2 \in Y$ and $v_1,v_2 \in H$,
\begin{align}\label{Frechet_der}
\| B_x(x' \, , u_1\, , u_2)\|_{X^{*}} &  \le C_0 N(x',u_1) N(x',u_2),\\
\| \Gamma_x(x' \, , v_1\, , v_2)\|_{X^{*}} &  \le C_0 N(x',v_1) N(x',v_2).\nonumber
\end{align}
\end{enumerate}

\begin{rem}\label{eq:alternative_assumptions}
We have that [ (C) and (D) ] $\Leftarrow$ [ (C') and (D') ] where the latter are defined as follows.
\begin{enumerate}
\item[(C')] There exist $C \ge 1$ and a neighborhood $V \subset \Omega$ of $x$ such that for any $x' \in V$,
\begin{align*}
C^{-1} Q(x'\, ,\cdot) & \le Q(x\, ,\cdot) \le C Q(x'\, ,\cdot),\\
C^{-1} N(x'\, ,\cdot) & \le N(x\, ,\cdot) \le CN(x'\, ,\cdot).
\end{align*}

\item[(D')] The injection $(H,N(x,\cdot))  \hookrightarrow (Y,Q^{1/2}(x,\cdot))$ is compact.
\end{enumerate}

\end{rem}

\subsubsection*{Elementary consequences}

\begin{enumerate}
\item By (A), we directly obtain that:
\begin{itemize}
\item $\cR(x\, ,u) = +\infty$ if and only if $u \in Y \backslash H $,
\item $\cR$ is homogeneous with respect to the second variable, that is to say, $\cR(x \, ,t u) = \cR(x\, ,u)$ for any $t \in \setR \backslash \{0\}$ and $u \in H$.
\end{itemize} 
\item By (B) and standard functional analysis,  there exists a $B(x,\cdot,\cdot)$-self-adjoint operator $L(x)$ on $H$ defined by
\begin{align*}
\mathcal{D}(L(x)) & \coloneqq \{ u \in H : \exists ! \, \gamma \in Y, \,\,\forall \, v \in H, \,\, \Gamma(x,u,v) = B(x,\gamma,v)\},\\
L(x)u & \df \gamma \quad \text{for any $u \in \mathcal{D}(L(x))$.}
\end{align*}
\item By (D), considering the sequence $(x_n)$ constantly equal to $x$, we obtain that the injection
\begin{equation*}
(H(x),N(x,\cdot)) \hookrightarrow (Y,Q^{1/2}(x,\cdot))
\end{equation*}
is compact.  This implies that $L(x)$ has a compact resolvent, hence it admits isolated eigenvalues,  all with finite multiplicity,  forming a sequence \[-\Theta(x) \le  \lambda_1(x) \le  \lambda_2(x) \le \cdots \to +\infty.\]  Moreover, there exists an orthonormal basis $(u_k)_{k \in \setN^*}$ of $Y$ made of corresponding eigenvectors.  Lastly, the min-max principle yields that for any $k \in \setN$,
\begin{equation}\label{eq:def_lambda_k}
\lambda_k (x) =  \min_{E \in \mathcal{G}_{k}(H)} \max_{u \in S_{Q(x,\cdot)}(E)} \mathcal{R}(x,u).
\end{equation}

\item By (E),  we know that $\cR(\cdot,v)$ is Fréchet differentiable at any $v\in H$, and the quotient rule for the Fréchet derivative yields that when $v \neq 0$,
\begin{equation}\label{eq:quotientrule}
\mathcal{R}_x(x,v) = \frac{G_x(x,v)Q(x,v)-G(x,v)Q_x(x,v)}{Q(x,v)} \, \cdot
\end{equation}
\item By polarisation,  (F) implies that for any $(x_n) \subset \Omega$,  $(u_n)\subset Y$, $(v_n) \subset H$, $x \in \Omega$,  $u \in Y$,  $v \in H$ and $h \in X$,
\begin{equation}\label{eq:weakcontQ}
\|x_n-x\|_X + Q(x_n,u_n-u) \to 0  \quad \Rightarrow \quad \begin{cases}B(x_n,u_n,h) \to B(x,u,h),\\
 B_x(x_n,u_n,h) \to B_x(x,u,h),\end{cases}
\end{equation}
\begin{equation}\label{eq:weakcontG}
\|x_n-x\|_X + N(x_n,v_n-v) \to 0  \quad \Rightarrow \quad  \begin{cases}\Gamma(x_n,v_n,h) \to \Gamma(x,v,h),\\ \Gamma_x(x_n,v_n,h) \to \Gamma_x(x,v,h). \end{cases}
\end{equation}

\item Of course, (F), (A) and \eqref{eq:quotientrule} also imply that if $v \neq 0$, then
\begin{equation}\label{eq:contR}
\|x_n-x\|_X + N(x_n,v_n-v) \to 0  \quad \Rightarrow \quad \begin{cases}\cR(x_n,v_n) \to \cR(x,v),\\ \cR_x(x_n,v_n) \to \cR_x(x,v). \end{cases}
\end{equation}

\item By (C), the term $N(x_n,u_n-u)$ in the two previous points may be replaced by $N(x,u_n-u)$.

\end{enumerate}

\subsubsection*{Multiplicity notation} \hfill

For any $x \in \Omega$ and $k\in \setN^*$, we set
\begin{align*}
E_k(x) & \df \{ u \in H \, :\,  L(x)u = \lambda_k (x) u \}\\
& = \{ u \in \left( \oplus_{i=1}^{k-1} E_i(x)\right)^{\perp_{Q(x,\cdot)}} \, :\,  \mathcal{R}(x,u) = \lambda_k (x) \},\\
SE_k(x) & \df S_{Q(x,\cdot)}(E_k(x)).
\end{align*}
Note that if $u \in  SE_{k}(x)$ then
\begin{equation}\label{eq:Glambda}
G(x,u) = \lambda_k (x),
\end{equation}
\begin{equation}\label{eq:Rayleigh_der}
\mathcal{R}_x(x,u) =G_x(x,u)-\lambda_k(x)Q_x(x,u)\, \cdot
\end{equation}
We let:
\begin{itemize}
\item $\{\mu_i(x)\}_{i \in \setN^*}$ be the distinct elements of the collection $\{\lambda_k(x)\}_{k \in \setN^*}$ sorted in increasing order,
\item $\{F_i(x)\}_{i \in \setN^*}$ be the distinct elements of the collection $\{E_k(x)\}_{k \in \setN^*}$ sorted in accordance to the corresponding $\{\mu_i(x)\}$. 
\end{itemize}
 For any $i \in \setN^*$ we define:
\begin{itemize}
\item $\aleph_i(x)$ as the set of indices $k$ such that $\lambda_k(x) = \mu_i(x)$, 
\item $m_i(x)$ as the cardinality of $\aleph_{i}(x)$,
\item $j_i(x)$ as the lowest index in $\aleph_i(x)$,
\item $J_i(x)$ as the highest index in $\aleph_i(x)$.
\end{itemize}
For any $k \in \mathbb{N}^*$ we also define
\begin{itemize}
\item $i_k(x)$ as the unique integer such that $\mu_{i_k(x)}(x)=\lambda_k(x)$,
\item $p_k(x) \df k - j_{i_k(x)}(x) + 1$ as the position of $k$ in $\aleph_{i_k(x)}(x)$,
\item $\cU_k(x)$ as the set of $Q(x,\cdot)$-orthogonal families $(u_1,\ldots,u_k) \subset Y$ such that $u_j \in E_j(x)$ for any $1 \le j \le k$,
\item $\mathbf{U}_k(x)$ as the subset of $\cU_k(x)$ made of $Q(x,\cdot)$-orthonormal families.
\end{itemize}

\subsection{Continuity of the eigenvalues} In this section, we study the continuity properties of the functions $\lambda_k$. We begin with the following.

\begin{prop}\label{prop:uppersemi-continuity}
The function $\lambda_k$ is upper semi-continuous on $\Omega$ for any $k \in \mathbb{N}^*$.
\end{prop}

\begin{proof} 
Let $x_n \to x$ be a convergent sequence in $\Omega$. Let $\delta >0$. Let $E \in \mathcal{G}_k(x)$ be such that
\[
\max_{u \in \underline{E}} \cR(x \, , u) \leq \lambda_k(x) + \delta.
 \]
 Take a subsequence $(y_n)$ of $(x_n)$ such that $\lambda_k(y_n) \to  \limsup_{n\to +\infty} \lambda_k(x_n) $. For any $n$,
\[
\lambda_k(y_n) \leq \max_{u\in S_{Q(x,\cdot)}(E) }  \cR(y_n \, , u) =  \cR(y_n \, , u_n),
\]
where $u_n \in S_{Q(x,\cdot)}(E)$ is a maximizing point.  Then
\[
N(x,u_n) = (\cR(x,u_n) + \theta(x))^{1/2} \le (\lambda_k(x) + \delta + \theta(x))^{1/2}
\]
hence $(u_n)$ is bounded both in $Q(x, \cdot )^{\frac{1}{2}}$ and $N(x,\cdot)$. Since $E$ is finite dimensional, there is a subsequence still denoted $(u_n)$ such that $u_n \to u \in E$ strongly with respect to $Q(x,\cdot)^{\frac{1}{2}}$ and $N(x,\cdot)$. By \eqref{eq:contR}, we obtain that
\[ 
\limsup_{n\to +\infty} \lambda_k(x_n) \leq \lim\limits_{n \to +\infty}  \cR(y_n \, , u_n) =  \cR(x \, , u) \leq \lambda_k(x) + \delta. 
\]
Letting $\delta \to 0$ ends the proof of the proposition.
\end{proof}

\begin{prop} \label{prop:continuity}
The function $\lambda_k$ is locally Lipschitz on $\Omega$ for any $k \in \mathbb{N}^*$. 
\end{prop}

\begin{proof}
Let $x \in \Omega$ and $\eps > 0$ be such that $B_X(x,\eps)\subset V$ where $V$ is a neighborhood of $x$ in $\Omega$ where (C) and (D) hold. Consider $y_1,y_2 \in B_X(x,\eps)$ and $0<\delta<1$. 
By Proposition \ref{prop:uppersemi-continuity}, we can assume that $\eps $ is small enough to ensure that
$$ \lambda_k(y_i) \leq \lambda_k(x) +1$$
for any $i \in \{1,2\}$. Let $E \in \mathcal{G}_k(H)$ be such that
\[
\max_{u \in \underline{E}} \cR(y_1 \, , u) \leq \lambda_k(y_1) + \delta.
 \]
 We test $E$ in the variational characterization of $\lambda_k(y_2)$ to get
 \[
\lambda_k(y_2) \leq \max_{u \in \underline{E}} \cR(y_2 \, , u).
 \]
 We wish to compare $\cR(y_1 \, , u)$ with $\cR(y_2 \, , u)$ for a  given $u \in \underline{E}$.  Observe that
  \[
G(y_2 \, , u) = G(y_1 \, , u) + \int_{0}^1 \langle G_x(y_1 + s(y_2-y_1),u), y_2-y_1 \rangle \di s
 \]
 so that by (C) and (G), we get
   \[
\vert G(y_2 \, , u) - G(y_1 \, , u) \vert \leq C_0 N(y_1,u )^2 \Vert y_2 - y_1 \Vert_X.
 \]
 Since
    \[
 \frac{N(y_1,u )^2}{Q(y_1,u)} = \theta(x)+ \cR(y_1,u) \leq \theta(x) + \lambda_k(y_1)+\delta \leq \theta(x) + \lambda_k(x)+2
 \]
we obtain that there exists $K>0$ depending on $x$ only such that
    \[
\vert G(y_2 \, , u) - G(y_1 \, , u) \vert \leq K Q(y_1,u) \Vert y_2 - y_1 \Vert_X.
 \]
Likewise, up to changing the value of $K$, we have
   \[
\vert Q(y_2 \, , u) - Q(y_1 \, , u) \vert \leq K Q(y_1,u) \Vert y_2 - y_1 \Vert_X.
 \]
Thus
\begin{align*}
\cR(y_2 \, , u)  = \frac{G(y_2 \, , u)}{Q(y_2 \, , u)} & \le \frac{G(y_1 \, , u)+ KQ(y_1 \, , u)\|y_2 - y_1\|_X}{Q(y_2 \, , u)} \\
& \leq \frac{ \cR(y_1 \, , u) }{ 1 \pm K \Vert y_2-y_1 \Vert_X} + \frac{K \Vert y_2 -y_1 \Vert_X}{1-K \Vert y_2 -y_1 \Vert_X}
\end{align*} 
 where $\pm = -$ if $\cR(y_1 \, , u) \geq 0$ and $\pm = +$ if $\cR(y_1 \, , u) < 0$.  Assuming that $\eps$ is small enough to guarantee that
 \[
 \frac{1}{1\pm K \Vert y_2-y_1 \Vert_X} \le_{\mp} 1 \mp 2 K \Vert y_2-y_1 \Vert_X
 \]
 we eventually obtain that there exists $K'>0$ depending on $x$ only such that
  \[ \lambda_k(y_2) \leq \max_{u \in \underline{E}} \cR(y_2 \, , u) \leq \lambda_k(y_1) + \delta + K' \Vert y_2 -y_1 \Vert_X.
 \]
 Letting $\delta \to 0$ and exchanging the roles of $y_1$ and $y_2$ completes the proof.
\end{proof}

\begin{rem}
Note that the upper semi-continuity and the Lipschitz regularity obtained in the two previous propositions does not require Assumption (D). 
\end{rem}

We now obtain the following immediate but important corollary where Assumption (D) is crucial.

\begin{cor}\label{prop:closed_graph}
Let $x_\ell \to x$ be a convergent sequence in $\Omega$.  For $k \in \setN^*$, let $\{u_\ell\} \subset Y$ be such that $u_\ell \in SE_{k}(x_\ell)$ for any $\ell$. Then there exists $u \in E_{k}(x)$ such that $\{u_\ell\}$ subconverges  $Q^{1/2}(x,\cdot)$-strongly and $N(x,\cdot)$-strongly to $u$.
\end{cor}

\subsection{Differentiability of the eigenvalues} 

The next theorem is the key technical result of this paper.

\begin{theorem}\label{th:main} Consider $k \in \setN^*$,  $x \in \Omega$ and $h \in X$. Set $i=i_k(x)$.

\begin{enumerate}
\item For any $\{t_n\} \subset (0,+\infty)$ and $\{z_n\} \subset X$ such that $t_n \downarrow 0$ and $z_n = o(t_n)$ as $n\to +\infty$,  there exists a $Q$-orthonormal, $G$-orthogonal,  and $\mathcal{B}$-orthogonal (see Remark \ref{rk:derivative=eigenvalue} below) family $\{u_j\}_{1 \le j \le m} \subset E_{k}(x)$ such that for any $j \in \{1,\cdots,m\}$, up to extracting subsequences from $\{t_n\}$ and $\{z_n\}$, one has
\begin{equation}\label{eq:step1}
 \lim_{n\to +\infty}  \frac{\lambda_{j_i(x)+j-1}(x +t_n h + z_n)-\lambda_{j_i(x)+j-1}(x)}{t_n}  = \langle \cR_x(x,u_j),h \rangle.
 \end{equation}

\item If $y: (-\eps,\eps)\to X$ is differentiable, satisfies $y(0) = x$ and $y'(0) = h$, then
\begin{equation} \label{eqminmaxdiff} \lim_{t \downarrow 0} \frac{\lambda_k(y(t)) - \lambda_k(x)}{t}  = \min_{F \in \Gr_{p_k(x)}(E_k(x))} \max_{u \in \uF} \,  \langle\cR_x(x\, ;u),h\rangle.
\end{equation}
In particular, the function $\lambda_k$ admits a right derivative at $x$ in the direction $h$, and
\begin{align}\label{min_max_diff}
[\lambda_k]_r'(x \, ; h) & =   \min_{F \in \Gr_{p_k(x)}(E_k(x))} \max_{u \in \uF} \,  \langle\cR_x(x\, ;u),h\rangle.
\end{align}
\end{enumerate}
\end{theorem}

Before proving this theorem,  we point out a preliminary fact.

\begin{rem}\label{rk:derivative=eigenvalue}
For any $x \in \Omega$ and $h \in X$,  let us consider the bilinear form $\mathcal{B}(\cdot,\cdot)$ defined on $E=E_{k}(x)$ by setting 
\[
\mathcal{B}(u,v) \df   \langle \Gamma_x(x,u,v) - \lambda_k(x) B_x(x,u,v), h\rangle
\]
for any $u, v \in E$. By the Riesz representation theorem applied to the Hilbert space $(E, B(x,\cdot,\cdot))$, there exists a symmetric endomorphism $\mathcal{L}$  of $E$ such that $\mathcal{B}(u,v) = B(x,\mathcal{L}(u),v)$ for any $v \in E$.  Let
\[
\nu_1 \le \ldots \le \nu_m
\]
be the eigenvalues of $\mathcal{L}$. Since $B(x,\mathcal{L}(u),u)  = \langle \mathcal{R}_x(x,u),h \rangle$ for any $u \in E$,  the min-max principle applied to the endomorphism $\mathcal{L}$ on the Hilbert space $(E,B(x,\cdot,\cdot))$ yields that the right-hand side of \eqref{min_max_diff} coincides with $\nu_{p_k(x)}$, so that \eqref{eqminmaxdiff} and \eqref{min_max_diff} rewrite as
\begin{equation}\label{eq:derivative=eigenvalue}
\lim\limits_{t \downarrow 0}\frac{\lambda_k(y(t)) - \lambda_k(x)}{t} = [\lambda_k]_r'(x \, ;h)  = \nu_{p_k}(x).
\end{equation}
\end{rem}

\begin{rem}
Notice that if $t \mapsto y(t)$ is analytic, then \textit{(2)} follows from Rellich-Kato perturbation theory: see \cite[Theorem 2.6 in Chapter 8]{Kato}.  Here we do not assume analyticity.
\end{rem}

We are now in a position to prove Theorem \ref{th:main}.

\begin{proof}[Proof of Theorem \ref{th:main}]

Throughout the proof, we keep $x \in \Omega$ fixed and we write $m$, $Q$ and $G$  instead of $m_k(x)$, $Q(x,\cdot)$ and $G(x,\cdot)$, respectively.  For any $j \in \{1,\cdots,m\}$ and $y \in \Omega$,  we define \[\tilde{\lambda}_j(y) \df \lambda_{j_{i_k(x)}+j-1}(y).\] 

\textbf{Step 1.} We prove \eqref{eq:step1}.  For any $n$ we set 
\begin{align*}
w_n & \df x + t_n h +z_n,\\
w_n^s & \df x + s(t_n h +z_n) \qquad \forall s \in [0,1].
\end{align*}
With no loss of generality,  we assume that all the $w_n$, $w_n^s$ belong to the neighborhoods of $x$ given by Assumptions (B), (C) and (G).
For any $n$,  let $\{u^n_j\}_{1 \leq j \leq m}$ be a $Q(w_n,\cdot)$-orthonormal family of eigenfunctions associated with $\{ \tilde{\lambda}_j(w_n)\}_{1 \leq j \leq m}$. 
Consider $j \in \{1,\ldots,m\}$.  Then each $u^n_j$  satisfies
\begin{equation}\label{eq:good}
 \Gamma(w_n,\cdot,u_j^n) =  \tilde{\lambda}_j(w_n) B(w_n,\cdot,u_j^n). 
\end{equation}
Let $\pi$ be the $Q(x ,\cdot)$-orthogonal projection onto $E_k(x)$.  Set \[R_j^n \coloneqq u_j^n - \pi (u_j^n ) \in E_k(x)^{\perp_{Q}}.\] Then
\begin{align}\label{eq:à_diviser}
 & \phantom{=}  \Gamma(x,\cdot ,R_j^n)  - \tilde{\lambda}_j(x) B(x,\cdot ,R_j^n) \nonumber \\
=  & \phantom{=}  \Gamma(x,\cdot ,u_j^n)   - \tilde{\lambda}_j(x) B(x,\cdot,u_j^n)  \nonumber \\
= & \phantom{=} \Gamma(x,\cdot ,u_j^n) - \Gamma(w_n,\cdot ,u_j^n) +  \tilde{\lambda}_j(w_n) (B(w_n,\cdot,u_j^n) - B(x,\cdot,u_j^n)) \nonumber \nonumber  \\
& \qquad \qquad \qquad \qquad +  (\tilde{\lambda}_j(w_n)-\tilde{\lambda}_j(x)) B(x,\cdot,u_j^n)  \nonumber \\
= &  - t_n \int_{0}^1 \left\langle \Gamma_x(w_n^s ,\cdot,u_j^n) ,h + \frac{z_n}{t_n} \right\rangle \di s + \tilde{\lambda}_j(w_n)  t_n  \int_{0}^1 \left\langle B_x(w_n^s,\cdot,u_j^n),h + \frac{z_n}{t_n} \right\rangle \di s  \nonumber \\
&  \qquad \qquad +  (\tilde{\lambda}_j(w_n)-\tilde{\lambda}_j(x)) B(x,\cdot,u_j^n) 
 \end{align}
where we have used \eqref{eq:good} to obtain the second equality and \eqref{eq:DL1} to get the last one. Set
\[
\tau_j^n \df  t_n + \sqrt{Q(x,R_j^n)}
\]
and
\[
D_j^n  \df \frac{ \tilde{\lambda}_j(w_n) - \tilde{\lambda}_j(x)}{\tau_{j}^n} \,, \qquad \qquad  \overline{t}_j^n \df \frac{t_n}{\tau_j^n} \, ,  \qquad \qquad \overline{R}_j^n \df \frac{R_j^n}{\tau_j^n} \, \cdot
\]
Divide \eqref{eq:à_diviser} by $\tau_j^n$ to get that for any $v \in H$,
\begin{align}\label{eq:topass}
 &   \phantom{=}  \Gamma(x,v ,\overline{R}_j^n)  - \tilde{\lambda}_j(x) B(x,v ,\overline{R}_j^n) =   - \overline{t}_j^n  \int_{0}^1 \left\langle \Gamma_x(w_n^s,v,u_j^n) ,   h +  \frac{z_n}{t_n} \right\rangle \di s  \\
 &  \qquad \qquad \qquad \qquad + \overline{t}_j^n \tilde{\lambda}_j(w_n) \int_{0}^1 \left\langle B_x(w_n^s,v,u_j^n), h + \frac{z_n}{t_n} \right\rangle \di s + B(x,v ,u_j^n) D_j^n. \nonumber
\end{align}
Let us justify that, up to extracting a subsequence, we can take the  limit as $n\to+\infty$ in the previous equality.
\begin{enumerate}
\item[(i)] Since $\{D_j^n\}_n $ is bounded by Proposition \ref{prop:continuity} and $\{\overline{t}_j^n\}_n \subset [0,1]$,  there exist $C>0$, $D_j \in [-C,C]$ and $\overline{t}_j \in [0,1]$ such that, up to extracting subsequences, 
$$ \overline{t}_j = \lim_{n\to +\infty} \overline{t}_j^n
\quad \text{and} \quad  D_j = \lim_{n\to +\infty} D_j^n.$$

\item[(ii)] For any $n$ and $j$,  the element $u_j^n$ is such that $Q(w_n,u_j^n)=1$, and \eqref{eq:good} implies that $G(w_n,u_j^n)=\tilde{\lambda}_j(w_n)$. Then the continuity of $\tilde{\lambda}_j$ and the local boundedness of $\theta$ imply that there exists $C=C(j,x)>0$ such that
\begin{equation}\label{eq:oo}
\sup_n N(w_n,u_j^n) \le C.
\end{equation}
By Assumption (D),  there exists $u_j \in H$ such that, up to extracting subsequences,
\[
\begin{cases}
Q(w_n,u_j^n-u_j) \to 0,\\
\Gamma(w_n,u_j^n,h) \to \Gamma(x,u_j,h).
\end{cases}
\]
In particular, thanks to \eqref{eq:weakcontQ}, we obtain that
\begin{equation}
B(w_n,u_j^n,v) \to B(x,u_j,v).
\end{equation}
Moreover,  by polarisation,  we easily get that for any $\ell,j \in \{1,\ldots,m\}$,
\[
\begin{cases}
\delta_{\ell,j} = B(w_n,u_\ell^n,u_j^n) \to B(x,u_\ell,u_j),\\
c_{\ell,j}^n\delta_{\ell,j} = \Gamma(w_n,u_\ell^n,u_j^n) \to \Gamma(x,u_\ell,u_j),
\end{cases}
\]
where $c_{\ell,j}^n$ is some positive  constant.  Then $(u_j)$ is $Q$-orthonormal and $G$-orthogonal. 

\item[(iii)] On the Hilbert space $(H,N(x,\cdot))$, consider the linear form $L_j^n$ which is dual to $\overline{R}_j^n$, that is to say,  $L_j^n$ is such that for any $v \in H$,
\[
L_j^n(v) = P(x,v,\overline{R}_j^n) 
\]
where we recall that $P(x,\cdot,\cdot)$ is the scalar product associated with $N(x,\cdot)$.
Use \eqref{eq:topass} to obtain the second inequality below :
\[
\left\vert L_j^n(v) \right\vert \leq |\Gamma(x,v,\overline{R}_j^n)| + \theta(x)  |B(x,v,\overline{R}_j^n)|
 \le \mathrm{I} + \mathrm{II} + \mathrm{III} + \mathrm{IV}
\]
where
\begin{align*}
\mathrm{I} & \df \left( \tilde{\lambda}_j(y_n) +  \theta(y_n) \right) |B(y_n,v,\overline{R}_j^n)|  \\
\mathrm{II} & \df  |\overline{t}_j^n|  \int_{0}^1 \left| \left\langle  \Gamma_x(w_n^s,v,u_j^n) ,   h +  \frac{z_n}{t_n} \right\rangle  \right| \di s \\
\mathrm{III} & \df | \tilde{\lambda}_j(w_n)| |\overline{t}_j^n|  \int_{0}^1 \left| \left\langle  B_x(w_n^s,v,u_j^n) ,   h +  \frac{z_n}{t_n} \right\rangle  \right| \di s \\
\mathrm{IV} & \df |B(x,v,u_j^n)| |D_j^n|.
\end{align*}
Note that $\overline{R}_j^n$ is defined in such a way that $Q(x, \overline{R}_j^n)\le 1$. The continuity of $\tilde{\lambda}_j$, the local boundedness of $\theta$ and the Cauchy-Schwarz inequality easily imply that there exists $C=C(j,x)>0$ such that
\begin{align*}
\mathrm{I} & \le CQ^{1/2}(x,v) Q^{1/2}(x,\overline{R}_j^n) \le C N(x,v).
\end{align*}
The Cauchy-Schwarz inequality also implies that
\begin{align*}
\mathrm{IV} & \le N(x,v) N(x,u_j^n)  \le C N(x,v)
\end{align*}
where $C=C(j,x)>0$ is given by \eqref{eq:oo}. Moreover,
\begin{align*}
\mathrm{II} & \le \int_{0}^1 \left\|   \Gamma_x(w_n^s,v,u_j^n)  \right\|_{X^*}   \di s \left\| h +  \frac{z_n}{t_n} \right\|_X\\
& \le C_0 \int_0^1 N(w_n^s,v) N(w_n^s,u_j^n) \di s \left( \| h \|_X +o(1)\right)\\
& \le C_0 C  \left( \| h \|_X +o(1)\right) N(x,v)
\end{align*}
where we have used Assumption (E) to get the second inequality, and both Assumption (D) and \eqref{eq:oo} to get the last one. Lastly,  we may use the continuity of $\tilde{\lambda}_j$, Assumption (E), Assumption (D) and \eqref{eq:oo} to obtain that
\[
\mathrm{III} \le C(j,x)C_0\left( \| h \|_X +o(1)\right) N(x,v).
\]
In the end, we have
\[
|L_j^n(v)| \le \overline{C} N(x,v)
\]
for some $\overline{C}=\overline{C}(j,x,\|h\|_X)>0$ independent of $n$.  Therefore, by duality, we obtain that
\[
\sup_n N(x,\overline{R}_j^n) \le \overline{C}.
\]
Then Assumption (F) implies that, up to extracting subsequences,
\begin{equation}\label{eq:ortho}
\begin{cases}
B(x,v,\overline{R}_j^n) \to B(x,v,\overline{R}_j),\\
\Gamma(x,v,\overline{R}_j^n) \to \Gamma(x,v,\overline{R}_j).
\end{cases}
\end{equation}

\item[(iv)] From \eqref{eq:reste_integral}, we know that
\begin{align*}
 \int_{0}^1 \left\langle \Gamma_x(w_n^s,v,u_j^n) ,h + \frac{z_n}{t_n} \right\rangle \di s & \to   \langle \Gamma_x(x,v,u_j) ,h \rangle, \\
 \int_{0}^1 \left\langle B_x(w_n^s,v,u_j^n),h + \frac{z_n}{t_n} \right\rangle \di s & \to \langle B_x(x,v,u_j),h \rangle. 
\end{align*}
\end{enumerate} 
Thus we can let $n$ tend to $+\infty$ in \eqref{eq:topass}. We get
\begin{align}\label{eq:chooseBeta}
\Gamma(x,v,\overline{R}_j)  - \lambda_k(x) B(x,v,\overline{R}_j) =&  - \overline{t}_j \langle \Gamma_x(x,v,u_j) ,h \rangle + B(x,v,u_j) D_j \nonumber \\
& + \lambda_k(x) \overline{t}_j \langle B_x(x,v,u_j),h \rangle.
\end{align}
Choosing $v = u_j$, we obtain
\begin{equation}\label{eq:almost_done}
0 = - \overline{t}_j \langle G_x(x,u_j) ,h \rangle + Q(x,u_j) D_j  + \lambda_k(x) \overline{t}_j \langle Q_x(x,u_j),h \rangle. 
\end{equation}
If $\overline{t}_j = 0$, then $D_j = 0$ and we obtain $\Gamma(x,v,\overline{R}_j)  = \lambda_k(x) B(x,v,\overline{R}_j)$.  Since $\overline{R}_j \in E_k(x)^{\perp_{Q(x,.)}}$, this yields $\overline{R}_j = 0$, which contradicts 
$$1 = \overline{t}_j + \left\vert D_j \right\vert + \sqrt{Q(x,\overline{R}_j)}.$$
Thus $\overline{t}_j\neq 0$,  hence we can divide \eqref{eq:almost_done} by $ \overline{t}_j /Q(x,u_j)$ to get
$$ \frac{D_j}{\overline{t}_j} = \frac{\lambda_k(x) \langle Q_x(x,u_j),h \rangle -  \langle G_x(x,u_j) ,h \rangle}{Q(x,u_j)}$$
which easily rewrites as
\begin{equation*}
 \lim_{n\to +\infty} \frac{\tilde{\lambda}_j(w_n)-\lambda_{k}(x)}{t_n} = \langle\cR_x(x,u_j),h \rangle.
 \end{equation*}
 This is exactly \eqref{eq:step1}.  Moreover, choosing $v=u_l$ for $l\neq j$ in \eqref{eq:chooseBeta} and applying \eqref{eq:ortho},  we obtain that
\begin{align*}
0 &  = - \overline{t}_j \bigg( \langle \Gamma_x(x,u_j,u_l) ,h \rangle -\lambda_k(x) \langle B_x(x,u_j,u_l),h \rangle)  \bigg)
\end{align*}
which yields that $\{u_j\}_{1\leq j \leq m}$ is $\mathcal{B}$-orthogonal.  

\hfill

\textbf{Step 2.} We prove \eqref{eqminmaxdiff} and \eqref{min_max_diff}. Of course \eqref{min_max_diff} is a particular case of \eqref{eqminmaxdiff} for $y(t) = x+th$.
Now, if $y: (-\eps,\eps)\to X$ is differentiable, satisfies $y(0) = x$ and $y'(0) = h$, we have that $y(t) = x + t h + z(t)$, where $z(t)=o(t)$ as $t\to 0$. Let $t_n \downarrow 0$. We set $z_n = z(t_n)$.
By the previous step,  we know that there exists an extraction $\varphi(n) \to +\infty$ and a $Q$-orthonormal, $G$-orthogonal,  and $\mathcal{B}$-orthogonal family $\{u_j\}_{1 \le j \le m} \subset E_k(x)$ such that
\[
 \lim_{n\to +\infty}  \frac{\tilde{\lambda}_j(x +t_{\varphi(n)} h + z_{\varphi(n)})-\lambda_k(x)}{t_n}  = \langle \cR_x(x,u_j),h \rangle
 \]
 for any $j \in \{1,\cdots,m\}$. Note that for any $n$, \[\tilde{\lambda}_1(x +t_{\varphi(n)} h+z_{\varphi(n)}) \le \ldots \le \tilde{\lambda}_m(x +t_{\varphi(n)} h+z_{\varphi(n)}).\] Subtract $\lambda_k(x)$, divide by $t_{\varphi(n)}$ and pass to the limit as $n \to +\infty$ to get that
 \[\langle \cR_x(x,u_1),h \rangle \le \ldots \le \langle \cR_x(x,u_m),h \rangle.\]
Since $\{u_j\}$ is $Q$-orthonormal and $\mathcal{B}$-orthogonal, following the notations of Remark \ref{rk:derivative=eigenvalue},  each $u_i$ has to be an eigenfunction of $\mathcal{B}$ with eigenvalue $\nu_i(x,h)$. Then, there is a unique subsequential limit and we get Step 2 by the min-max characterization of eigenvalues of endomorphisms.

%
%
\end{proof}

%

We now build upon the previous result to give properties of the sub/superdifferentials of an eigenvalue.  We begin with a collection of elementary consequence of Theorem \ref{th:main}.

\begin{prop} Consider $x \in \Omega$ and $i\in \mathbb{N}^*$ such that  $m_i(x) \ge 2$.  Then for any $k \in \{0,\ldots,m_{i}(x)-2\}$,
\begin{equation}\label{eq:inclusionsub}
\partial^-[ \lambda_{j_i(x) + k} ](x) \subset \partial^-[ \lambda_{j_i(x) + k + 1} ](x),
\end{equation}
\begin{equation}\label{eq:inclusionsup}
\partial^+[ \lambda_{j_i(x) + k + 1 } ](x) \subset \partial^+[ \lambda_{j_i(x) + k} ](x),
\end{equation}
\begin{equation} \label{eqderivativemult}  [\lambda_{j_i(x)+k}]_r'(x;h)  =  [\lambda_{J_i(x) - k}]_\ell'(x; h) \qquad \forall \, h \in X,\end{equation}
\begin{equation} \label{eqsubsuperdifferentialmult}   \partial^{-} [\lambda_{j_i(x)+k}](x) = \partial^+ [\lambda_{J_i(x)-k}](x) .
\end{equation} 
\end{prop}

\begin{proof}
Equality  \eqref{min_max_diff} implies that $[\lambda_{j_i(x) + k}]_r'(x;h) \leq [\lambda_{j_i(x) + k+1}]_r'(x;h) $ for any $h \in X$, hence we get \eqref{eq:inclusionsub} and \eqref{eq:inclusionsup}. Let us prove \eqref{eqderivativemult}. From Equality  \eqref{min_max_diff} and Remark \ref{rk:derivative=eigenvalue},  we know that
\[
[\lambda_{j_i(x)+ k}]_r'(x;h) = \nu_{k+1}(x,h)
\]
where $\nu_{k+1}(x,h)$ is the $(k+1)$-th eigenvalue associated to the endomorphism $\mathcal{L}(x,h)$ defined by
\[
B(x,\mathcal{L}(x,h)(u),v) = \langle \Gamma_x(x,u,v) - \lambda_k(x) B_x(x,u,v), h\rangle
\]
for any $u,v \in F_i(x)$. From this defining equality, we get that $\mathcal{L}(x,h) = - \mathcal{L}(x,-h)$. Since the eigenvalues of $- \mathcal{L}(x,-h)$ sorted in non-decreasing order are $-\nu_{m_i(x)}(x,-h)\le \ldots \le -\nu_{1}(x,-h)$, we obtain that
\begin{align*}
\nu_{k+1}(x,h) & = - \nu_{m_i(x)-k}(x,-h)\\
& = -[\lambda_{j_i(x)+ m_i(x) - k-1}]_r'(x;-h)\\
& = [\lambda_{j_i(x)+ m_i(x) - k-1}]_\ell'(x;h) \,\, \qquad \text{by \eqref{eq:left_and_right}}\\
&= [\lambda_{J_i(x) - k}]_\ell'(x;h).
\end{align*}
Lastly, \eqref{eqsubsuperdifferentialmult} follows from \eqref{eqderivativemult}, \eqref{eq:left_and_right} and the definition of subdifferentials and superdifferentials.
\end{proof}

The previous implies a simple result for ``middle'' eigenvalues.

\begin{cor}
Consider $x \in \Omega$ and $i\in \mathbb{N}^*$ such that $J_i(x)-j_i(x)$ is even.  Set $k\df(J_i(x)+j_i(x))/2$.  Then the function $\lambda_k$ is directionally differentiable at $x$ in any direction $h \in X$.
\end{cor}

\begin{proof}
Set $\ell \df (J_i(x)-j_i(x))/2$ and note that $j_i(x)+\ell = J_i(x)- \ell = k$. Then \eqref{eqderivativemult} writes as $[\lambda_{k}]_r'(x;h)  =  [\lambda_k]_\ell'(x; h)$ for any $h \in X$.
\end{proof}

\begin{rem}
In the setting of the previous Corollary,  Proposition \ref{prop:C} immediately implies that, if there exists $h \in X$ such that $[\lambda_{k}]_r'(x;h)>0$, then $0 \notin \partial^- \lambda_{k} (x) \cup \partial^+ \lambda_{k} (x)$.
\end{rem}

In fact, we have the following more general characterization of the classical sub/superdifferential:
\begin{prop}\label{propclassicalsubdifferential}
For any $x \in \Omega$ and $k \in \mathbb{N}^*$, 
\begin{align}\label{eq:diff_single_rescaled_classical-}
\partial^- \lambda_k(x)  = A_k^-(x) \df  \bigcap_{F \in \mathcal{G}_{p_k(x)}(E_k(x))} \conv \bigg\{\,  \mathcal{R}_x(x,u) : u \in S(F) \bigg\}.
\end{align}
and
\begin{align}\label{eq:diff_single_rescaled_classical+}
\partial^+ \lambda_k(x)  = A_k^+(x) \df \bigcap_{F \in \mathcal{G}_{m_k(x) - p_k(x)+1}(E_k(x))} \conv \bigg\{\,  \mathcal{R}_x(x,u) : u \in S(F) \bigg\}.
\end{align}
\end{prop}

\begin{proof}
Let us prove \eqref{eq:diff_single_rescaled_classical-}, the other one being a consequence of \eqref{eqsubsuperdifferentialmult}. Let $ \zeta \in A_k^-(x) $ and
$h \in X$. By \eqref{min_max_diff}, we have the existence of $F_h \in  \mathcal{G}_{p_k(x)}(E_k(x))$ such that
\[
 [\lambda_k]_r' (x\, ,h) = \min_{F \in  \mathcal{G}_{p_k(x)}(E_k(x))} \max_{u \in \underline{F} } \langle \cR_x(x\, ,u),h\rangle = \max_{u \in \underline{F}_h } \langle \cR_x(x\, ,u),h\rangle  \ge \langle \cR_x(x\, ,v),h\rangle
\]
for any $v\in \underline{F}_h $. Knowing that $\zeta = \sum_{i} t_i \mathcal{R}_x(x,u_i)$ with $\sum_{i} t_i = 1$ and $0\leq t_i \leq 1$ and $u_i \in S(F_h)$,
by a convex combination of the previous inequality specified to some $u_i \in \underline{F}_h $, we obtain $[\lambda_k]_r' (x\, ,h)\geq \langle \zeta,h \rangle$. Since it is true for any $h$, $ \zeta   \in \partial^- \lambda_k(x)$.  Thus 
\[
A_k^-(x) \subset \partial^- \lambda_k(x).
\]
Let us prove the converse inclusion. We have by \eqref{min_max_diff} that
\begin{align*}
\partial^- \lambda_k(x) & = \{\zeta \in X^* : \forall h \in X , \left\langle \zeta,h \right\rangle \leq \min_{F \in \mathcal{G}_{p_k(x)}(E_k(x))} \max_{u\in \underline{F}} \left\langle \mathcal{R}_x(x,u),h \right\rangle  \}  \\
& = \{\zeta \in X^* : \forall h \in X, \forall F \in \mathcal{G}_{p_k(x)}(E_k(x)) , \left\langle \zeta,h \right\rangle \leq \max_{u\in \underline{F}} \left\langle \mathcal{R}_x(x,u),h \right\rangle  \} \\
& = \bigcap_{ F \in \mathcal{G}_{p_k(x)}(E_k(x)) } \{\zeta \in X^* : \forall h \in X, \left\langle \zeta,h \right\rangle \leq \max_{u\in \underline{F}} \left\langle \mathcal{R}_x(x,u),h \right\rangle  \} \\
& \subset \bigcap_{ F \in \mathcal{G}_{p_k(x)}(E_k(x)) } \conv \bigg\{\,  \mathcal{R}_x(x,u) : u \in S(F) \bigg\} = A_k^-(x)
\end{align*}
where the latter inclusion is due to the following consequence of a Hahn-Banach separation theorem: for any $F\in \mathcal{G}_{p_k(x)}(E_k(x))$,
\[ 
K_F:= \bigg\{\zeta \in X^* : \forall h \in X, \left\langle \zeta,h \right\rangle \leq \max_{u\in \underline{F}} \left\langle \mathcal{R}_x(x,u),h \right\rangle  \bigg\} \subset \conv \bigg\{\,  \mathcal{R}_x(x,u) : u \in S(F) \bigg\}
\]
Indeed, if $\zeta \notin  \conv \bigg\{\,  \mathcal{R}_x(x,u) : u \in S(F) \bigg\}$,  then there exists $h \in X$ such that 
\[  \forall \xi \in  \conv \bigg\{\,  \mathcal{R}_x(x,u) : u \in S(F) \bigg\}, \langle \zeta,h \rangle > \langle \xi,h \rangle
\]
Then for any $u\in S(F)$, $\left\langle \zeta,h \right\rangle > \left\langle \mathcal{R}_x(x,u),h \right\rangle$, so that $\zeta \notin K_F$.

\end{proof}

The previous formula suggests that $\partial^- \lambda_k(x)$ is often empty.  The next proposition provides an example.

\begin{prop} \label{remempty} Let $\{\lambda_k(\cdot)\}$ be the Laplace eigenvalues defined on the conformal class of the round metric $g$ of the sphere $\mathbb{S}^2$.  Then
$$ \partial^- \lambda_1(g) = \partial^- \lambda_2(g) = \partial^+ \lambda_2(g) = \partial^+ \lambda_3(g) = \emptyset.$$
\end{prop}

\begin{proof}
We prove that $ \partial^{\pm} \lambda_i(g)$ is invariant by rotation. By the previous formula, and a straightforward computation
$$ \partial^{\pm} \lambda_i(g) \subset \left\{ \zeta_\psi ; \psi = - 2\sum_{i=1}^3 t_i  \frac{x_i^2}{\int_{\mathbb{S}^2} x_i^2} \text{s.t.} t_i \geq 0, \sum_i t_i = 1 \right\} \text{ where } \zeta_{\psi}(\phi) = \int_{\mathbb{S}^2} \psi \phi $$
Let $R \in SO(\R^3)$, let $\zeta_{\psi} \in  \partial^{\pm} \lambda_i(g)$. Let's prove that $\zeta_{\psi \circ R} \in  \partial^{\pm} \lambda_i(g)$. Let $h\in X$. Then
\begin{equation*}
\begin{split} \left\langle \zeta_{\psi\circ R}, h \right\rangle = \int_{\mathbb{S}^2} \psi\circ R . h dA_g = \int_{\mathbb{S}^2} \psi . h\circ R^{-1} dA_g \\
= \left\langle \zeta_{\psi}, h\circ R^{-1} \right\rangle \leq \bar{\lambda}_i'(g, h\circ R^{-1}g) =  \bar{\lambda}_i'(g, h g)
\end{split} \end{equation*}
because the eigenvalues are invariant by isometry. We assume by contradiction that $\partial^{\pm} \lambda_i(g) \neq \emptyset$, let $\zeta_\psi \in \partial^{\pm} \lambda_i(g)$ where
$$ \psi = -2 \sum_{i=1}^3 t_i  \frac{x_i^2}{\int_{\mathbb{S}^2} x_i^2} $$
where $t_1+t_2+t_3 = 1$ and $t_i \geq 0$. We define $R_1,R_2,R_3 \in SO(3)$ such that
$$ R_1 = Id_{\R^3} ; $$ 
$$ R_2(e_1) = e_2 ; R_2(e_2) = e_3 ; R_2(e_3) = e_1 $$
$$ R_3(e_1) = e_3 ; R_3(e_2) = e_1 ; R_3(e_3) = e_2 $$
Then, by invariance by rotation, $\zeta_\psi\circ R_i \in \partial^{\pm} \lambda_i(g)$ for $i=1,2,3$ and since $\partial^{\pm} \lambda_i(g)$ is convex, 
$$ \frac{1}{3} \sum_{i=1}^3 \left( \zeta_\psi\circ R_i \right) \in \partial^{\pm} \lambda_i(g)$$
By a straitforward computation, knowing that $x_1^2 + x_2^2 + x_3^2 = 1$, we obtain that $- \frac{1}{2\pi} \in \partial^{\pm} \bar{\lambda}_i(g)$. However, we cannot find any constant function in a set 
$$\conv\{ \left\langle x,p \right\rangle^2 ; p \in E\cap \mathbb{S}^2 \}$$
where $E$ is a subspace of $\mathbb{R}^3$ of dimension lower than 2.
\end{proof}

\subsection{Critical points of a single eigenvalue}
We begin with the following result which is an immediate consequence of (2) in Lemma \ref{lem:minmax} and Proposition \ref{propclassicalsubdifferential}.
\begin{cor} 
If $\lambda_k$ admits a local minimum at $x$, then 
\begin{align}\label{eq:crit_single_rescaled_classical-}
0\in  \bigcap_{F \in \mathcal{G}_{p_k(x)}(E_k(x))} \conv \bigg\{\,    \cR_x(x\, ,u) : u \in S(F) \bigg\}.
\end{align}
If $\lambda_k$ admits a local maximum at $x$ then
\begin{align}\label{eq:crit_single_rescaled_classical+}
0\in \bigcap_{F \in \mathcal{G}_{m_k(x)-p_k(x)+1}(E_k(x))} \conv \bigg\{\,  \cR_x(x\, ,u)   : u \in S(F) \bigg\}.
\end{align}
\end{cor}

\begin{rem}
If $k=j_{i(k)}$ and if $\lambda_k$ admits a local minimum at $x$, then
$$ \cR_x(x\, ,u)  = 0 \qquad \forall u \in E_k(x).$$
There are two known examples where this phenomenon occurs.
\begin{itemize}
\item Minimizers of the first Dirac eigenvalue in a conformal class on surfaces \cite{Ammann}. In this case, it implies that a first eigenspinor is a harmonic mapping into a sphere. The construction of such mappings was a motivation to build CMC surfaces.
\item Minimizers of the second eigenvalue of the conformal Laplacian in a conformal class \cite{AmmannHumbert}. In this paper, we can even deduce that the multiplicity of the second eigenvalue of a local minimizer has to be one and that an eigenfunction is a nodal solution of the Yamabe problem. Notice that if the conformal Laplacian have negative eigenvalues, a local maximizer of the biggest negative eigenvalue also provides nodal Yamabe solutions.
\end{itemize}
\end{rem}

However, we know that minimality and maximality does not capture all the possible critical points. The classical sub/superdifferential is not adapted in the general case: we use the Clarke sub/superdifferential. We have the following:

\begin{prop}\label{prop:diff_single}
For any $x \in \Omega$ and $k \in \mathbb{N}^*$,  set
\[
C_k(x) \df \conv \bigg\{\,  \cR_x(x\, ,u) : u \in SE_{k}(x) \bigg\}.
\]
Then
\begin{align}\label{eq:diff_single_rescaled}
\partial_C^\pm \lambda_k(x)  \subset C_k(x).
\end{align}
Moreover:
\begin{align}\label{eq:diff_single_rescaled_equal}
\pm(\lambda_k(x) - \lambda_{k\pm1}(x))<0 \quad \Rightarrow \quad  \partial^\mp \lambda_k(x)  =   \partial_C^\mp \lambda_k(x) = C_k(x).
\end{align}
\end{prop}

\begin{proof}[Proof of Proposition \ref{prop:diff_single}]
We prove the proposition for $\partial_C^-$ --- the result for $\partial_C^+$ is proved in a similar way. Let's prove \eqref{eq:diff_single_rescaled}.

\textbf{Step 1.} We show the following compactness property of the set-valued map $C_k(\cdot)$ : for any $x \in \Omega$,  $\{y_n\}\subset \Omega$ satisfying $y_n \to x$,  and $\{\xi_n\} \subset X^*$ such that $\xi_n \in C_k(y_n)$ for any $n$, there exists $\xi \in C_k(x)$ such that $\xi_n \to \xi$ in $X^{*}$, up to a subsequence.

For any $n$ there exist $\{t_{\alpha,n}\}_{\alpha \in A(n)} \subset [0,1]$ and $(u_{\alpha,n}) \in SE_k(y_n)$ such that $A(n) \in \setN$ is non-empty, $\sum_\alpha t_{\alpha,n}=1$ and
\[
\xi_n = \sum_{\alpha} t_{\alpha,n} \cR_x(y_n,u_{\alpha,n}).
\]
Since $\{\mathrm{dim}(E_k(y_n))\}_n$ is bounded,  there exists a positive integer $M$ such that $A(n) \subset \{0,\ldots,M\}$ for any $n$. Set $t_{\alpha,n}=0$ for any $\alpha \in \{0,\ldots,M\} \backslash A(n)$. For any $\alpha \in \{0,\ldots,M\}$ the sequence $\{t_{\alpha,n}\}_n$ subconverges to some value $t_\alpha$, and $\sum_{\alpha} t_\alpha = 1$. By Corollary \ref{prop:closed_graph}, we know that for any $\alpha \in \{0,\ldots,M\}$ the sequence $\{u_{\alpha,n}\}_n$ subconverges to some $u_\alpha \in SE_k(x)$.  Then the continuity of $(x,u) \mapsto \cR_x(x\, ,u)$ implies that $\xi_n$ subconverges to 
\[
\xi = \sum_{\alpha} t_{\alpha} \cR_x(x,u_{\alpha})
\]
which belongs to $C_k(x)$.

\textbf{Step 2.} We show that for any $V \in \mathcal{G}_k(X)$,
\[
\limsup_{y\to x} \widetilde{\partial^- \left(\lambda_k\right)_{\vert y+V}}(y) \subset \{ \zeta \in X^* \, : \,  \exists \, \xi \in C_k(x) \text{ s.t. }  \zeta_{\vert V} = \xi_{\vert V} \}. 
\]

By definition,
$$ \partial^- \left(\lambda_k\right)_{\vert y+V}(y) =  \left\{ \zeta \in V^* \, : \,  \forall \, h \in V,  \, \, \langle\zeta,h \rangle \leq \left(\left(\bar{\lambda}_k\right)_{\vert y+V}\right)_r'(y,h)  \right\}.  $$
Acting as in the proof of Proposition \ref{propclassicalsubdifferential}, we get that
$$ \partial^- \left(\lambda_k\right)_{\vert y+V}(y) =  \bigcap_{F \in \mathcal{G}_{p_k(x)}(E_k(x))} \conv \{ \zeta \in V^* \, : \,  \exists u \in S(F), \zeta = \cR_x(y\, ,u)_{\vert V}   \}  $$
We easily deduce that
$$  \widetilde{\partial^- \left(\lambda_k\right)_{\vert y+V}}(y) \subset \{ \zeta \in X^* \, : \,  \exists \xi \in C_k(y),  \zeta_{\vert V} = \xi_{\vert V} \}.$$
Take $\limsup_{y\to x}$ and use Step 1 to obtain that
\begin{align*}
\limsup_{y\to x} \widetilde{\partial^- \left(\lambda_k\right)_{\vert y+V}}(y) &  \subset \limsup_{y\to x} \{ \zeta \in X^* \, : \,  \exists \xi \in C_k(y),  \zeta_{\vert V} = \xi_{\vert V} \}\\
& \subset \{ \zeta \in X^* \, : \,  \exists \xi \in C_k(x),  \zeta_{\vert V} = \xi_{\vert V} \}.
\end{align*} 

\textbf{Step 3.} We conclude. Step 2 ensures that the inclusion below holds true :
\begin{align*}
\partial_A^- \lambda_k(x) & = \bigcap_{V \in \bigcup_{\ell \in \setN}\mathcal{G}(X) } \limsup_{y \to x} \widetilde{\partial^- \left(\lambda_k\right)_{\vert y+V}}(y)\\
& \subset  \bigcap_{V \in \bigcup_{\ell \in \setN}\mathcal{G}(X)}  \{ \zeta \in X^* \, : \,  \exists \, \xi \in C_k(x) \text{ s.t. }  \zeta_{\vert V} = \xi_{\vert V} \} \\
& = C_k(x).
\end{align*}
Then Proposition \ref{prop:IOFFE} and  the fact that $C_k(x)$ is closed and convex yields \eqref{eq:diff_single_rescaled}.


As for \eqref{eq:diff_single_rescaled_equal}, this is a direct consequence of Proposition \ref{propclassicalsubdifferential}. Indeed, if $\lambda_k(x) < \lambda_{k+1}(x)$, then $\partial^{-}\lambda_k(x) = A_k^-(x) = C_k(x)$, so the inclusions 
$\partial^{-}\lambda_k(x) \subset \partial_C^{-}\lambda_k(x) \subset  C_k(x)$ are all equalities.  In the same way, if $\lambda_k(x) > \lambda_{k-1}(x)$, then $\partial^{+}\lambda_k(x) = A_k^+(x)=C_k(x)$ hence $\partial^{+}\lambda_k(x)= \partial_C^{+}\lambda_k(x) = C_k(x)$.
\end{proof}

\begin{rem}
On the round sphere, by Proposition \ref{remempty}, we have that $0 \notin \partial^{\pm} \lambda_2(g)$. However, by a similar proof as the one for Proposition \ref{remempty}, since $\partial_C^{\pm} \lambda_i(g)$ is not empty, we deduce by invariance by rotation and Proposition \ref{prop:diff_single} that $0 \in \partial_C^{\pm} \lambda_i(g)$.
\end{rem}

\begin{rem}
If $k=j_{i(k)}$ and if $\lambda_k$ admits a local minimum at $x$, then $\lambda_k$ is differentiable at $x$ and $ D \bar{\lambda}_k(x) = 0 $.
\end{rem}

\begin{prop} For $x\in \Omega$ and $k\in \mathbb{N}^*$, we consider the following statements.
\begin{itemize}
\item[(1$\pm$)] $0 \in  \partial^{\pm} \lambda_k(x)$.
\item[(1)] $0 \in  \partial^{+}\lambda_k(x)  \cup  \partial^{-} \lambda_k(x)$.
\item[(2)]  $\sup_{h \in X} [\lambda_k]_r'(x\, ;h) [\lambda_k]_\ell'(x\, ;h) \leq 0$.
\item[(3)] $0 \in  \conv\{ \partial^{+} \lambda_k(x)  \cup  \partial^{-} \lambda_k(x)\}$.
\item[(3C)] $0 \in  \conv\{ \partial_C^{+} \lambda_k(x)  \cup  \partial_C^{-} \lambda_k(x)\}$.
\item[(4)] For any $h \in X$, the quadratic form $u \in E_k(x) \mapsto \left\langle \cR_x(x\, ,u), h \right\rangle$ is indefinite.
\end{itemize}
Then: $$(1\pm) \Rightarrow (1) \Rightarrow (2)  \Rightarrow  (3) \Rightarrow (3 \mathrm{C}) \Rightarrow (4).$$  Moreover: $$\pm(\lambda_k(x) - \lambda_{k\pm1}(x))<0 \Rightarrow \big[(4) \Rightarrow (1\pm)\big].$$
\end{prop}

\begin{proof} The implications $(1\pm) \Rightarrow (1)$ and $(3) \Rightarrow (3\mathrm{C})$ are obvious, while $(1) \Rightarrow (2)  \Rightarrow  (3)$ come from a direct application of Proposition \ref{prop:CC}.  The implication $(3)  \Rightarrow  (4)$ is a consequence of Proposition \ref{prop:diff_single}, since (4) is obviously equivalent to $0 \in C_k(x)$.  The equality case in Proposition \ref{prop:diff_single} implies that: $\pm(\lambda_k(x) - \lambda_{k\pm1}(x))<0$ $\Longrightarrow$  $\partial^\mp \lambda_k(x) = \partial_C^\mp \lambda_k(x) = C_k(x)$ and the latter equalities trivially imply that $(4) \Rightarrow (1\mp)$.
\end{proof}

\subsection{Finite combinations of eigenvalues} Let us now consider a positive integer $N$ and a function $F \in \mathcal{C}^1(\setR^N,\setR)$ with partial derivatives $\partial_1 F,\ldots, \partial_N F$.  For any $x \in \Omega$,  define
 \[\mathfrak{F}(x) \df F\left(\lambda_1(x),\cdots, \lambda_N(x)\right)\] 
 and for any $k \in \setN^*$, set
 \[
 d_k(x) \df 
  \begin{cases}
 \partial_k F\left(\lambda_1(x),\cdots,\lambda_N(x)\right) &  \text{if $1 \le k \le N$}\\
 0 & \text{otherwise}.
 \end{cases}
 \]
Recall that $\mathbf{U}_{N}(x)$ is defined at the end of Section \ref{subsec:abstract_Rayleigh}.
 Then the following holds.

\begin{prop}\label{prop:diff}
For any $x \in \Omega$, set 
\[
C'_N(x) \df \conv \left\{ \sum_{k=1}^N d_k(x)  \cR_x(x\, ,u_k)   \,  :  (u_k) \in \mathbf{U}_N(x) \right\}.
\]
Then 
\begin{align}\label{eq:diff}
\partial_C^{\pm} \mathfrak{F}(x) & \subset C'_N(x).
\end{align}
We also have that
\begin{align}\label{eq:diff_rescaled_equal_1}
\begin{cases}
\lambda_N(x) < \lambda_{N+1}(x)\\
\text{$k \mapsto d_k(x)$ is constant on each $\aleph_i(x)$}
\end{cases}
\Rightarrow  \quad \begin{cases}
\partial_C^- \mathfrak{F}(x) \cup \partial_C^+ \mathfrak{F}(x) = C'_N(x) \\
\sharp C'_N(x) = 1,
\end{cases}
\end{align}
in particular, $\fk$ is differentiable at $x$. More generally:
\begin{align}\label{eq:diff_rescaled_equal_2}
\forall   k \left[ d_k(x) \neq 0  \, \Leftrightarrow \, \begin{cases} \pm d_k(x)(\lambda_k(x) - \lambda_{k\pm1}(x))<0 \\ \text{ or $k \mapsto d_k(x)$ is constant on $\aleph_{i_k(x)}(x)$} \end{cases} \right] \Rightarrow  \partial_C^{\pm} \mathfrak{F}(x) = C'_N(x).
\end{align}
\end{prop}

\begin{proof}
We prove \eqref{eq:diff} only for $\pm$ equal to $-$, since the other case is obtained in a similar way. 

Let us first prove that the result is true for the classical subdifferential: $\partial^{-} \mathfrak{F}(x)  \subset C'_N(x).$
 Consider $\zeta \notin C_N'(x) $. By the Hahn-Banach theorem, there exists $h\in X$ such that $\langle \xi , h \rangle <  \langle \zeta , h \rangle$ for any $\xi \in C_N'(x)$. 
By Theorem \ref{th:main}, for any sequence $t_n \searrow 0$, up to a subsequence we $t_n$, we have the existence of orthonormal families $( u_k ) \in U_N(x)$ such that 
\begin{equation} \label{eq:cases} \lim_{n\to+\infty} \frac{\lambda_k(x +t_n h) - \lambda_k(x)}{t_n}  = \langle \cR_x(x,u_k),h \rangle
\end{equation} 
for any $1\leq k\leq N$. Since $F$ is a $\mathcal{C}^1$ function, we have that
\begin{align*}
\fk'_r(x,h) &  = \sum_{k=1}^N d_k(x) \lim_{n\to+\infty} \frac{\lambda_k(x +t_n h) - \lambda_k(x)}{t_n}   = \langle \xi , h \rangle
\end{align*}
where
\[
\xi \df \sum_{k=1}^N d_k(x)  \cR_x(x,u_k) 
\]
belongs to $ C_N'(x)$. Then $\fk'_r(x,h) < \langle \zeta , h \rangle$.  This means that $\zeta \notin \partial^{-} \fk(x)$. 

To obtain \eqref{eq:diff} from the inclusion $\partial^{-} \mathfrak{F}(x)  \subset C'_N(x)$, we simply act as in the proof of Proposition \ref{prop:diff_single} using the compactness property of the set-valued function $C_N'(\cdot)$, the approximate sub/superdifferentials, and Proposition \ref{prop:IOFFE}.

Assume that $\lambda_N(x)<\lambda_{N+1}(x)$. Consider $h \in X$ and $(u_k) \in \mathbf{U}_N(x)$ satisfying \eqref{eq:cases}. Consider $\xi \in C'_N(x)$ such that
\[
\xi = \sum_{k=1}^N d_k(x)   \cR_x(x,v_k) 
\]
for some $(v_k) \in \mathbf{U}_N(x)$. For any $i \in \{1,\ldots,i_N(x)\}$ there exist $(P_{jk}^{i})_{j,k \in \aleph_i(x)} \in \cO(m_i(x))$ such that
\[
v_k = \sum_{j \in \aleph_i(x)} P_{kj} u_j
\]
for any $k \in \aleph_i(x)$. Thus
\begin{align*}
 & \phantom{=}  \phantom{=}\sum_{k=1}^N d_k(x) \langle \cR_x(x\, ,v_k) , h \rangle & \\
 &  =  \sum_{i=1}^{i_N(x)} d_{j_i(x)}(x) \sum_{k \in \aleph_i(x)}  \langle \cR_x(x\, ,v_k) , h \rangle & \\
 &  =  \sum_{i=1}^{i_N(x)} d_{j_i(x)}(x) \sum_{k \in \aleph_i(x)}  \sum_{j,j' \in \aleph_i(x)} P_{jk}P_{j'k} \langle \Gamma_x(x,u_j,u_{j'}) - \lambda_k(x) B_x(x,u_j,u_{j'}), h\rangle & \\
 &  =  \sum_{i=1}^{i_N(x)} d_{j_i(x)}(x) \sum_{k \in \aleph_i(x)}  \sum_{j \in \aleph_i(x)} P_{jk}^2 \langle \cR_x(x\, ,u_j) , h \rangle & \\
 &  =  \sum_{i=1}^{i_N(x)} d_{j_i(x)}(x)   \sum_{j \in \aleph_i(x)} \underbrace{\sum_{k \in \aleph_i(x)}P_{jk}^2}_{=1} \langle \cR_x(x\, ,u_j) , h \rangle  = \sum_{k=1}^N d_k(x) \langle \cR_x(x\, ,u_k) , h \rangle = \fk_+^\circ(x,h).
 \end{align*}
Since $h$ is arbitrary, the previous implies that
\[
C'_N(x) = \left\{ \sum_{k=1}^N d_k(x)  \cR_x(x,u_k) \, : \, (u_k) \in \mathbf{U}_N(x)\right\} \subset \partial_C^- \fk(x) \cup  \partial_C^+ \fk(x),
\] 
as claimed, since $\partial_C^\pm \fk(x)$ have to be non-empty.

Let us now prove \eqref{eq:diff_rescaled_equal_2}. Assume that for any $k \in \{1,\ldots,N\}$, one has 
\[ d_k(x) \neq 0  \,\, \Leftrightarrow \,\, \begin{cases} d_k(x)\left(\lambda_k(x)-\lambda_{k+1}(x)\right) < 0 \\ \text{ or $k \mapsto d_k(x)$ is constant on $\aleph_{i_k(x)}(x)$} \end{cases}
 \] 
Consider $h \in X$ and $\xi \in C'_N(x)$ written as
\[ \xi  = \sum_{k=1}^N d_k(x)  \cR_x(x,v_k) \]
for some $(v_k) \in \mathbf{U}_N(x)$. We can write $\xi = \xi_1 + \xi_2 + \xi_3$ with
\begin{align*}
\xi_1 & \coloneqq  \sum_{i \in I_1} d_{J_i(x)}(x) \cR_x(x,v_{J_i(x)})  
\end{align*}
\begin{align*}
\xi_2 & \coloneqq   \sum_{i \in I_2} d_{j_i(x)}(x) \cR_x(x,v_{J_i(x)})
\end{align*}
\begin{align*}
\xi_3 & \coloneqq   \sum_{i \in I_3}  \sum_{k \in \aleph_i(x) }  d_k(x)  \cR_x(x,v_{k})  
\end{align*}
where $I_3$ is the set of indices $i$ such that $k\mapsto d_k(x)$ is constant on $\aleph_i(x)$, $I_1$ is the set of indices that are not in $I_3$ such that $d_{J_i}(x)>0$ and $I_2$ is the set of indices that are not in $I_3$ such that $d_{j_i}(x)>0$.
As for the proof of \eqref{eq:diff_rescaled_equal_1}, we obtain that
\begin{align*}
\xi_3 & =  \sum_{i \in I_3} d_k(x) \sum_{k \in \aleph_i(x) }   \cR_x(x,u_{k})  
\end{align*}
where $(u_k)$ are defined in \eqref{eq:cases}. Moreover, Theorem \ref{th:main} implies that for any $i\in \{1,\ldots,i_N(x)\}$, 
\[
[\lambda_{J_i(x)}]_r(x,h) = [\lambda_{J_i(x)}]^\circ_+(x,h) = \max_{u \in \underline{F}_{i}(x)} \langle \cR_x(x ; u), h  \rangle
\]
Thus, by \eqref{eq:cases}
\begin{align*}
\langle \xi_1 , h \rangle & = \sum_{i \in I_1} d_{J_i(x)}(x) \langle \cR_x(x,v_{J_i(x)}),h\rangle \\
& \le \sum_{i \in I_1} d_{J_i(x)}(x)    [\lambda_{J_i(x)}]^\circ_+(x,h)\\
& = \sum_{i \in I_1} d_{J_i(x)}(x)   \langle \cR_x(x,u_{J_i(x)}),h\rangle 
\end{align*}
The same reasoning easily adapts to $\xi_2$ since Theorem \ref{th:main} implies that for any $i\in \{1,\ldots,i_N(x)\}$, 
\[
[\lambda_{j_i(x)}]^\circ_+(x,h) = \min_{u \in \underline{F}_{i}(x)} \langle \cR_x(x ; u), h  \rangle
\]
so that
\[ \langle \xi_2 , h \rangle \leq \sum_{i \in I_2} d_{j_i(x)}(x) \langle \cR_x(x,u_{J_i(x)}),h\rangle. \]
In the end we get that
\[ \langle \xi,h \rangle = \langle \xi_1,h \rangle + \langle \xi_2,h \rangle +\langle \xi_3,h \rangle   \le \fk^\circ_+(x,h)  \]
from where we conclude, again, by the arbitrariness of $h$ and $\xi$. The other case of\eqref{eq:diff_rescaled_equal_2} is similar.
\end{proof}

\subsection{Euler--Lagrange characterization of critical points}

In this section, we provide an Euler--Lagrange equation that we shall constantly use in the rest of the paper. 

Let $\scaling : \Omega \to (0,+\infty)$ be a $\mathcal{C}^1$-Fréchet differentiable scaling function.  We consider the rescaled Rayleigh quotient defined by
\begin{align*}
\overline{\cR}(x,v)& \coloneqq \scaling(x)\,  \cR(x,v)
\end{align*}
for any $x \in \Omega$ and $v \in H$. We let 
\[-\scaling(x)\Theta(x) \le  \overline{\lambda}_1(x) \le   \overline{\lambda}_2(x) \le \cdots \to +\infty.\] 
be the associated isolated eigenvalues. Note that
\[
\overline{\cR}_x(x,v) = \scaling(x) \cR_x (x,v) + \scaling_x(x) \cR (x,v).
\]
Consider a positive integer $N$ and a function $F \in \mathcal{C}^1(\setR^N,\setR)$ with partial derivatives $\partial_1 F,\ldots, \partial_N F$.  For any $x \in \Omega$,  define
 \[\mathfrak{F}(x) \df F\left(\overline{\lambda}_1(x),\cdots, \overline{\lambda}_N(x)\right)\] 
 and for any $k \in \setN^*$, set
 \[
 d_k(x) \df 
  \begin{cases}
 \partial_k F\left(\overline{\lambda}_1(x),\cdots,\overline{\lambda}_N(x)\right) &  \text{if $1 \le k \le N$}\\
 0 & \text{otherwise}.
 \end{cases}
\]
We use the shorthand $d_A$ to denote a tuple of numbers $(d_i)_{i\in A}$.  For any $x \in \Omega$, we set
\[
S(x)\df  \sum_{k=1}^{N} d_k(x) \lambda_k(x)
\]
and
\[
c(x) \df
\begin{cases}
\qquad 0 & \text{if  $S(x) =0$}\\
S(x)/|S(x)| & \text{otherwise.}
 \end{cases}
\]
Recall that $\Mix(\cdot)$ is defined in \eqref{eq:mix}. We also refer to the end of Section \ref{subsec:abstract_Rayleigh} for the definitions of $\cU_{J_N(x)}(x)$ and $\mathbf{U}_{J_N(x)}(x)$.

\begin{theorem}\label{th:main_cri}
Let $x \in \Omega$ be critical for $\fk$.  Then the following hold.
\begin{enumerate}
\item[(A)]  There exist $(u_k) \in \mathbf{U}_{J_N(x)}(x)$ and $\tilde{d} \in \prod_{i=1}^{i_N(x)} \Mix(d_{\aleph_i(x)})$ such that
\begin{equation}\label{eq:EL1}\tag{EL}
\sum_{k=1}^{J_N(x)} \tilde{d}_k G_x(x,u_k) + \frac{\scaling_x(x)}{\scaling(x)} S(x) = \sum_{i=1}^{i_N(x)}  \mu_i(x) \sum_{k\in \aleph_i(x)}  \tilde{d}_k Q_x(x,u_k).
\end{equation}
\item[(B)] There exist $(\phi_k) \in \cU_{J_N(x)}(x)$ and $(\eps_k) \in \{0,\pm 1\}^{J_N(x)}$ such that
\begin{equation}\label{eq:EL2}\tag{r-EL}
\sum_{k=1}^{J_N(x)} \eps_k G_x(x,\phi_k) +  \frac{\scaling_x(x)}{\scaling(x)} c(x) = \sum_{i=1}^{i_N(x)}  \mu_i(x) \sum_{k\in \aleph_i(x)}  \eps_k Q_x(x,\phi_k).
\end{equation}
\item[(C)] For any $k \in \{1,\ldots,J_N(x)\}$,  the $\phi_k$ from (B) and the $\tilde{d}_k$ from (A) satisfy
\begin{equation}\label{eq:normal}
Q(x,\phi_k) =  \begin{cases} \,\,  \, |\tilde{d}_k| & \text{if S(x)=0},\\
\displaystyle\frac{|\tilde{d}_k|}{\left| S(x) \right|}  & \text{otherwise}.
\end{cases}
\end{equation}
\item[(D)] Assume that there exists $i \in \{1,\ldots,i_N(x)\}$ such that for any $k \in \aleph_i(x)$,
\begin{equation}\label{eq:pos}
d_k(x) > 0.
\end{equation}
Then $\tilde{d}_k(x)>0$ and $\eps_k=1$ for any $k \in  \aleph_i(x)$.
\end{enumerate}
\end{theorem}

\begin{proof}
For any $\mu>0$,  let $A_\mu : H \to X$ be the quadratic form defined by
\[
A_\mu(u) \df \scaling(x) \left( G_x(x,u) - \mu \, Q_x(x,u) \right).
\]
With this notation, Proposition \ref{prop:diff} applied with $\overline{\cR}$ and \eqref{eq:Minkowski} imply that
\begin{align*}
\partial_C^\pm \fk(x)  \subset \scaling_x(x) \sum_{k=1}^{N} d_k(x) \lambda_k(x)  + \sum_{i=1}^{i_N(g)} \conv &  \Bigg\{ \sum_{k\in \aleph_i(x)} d_k(x) A_{\mu_i(x)}(u_k)  :   \Bigg. \\
& \qquad  \quad \,\,  \,\quad  \Bigg. (u_k) \in \cO(F_i(x))  \Bigg\}.
\end{align*}
Lemma \ref{lem:mixing} applied to each convex hull yields that
\begin{align*}
\partial_C^\pm \fk(x) & \subset \scaling_x(x) \sum_{k=1}^{N} d_k(x) \lambda_k(x) + \sum_{i=1}^{i_N(g)} \Bigg\{ \sum_{k\in \aleph_i(x)} \tilde{d}_k A_{\mu_i(x)}(u_k)  \Bigg. : (u_k) \in \cO(F_i(x)), \\
& \qquad \qquad  \qquad \qquad  \qquad  \qquad  \qquad \qquad  \quad  \quad  \qquad \Bigg.    \text{and} \, \,  \tilde{d}_{\aleph_i(x)} \in \Mix(d_{\aleph_i(x)})\Bigg\},
\end{align*}
 According to Definition \ref{def:crit}, we have $0 \in \partial_C^- \fk(x) \cup \partial_C^+ \fk(x)$. Then there exists $(u_k) \in \prod_{i=1}^{i_N(x)}  \cO(F_i(x))$ and $\tilde{d} \in \prod_{i=1}^{i_N(x)} \Mix(d_{\aleph_i(x)})$ such that
\begin{align*}
\scaling_x(x) \sum_{k=1}^{N} d_k(x) \lambda_k(x) & = - \sum_{i=1}^{i_N(x)} \sum_{k\in \aleph_i(x)} \tilde{d}_k A_{\mu_i(x)}(u_k).
\end{align*}
Upon using the definition of $A_{\mu_i(x)}(u_k)$  and dividing by $\scaling(x)$, we obtain  \eqref{eq:EL1}.  To get \eqref{eq:EL2},  we first set $\psi_k \df \sqrt{|\tilde{d}_k|} u_k$ for any $k \in \{1,\ldots J_N(x)\}$. If $S(x)=0$ then we set $\phi_k=\psi_k$ for any $k$, otherwise we set
\[
\phi_k \df\frac{\psi_k}{\sqrt{\left| \sum_{k=1}^{N} d_k(x) \lambda_k(x)\right| }}
\]
and we do have \eqref{eq:normal}. The last two points are obvious from the proof.
\end{proof}

Let us state a practical converse of the previous theorem -- note that the assumption below is even weaker than the conclusion of the latter. 
 This result will allow us to interpret some suitable mappings which coordinates are eigenfunctions as critical points of well-chosen eigenvalue functionals.

\begin{theorem}\label{th:converse}
Assume that for some $x \in \Omega$, there is a finite family $(\phi_j)_{j=1,\cdots,l}$ such that $\phi_j \in  E_{k_j}(x)$ for any $j$, where $k_j$ is some positive integer, with associated weights $(\delta_j)_{j=1,\cdots,l} \in \mathbb{R}^l$ such that
\begin{equation}\label{eq:converseELnonortho}
\sum_{j=1}^{l} \delta_j \mathcal{R}_x(x,\phi_j) + \frac{\scaling_x(x)}{\scaling(x)} \sum_{j=1}^l \delta_j \lambda_{k_j}(x) = 0.
\end{equation}
Then there  $N \in \mathbb{N}^*$ and $d_1^{\pm},\cdots,d_N^\pm \in \R$ such that for any $F$ satisfying 
\begin{equation} \label{eqdefdiconverse}\partial_i F^\pm(\bar{\lambda}_1(x),\cdots,\bar{\lambda}_N(x)) = d_i^\pm \end{equation}
for any $i\in \{1,\cdots, N\}$,  setting  $\fk^\pm := F^\pm \circ (\lambda_1,\ldots,\lambda_N)$ we have
\begin{equation}\label{eq:converse}
0 \in \partial^\pm \fk^\pm (x).
\end{equation}

\end{theorem}

\begin{proof} Let us prove the case where $\pm$ is $+$. The other case is left to the reader. 

\medskip

\textbf{Step 1:} We prove that up to rearrangements of the family of eigenfunctions $(\phi_j)$, and a multiplication by a constant in \eqref{eq:converseELnonortho} we can find $N \in \mathbb{N}^*$ a family $(u_k) \in \mathbf{U}_{J_N(x)}(x)$ and $\tilde{d}_1,\cdots, \tilde{d}_N \in \R$ and $c \in \{0,\pm 1\}$ such that
\begin{equation}\label{eq:converseEL} \sum_{i=1}^{i_N(x)} \sum_{k \in \aleph_i(x)} \tilde{d}_k \mathcal{R}_x(x,u_k) + \frac{\scaling_x(x)}{\scaling(x)} \sum_{k=1}^N \tilde{d}_k \lambda_k(x)  = 0 \end{equation}

\medskip

We denote by $N$ the highest index between all the eigenvalues associated to the eigenfunctions $(\phi_j)$. We reorder and renumber $(\phi_j)$ as $(\phi_{i}^j)$  for $i = 1,\cdots i_N(x)$ and $ j = 1,\cdots,l_i$ so that $\phi_i^j \in F_i(x)$ and $l_i$ is the number of eigenfunctions in the family associated to $\mu_i(x)$. Let $(\delta_i^j)$ be the relabelled associated weights. Notice that \eqref{eq:converseELnonortho} becomes
\begin{equation} \label{eqrelabelledconverse}
\sum_{i=1}^{i_N(x)} \sum_{j=1}^{l_i} \delta_i^j \frac{ G_x(x,\phi_i^j) - \mu_i(x) Q_x(x,\phi_i^j)}{Q(x,\phi_i^j)} + \frac{\scaling_x(x)}{\scaling(x)} \sum_{i=1}^{i_N(x)} \sum_{j=1}^{l_i} \delta_i^j \mu_i(x) = 0
\end{equation}
For any $\mu>0$,  let $A_\mu : H \to X$ be the quadratic form defined by
\[
A_\mu(u) \df \scaling(x) \left( G_x(x,u) - \mu \, Q_x(x,u) \right).
\]
Then,  setting $\tilde{\phi}_i^j \df \frac{\phi_i^j}{\sqrt{Q(x,\phi_i^j)}}$, \eqref{eqrelabelledconverse} becomes
\begin{equation} \label{eqrelabelledconverse2}
\sum_{i=1}^{i_N(x)} \sum_{j=1}^{l_i} \delta_i^j A_{\mu_i(x)}(\tilde{\phi}_i^j) + \scaling_x(x) \sum_{i=1}^{i_N(x)} \sum_{j=1}^{l_i} \delta_i^j \mu_i(x) = 0
\end{equation}
We let $(v_i^p)$ for $i = 1,\cdots i_N(x)$ and $ p = 1,\cdots, m_i(x)$ be an orthonormal family of eigenfunctions with respect to $Q(x,.)$ such that $v_i^p \in F_i(x)$. We then write $\tilde{\phi}_i^j$ in the basis $(v_i^p)_{p=1,\cdots,m_i(x)}$:
\begin{equation} 
\label{eqbasisconverse} \tilde{\phi}_i^j = \sum_{p=1}^{m_i(x)} \alpha_{i,j,p} v_i^p \end{equation}
Denoting $B_{\mu_i(x)}$ the symmetric bilinear form associated to $A_{\mu_i(x)}$,
\begin{equation} \label{eqdevelopbilinear} \sum_{j=1}^{l_i} \delta_i^j A_{\mu_i(x)}(\tilde{\phi}_i^j) = \sum_{1 \leq p,q \leq m_i(x)} \sum_{j=1}^{l_i} \delta_i^j \alpha_{i,j,p} \alpha_{i,j,q} B_{\mu_i(x)}(v_i^p,v_i^q)  \end{equation}
We set the matrix $A_i \in \mathcal{M}_{m_i(x)}(\R)$ as
$$ (A_i)_{p,q} = \sum_{j=1}^{l_i} \delta_i^j \alpha_{i,j,p} \alpha_{i,j,q} $$
This matrix is symmetric so that there is $P_i \in \mathbf{O}_{m_i(x)}(\R)$ and $D_i \in \mathcal{M}_{m_i(x)}(\R)$ a diagonal matrix such that $A_i = P_i^T D_i P_i$.  \eqref{eqdevelopbilinear} becomes
\begin{eqnarray*}  \sum_{j=1}^{l_i} \delta_i^j A_{\mu_i(x)}(\tilde{\phi}_i^j) = \sum_{1 \leq p,q,r \leq m_i(x)} (P_i)_{r,p} (D_i)_{r,r} (P_i)_{r,q}  B_{\mu_i(x)}(v_i^p,v_i^q) \\ =  \sum_{r=1}^{ m_i(x)}  (D_i)_{r,r} B_{\mu_i(x)}\left(\sum_{p=1}^{m_i(x)}(P_i)_{r,p}v_i^p, \sum_{q=1}^{m_i(x)}(P_i)_{r,q} v_i^q\right) 
\end{eqnarray*}
so that setting $\tilde{d}_{i}^j := (D_i)_{r,r}$, $u_i^r := \sum_{p=1}^{m_i(x)}(P_i)_{r,p}v_i^p$, noticing that 
$$ \sum_{j=1}^{l_i} \delta_i^j = \tr(A_i) = \sum_{j=1}^{m_i(x)} \tilde{d}_i^j$$ 
and up to make a relabelling, we obtain from \eqref{eqrelabelledconverse2} the desired formula \eqref{eq:converseEL}.

\medskip

\textbf{Step 2:} We define $d_1,\cdots,d_N$ in order to get the theorem for $\pm = +$

\medskip

Let $d_1,\cdots,d_N$ we shall define later. Consider $\fk = F \circ (\lambda_1,\ldots,\lambda_N)$ where $F$ satisfies  \eqref{eqdefdiconverse} for $\pm = +$.  Recall that $\bar{\lambda}_{j_1(x)},\cdots,\bar{\lambda}_{j_N(x)}$ are lower regular at $x$ and $\bar{\lambda}_{J_1(x)},\cdots,\bar{\lambda}_{J_N(x)}$ are upper regular at $x$. We assume that for $i\in \{1,\cdots,i_N(x)\}$, $d_{j_i(x)} >0$ and $d_{J_i(x)} <0$ and that $d_k = 0$ if $k \notin \{j_i(x),J_i(x) ; i = 1,\cdots i_N(x) \}$. Then \eqref{eq:chain_rule_Clarke} and \eqref{eq:equal_regular} imply that
\begin{align*}
\partial^+\fk(x)  & = \sum_{i=1}^{i_N(x)}  \bigg( d_{j_i(x)} \partial^+ [\bar{\lambda}_{j_i(x)}] (x) + d_{J_i(x)}\partial^- [\bar{\lambda}_{J_i(x)}](x) \bigg)\\
& = \sum_{i=1}^{i_N(x)}  \bigg( d_{j_i(x)} C_{j_i(x)}(x) + d_{J_i(x)}C_{J_i(x)}(x) \bigg)
\end{align*}
where we have used the equality case in Proposition \ref{prop:diff_single} to get the second line. We have $C_{j_i(x)}(x) = C_{J_i(x)}(x)$, and by lemma \ref{lem:mixing},
\begin{align*}
C_{j_i(x)}(x) = C_{J_i(x)}(x) & = \scaling_x(x)  \mu_i(x) + \scaling(x) \, \mathcal{C}
\end{align*}
where
$$ \mathcal{C} = \Bigg\{ \sum_{k \in \aleph_i(x)} \tilde{\delta}_k \mathcal{R}_x(x,u_k) \, :   \, (u_k) \in \mathbf{O}(F_i(x)),  \tilde{\delta} \in \Mix((1,0,\cdots,0)) \Bigg\}. $$
Then, setting  for $i\in \{1,\cdots, i_N(x)\}$
$$ d_{j_i(x)} := \sum_{j \in \{1,\cdots,m_i(x)\} ; \tilde{d}_i^j > 0 } \tilde{d}_i^j \text{ and } d_{J_i(x)} := \sum_{j \in \{1,\cdots,m_i(x)\} ; \tilde{d}_i^j < 0 } \tilde{d}_i^j   $$
we immediately deduce Step 2 from Step 1.
\end{proof}

\section{Laplace and Steklov functionals on surfaces}\label{sec:surfaces}

In this section, we apply the theory developed in the previous section to functionals depending either on Laplace or Steklov eigenvalues on a compact,  connected, smooth surface $\Sigma$.  We consider both criticality over the whole set of Riemannian metrics of $\Sigma$ and over a fixed conformal class. In this way, we recover results from \cite{Nadi,EI1,FraserSchoen,PetridesEllipsoids}. We let $N$ be a positive integer and $F \in \cC^1(\setR^N)$ be fixed throughout the section.

\subsection{Laplace functionals}

Assume that $\partial \Sigma = \emptyset$.  For any $g \in \mathcal{R}^2(\Sigma)$,  the Laplace--Beltrami operator $\Delta_g$ admits a discrete spectrum whose $k$-th element, $k \in \setN^*$, is given by the min-max formula
\begin{align*}
\lambda_k(g) & = \min_{E \in \cG_k(H^{1}(\Sigma))} \max_{u \in \underline{E}} \frac{\int_\Sigma |du|^2_{g} \di v_{g}}{\int_\Sigma u^2 \di v_{g}} \, \cdot
\end{align*}
We define the scale-invariant quantity
 \[
\overline{\lambda}_k(g) \df \scaling(g)\lambda_k(g)
\]
where
\[
\scaling(g) \df v_g(M)^{2/n}
 \]
and we set
\[
d_k(g) \df
\begin{cases}
 \partial_kF \left(\overline\lambda_1(g),\cdots, \overline\lambda_N(g)\right) & \text{if $k\le N$,}\\
 \qquad \qquad \quad  0 & \text{otherwise.}
\end{cases}
\]
We also define
\[
c(g) \df
\begin{cases}
\qquad 0 & \text{if  $S(g) \df  \sum_{k=1}^{N} d_k(g) \lambda_k(g) =0$,}\\
S(g)/|S(g)| & \text{otherwise.}
 \end{cases}
\]
Lastly, we recall defnition \eqref{eq:def_quadric} of $\mathcal{Q}(\lambda,\eps,c)$ and we recall the multiplicity notation introduced in \ref{subsec:abstract_Rayleigh}: $\{\mu_i(g)\}$, $\aleph_i(g)$, $J_N(g)$ , $i_N(g)$, etc.  We also recall that  $\langle \cdot , \cdot \rangle_{g,h_\eps}$ and $\cM(\eps)$ are defined in \eqref{eq:prod} and \eqref{eq:mappings} respectively. Our first main theorem is the following.

\begin{theorem} \label{theocriticallaplace2} Let $\Sigma$ be a closed, connected, smooth surface.  

\begin{enumerate}
\item For a high enough integer $m$, if $g \in \mathcal{R}^m(\Sigma)$ is critical for the spectral functional
\[ \mathfrak{F} : \mathcal{R}^m(\Sigma) \ni g' \mapsto  F\left(\overline\lambda_1(g'),\cdots, \overline\lambda_N(g')\right),\]
then there exists $\eps \in \{0,\pm 1\}^{J_N(g)}$ such that the following holds.

\begin{enumerate} 
\item There exists a $\Delta_g$-harmonic mapping
\[
\Phi  = (\phi_k) : (\Sigma,g) \to \left(\mathcal{Q}\left(\Lambda(g),\eps, \frac{c(g)}{v_g(\Sigma)}\right),h_\eps \right)
\]
where $\Lambda(g):= (\lambda_1(g),\cdots,\lambda_{J_N(g)}(g))$ such that:
\begin{itemize}
\item[(i)] each $\phi_k$ belongs to $E_k(g)$,
\item[(ii)] the family $(\phi_k)$ is $\|\cdot\|_{L^2(\beta)}$-orthogonal,
\item[(iii)] the following identity holds:
\begin{equation}\label{eq:conformal}
\Phi^* h_{\eps} = \frac{\langle d \Phi, d \Phi \rangle_{g,h_{\eps}}}{2} \, g.
\end{equation}
\end{itemize}
In particular, if $\Phi \in \cM(\pm\eps)$,  then $\Phi$ is a minimal immersion.

\item There exists $\tilde{d} \in \prod_{i=1}^{i_N(g)} \Mix(d_{\aleph_i(g)})$ such that for any $1 \le k \le J_N(g)$, the $k$-th coordinate $\phi_k$ of $\Phi$ satisfies
\begin{equation}\label{eq:normLaplace}
\int_\Sigma \phi_k^2 \di v_g =\begin{cases}
\,\,\, |\tilde{d}_k| & \text{if $S(g)=0$,}\\
\displaystyle \frac{|\tilde{d}_k|}{|S(g)|} & \text{otherwise.}
 \end{cases}
\end{equation}
\end{enumerate}

\item Conversely, assume that there exist  $\eps = (\eps_j) \in \{0,\pm 1\}^N$ and $\Lambda = (\lambda_j) \in [0,+\infty]^N$ and a conformal minimal immersion
\[
\Phi : (M,g_0) \to (\cQ(\Lambda,\eps,1),h_\eps) 
\]
where $g_0$ is a Riemannian metric such that the function 
$$ f \df h_\eps\left( \Delta_{\Phi^*h_\eps} \Phi, \Lambda \Phi\right)  = \frac{\langle d \Phi, \Lambda d \Phi \rangle_{\Phi^*h_\eps, h_\eps} }{h_\eps(\Lambda \Phi, \Lambda \Phi)}$$ 
is well-defined and positive on $M$. Then the Riemannian metric
\[
g \df f \Phi^* h_\eps
\]
is critical for some finite combination of Laplace eigenvalues.
\end{enumerate}
\end{theorem}

\begin{rem} In (2), the assumption on $f$ is automatic if $\eps_j = 1 $ for any $j$.
\end{rem}

\begin{proof}
By classical elliptic regularity,  if $m$ is high enough, then any eigenfunction associated with a Laplace eigenvalue $\lambda_k(g)$ is $\mathcal{C}^2$. Choose any such a $m$ and set $X \df \cS^m(M)$, $\Omega \df \cR^m(M)$ , $Y \df L^2(\Sigma)$ and $H\df H^{1}(\Sigma)$.  For any $g' \in \mathcal{R}^m(\Sigma)$, define
\begin{itemize}
 \item $Q(g',u) \df \int_\Sigma u^2 \di v_{g'}$ for any $u \in Y$,
\item $G(g',u) \df \int_\Sigma \left\vert d u \right\vert_{g'}^2 \di v_{g'}$ for any $u \in H$,
\item $\mathcal{R}(g',u)\df G(g',u)/Q(g',u)$ for any $u \in \uY$.
\end{itemize}
Then $N(g',\cdot) \df \left( G(g',\cdot) + Q(g',\cdot)\right)^{1/2}$ coincides with the Hilbert norm $\|\cdot\|_{H^{1}(g')}$. The Rellich--Kondrachov theorem (see \cite[Corollary 3.7]{Hebey}, for instance) ensures that the embedding $(H, N(g',\cdot)) = (H^{1}(\Sigma),\|\cdot\|_{H^{1}(g')})  \hookrightarrow (L^2(\Sigma),\|\cdot\|_{L^2(g')}) = (Y,{Q^{1/2}(g',\cdot)) < +\infty}$ is compact.  Moreover,  by compactness of $M$, there exists a neighborhood $V \subset \mathcal{R}^m(\Sigma)$ of $g$ and $C\ge 1$ such that $C^{-1} g \le g' \le Cg$ for any $g' \in V$.  To sum up,  Assumptions (A), (B), (C'), (D') are satisfied.  (E) is also obviously true, since we can explicitely compute the Fréchet derivatives of $Q$, $G$ and $\scaling$: indeed, a classical calculation in local coordinates (see e.g.~\cite{Via}) yields that for any $h \in \cS^m(\Sigma)$, 
\[
\left. \frac{\di}{\di t} \right|_{t=0^+} \di v_{g+th} =  \frac{1}{2}\langle g,h \rangle_{g} \di v_{g}
\]
so that for any $u \in L^2(M)$ and $v \in H^1(M)$,
\begin{align*}
\langle Q_g(g,u), h \rangle & = \left. \frac{\di}{\di t} \right|_{t=0} Q(g+th,u)  = \left. \frac{\di}{\di t} \right|_{t=0}  \int_\Sigma u^2  \di v_{g+th}   =  \frac{1}{2}\int_\Sigma u^2  \langle g,h \rangle_g \di v_{g}\\
\langle G_g(g,v), h \rangle &  = \left. \frac{\di}{\di t} \right|_{t=0} G(g+th,v)  = \left.   \frac{\di}{\di t}  \right|_{t=0} \int_\Sigma  \left\vert d v \right\vert_{g+th}^2 \di v_{g'} + \left.   \frac{\di}{\di t}  \right|_{t=0} \int_\Sigma  \left\vert d v \right\vert_{g}^2    \di v_{g+th}   \\
&  = - \int_\Sigma \langle dv \otimes dv, h \rangle_g \di v_{g'} + \frac{1}{2}\int_\Sigma |dv|_g^2 \langle g,h \rangle_g \di v_{g}\\
\langle \scaling_g(g),h\rangle & = \left. \frac{\di}{\di t} \right|_{t=0} \scaling(g+th,u)   =  \frac{1}{2}\int_\Sigma \langle g,h \rangle_g \di v_{g}.
\end{align*}
This implies that
\[
Q_g(g,u) = \frac{1}{2} u^2 g \qquad G_g(g,v) = - dv \otimes dv + \frac{1}{2} |dv|_g^2 g \qquad \scaling_g(g,u) = \frac{1}{2}g.
\]
It is easily checked from these formulae that Assumptions (F) and (G) are satisfied.  Hence we can apply Theorem \ref{th:main_cri} which implies that there exist $\Phi \df (\phi_1,\ldots,\phi_{J_N(g)}) \in \cU_{J_N(g)}(g)$, $\tilde{d} \in \prod_{i=1}^{i_N(x)} \Mix(d_{\aleph_i(g)})$,  and $\eps = (\eps_k)\in\{0,\pm 1\}^{J_N(g)}$ such that \eqref{eq:normLaplace} holds and
\begin{equation}\label{eq:identityLaplace}
\sum_{k=1}^{J_N(g)} \eps_k \left(- d\phi_k \otimes d\phi_k + \frac{1}{2} |d \phi_k|_g^2 g \right) + \frac{c(g)}{2v_{g}(\Sigma) } g =  \sum_{i=1}^{i_N(g)} \mu_i(g) \sum_{k \in \aleph_i(g)} \frac{1}{2} \eps_k \phi_k^2 g.
\end{equation}
Since $\Phi \in \cU_{J_N(g)}(g)$ its coordinates are eigenfunctions $\|\cdot\|_{L^2(g)}$-orthogonal one to another. Moreover, taking the trace of \eqref{eq:identityLaplace}, we obtain
\begin{equation}\label{eq:trace}
\sum_{i=1}^{i_N(g)} \mu_i(g) \sum_{k \in \aleph_i(g)} \eps_k \phi_k^2 = \frac{c(g)}{v_{g}(\Sigma) } \, \cdot
\end{equation}
From Remark \ref{rem:obvious},  we obtain that $\Phi$ is a $\Delta_g$-harmonic mapping of $(\Sigma,g)$ into $(\cQ(g,\eps),h_\eps)$.  Inserting \eqref{eq:trace} back into \eqref{eq:identityLaplace},  we easily obtain \eqref{eq:conformal}.  The converse statement  is then obtained from a direct application of Theorem \ref{th:converse}, noticing that harmonicity implies by a direct computation of $0 = \frac{1}{2} \Delta_{\Phi^*h_\eps} h_\eps(\Lambda\Phi,\Phi)$ and the use of $\Delta_{\Phi^*h_\eps} \Phi \perp_{h_\eps} T_\Phi \mathcal{Q}$ that
$$ \Delta_{\Phi^*h_\eps} \Phi = f \Lambda \Phi $$
so that $ \Delta_g \Phi = \Lambda \Phi $. 
\end{proof}

Let us now treat the case of criticality in a conformal class $[g]$ where $g \in \mathcal{R}(\Sigma)$ is fixed. The conformal invariance of the Dirichlet energy on surfaces implies that for any $k \in \setN^*$ and $\tilde{g} = fg \in [g]$, $f \in \cC_{>0}^\infty(\Sigma)$,   the $k$-th eigenvalue of the Laplace--Beltrami operator $\Delta_{\tilde{g}}$ writes as 
\begin{align*}
\lambda_k(\tilde{g}) & = \min_{E \in \cG_k(H^{1}(\Sigma))} \max_{u \in \underline{E}} \frac{\int_\Sigma |du|^2_{g} \di v_{g}}{\int_\Sigma u^2 f \di v_{g}}  \, \cdot
\end{align*}
We extend this formula to conformal factors $f$ that may not be smooth by setting
\[
\lambda_k(\beta) \df \min_{E \in \cG_k(H^{1}(M))} \max_{u \in \underline{E}} \frac{\int_\Sigma |du|^2_{g} \di v_{g}}{\int_\Sigma u^2 \beta \di v_{g}}
\]
for any $\beta \in  \cC_{>0}(\Sigma)$.  Any eigenfunction $\varphi \in H^1(\Sigma)$ associated with $\lambda_k(\beta)$ satisfies
\begin{equation}\label{eq:Laplacebeta}
\Delta_g \varphi = \lambda_k(\beta) \varphi \beta
\end{equation}
weakly on $\Sigma$. Classical elliptic regularity theory ensures that if $\beta \in \cC^m(\Sigma)$ for some $m \in \setN$, then $\varphi \in \cC^{m+2}(\Sigma)$. 
We define the scale-invariant quantities
\[
\overline{\lambda}_k(\beta)  \df \scaling(\beta)\lambda_k(\beta), \qquad k \in \setN^*,
\]
where
\[
\scaling(\beta) \df \|\beta\|_{L^1(g)}.
\]
We adapt the notation in a natural way: $\{\mu_i(\beta)\}$, $\aleph_i(\beta)$, $J_N(\beta)$ , $i_N(\beta)$,  $d_k(\beta)$, $S(\beta)$, $c(\beta)$, etc.  We recall definition \eqref{eq:def_quadric} of $\mathcal{Q}(\lambda,\eps,c)$.  We also set
\[
\|u\|_{L^2(\beta)} \df \left( \int_\Sigma u^2 \beta\di v_g \right)^{1/2} \qquad \forall \, u \in L^2(\Sigma).
\]
Note that this defines a family of locally equivalent norms $\{\|\cdot\|_{L^2(\beta)}\}_{\beta \in  \cC_{>0}(\Sigma)}$ on $L^2(\Sigma)$.

\begin{theorem}\label{theoLaplace2conforme}
Let $\Sigma$ be a closed, connected, smooth surface. 
\begin{enumerate}
\item For a high enough integer $m$, if $\beta \in \cC_{>0}^m(\Sigma)$ is critical for the eigenvalue functional
\[
\fk : \cC_{>0}^m(\Sigma) \ni  \beta' \mapsto F(\overline{\lambda}_1 (\beta'),\ldots,\overline{\lambda}_N (\beta')),
\]
 then there exists $\eps \in \{0,\pm 1\}^{J_N(\beta)}$ such that the following holds.
 \begin{enumerate}
\item There exists a $\Delta_{g}$-harmonic mapping
$$\Phi = (\phi_k) : (\Sigma,g) \to \left(\cQ\left(\Lambda(\beta),\eps,\frac{c(\beta)}{\Vert \beta \Vert_{L^1}}\right),h_\eps\right)$$
where $\Lambda(\beta) = (\lambda_1(\beta),\cdots,\lambda_{J_N(\beta)}(\beta))$ such that
\begin{itemize}
\item[(i)] each $\phi_k$ belongs to $E_k(g)$,
\item[(ii)] the family $(\phi_k)$ is $\|\cdot\|_{L^2(\beta)}$-orthogonal,
\item[(iii)] the following identity holds: \begin{equation}\label{eq:beta*}
\beta = h_\eps(\Phi,\Delta_{g_0}\Phi) = \frac{\langle d\Phi, \Lambda(\beta) d \Phi \rangle_{g,h_\eps}}{h_\eps(\Lambda(\beta)\Phi,\Lambda(\beta)\Phi)} \, \cdot
\end{equation}
\end{itemize}

\item  There exists $\tilde{d} \in \prod_{i=1}^{i_N(\beta)} \Mix(d_{\aleph_i(\beta)})$ such that for any $1 \le k \le J_N(\beta)$, the $k$-th coordinate $\phi_k$ of $\Phi$ satisfies
\begin{equation}\label{eq:normLaplace2}
\int_\Sigma \phi_k^2 \di v_g =\begin{cases}
\,\,\, |\tilde{d}_k| & \text{if $S(\beta)=0$,}\\
\displaystyle \frac{|\tilde{d}_k|}{|S(\beta)|} & \text{otherwise.}
 \end{cases}
 \end{equation}

\end{enumerate}

\item Conversely,  the statement (2) of Theorem \ref{theocriticallaplace2}, holds replacing the assumption "conformal minimal immersion" by "harmonic map" and criticality of $\beta := h_\eps( \Phi,\Delta_{g_0}\Phi)$ holds for a combination of Laplace eigenvalues in a conformal class, as soon as it is positive.

\end{enumerate}

\end{theorem}

\begin{rem} In (2), if $\eps_j = 1 $ for any $j$, we obtain that $\beta$ is non-negative.
\end{rem}

\begin{proof}
By elliptic regularity theory,  if $m$ is high enough, then any eigenfunction associated with a Laplace eigenvalue $\lambda_k(\beta)$ is $\mathcal{C}^2$.  Choose one such a $m$ and set $X\df \cC^m(\Sigma)$, $\Omega\df\cC_{>0}^m(\Sigma)$, $Y \df L^2(\Sigma)$ and $H \df H^1(\Sigma)$.  For any $\beta' \in \Omega$, $u \in Y$ and $v \in H$, define
\[
Q(\beta',u) \df \int_\Sigma u^2 \beta' \di v_{g} \qquad G(v) \df \int_\Sigma |dv|^2_{g} \di v_{g} \qquad  \cR(\beta',v) \df \frac{G(v)}{Q(\beta',u)} \, \cdot
\]
Assumption (A) trivially holds. Building upon the local equivalence of the norms $\|\cdot\|_{L^2(\beta')}$ on  $Y$, we can adapt the arguments from the proof of the previous theorem to show that Assumptions (B), (C'), (D') are also satisfied. Obviously $Q$, $G$ and $\scaling$ are Fréchet differentiable with respect to the $\beta'$-variable, and a direct computation shows that
\[
Q_{\beta'}(\beta',u) = u^2 v_{g} \qquad G_{\beta'}(v) =  0 \qquad \scaling_{\beta'}(\beta') = \frac{1}{2}v_g.
\]
Then Assumptions (E), (F), (G) are easily checked and Theorem \ref{th:main_cri} applies: this brings the existence of $\Phi \df (\phi_k) \in \cU_{J_N(\beta)}(\beta)$ and $(\eps_k)\in\{0,\pm 1\}^{J_N(\beta)}$ such that
\begin{equation}\label{eq:identityLaplaceconformal}
\sum_{i=1}^{i_N(\beta)} \mu_i(\beta) \sum_{k \in \aleph_i(\beta)} \eps_k \phi_k^2 = \frac{c(\beta)}{\|\beta\|_{L^1(g)}} \, \cdot
\end{equation}
Then \textit{(i)} and \textit{(ii)} are satisfied, and the fact that $\Phi$ is $\Delta$-harmonic follows from Remark \ref{rem:obvious}.  Moreover, \eqref{eq:Laplacebeta}  gives that
\[
\Delta_g \phi_k = \lambda_k(\beta) \phi_k \beta
\]
for any $k \in \{1,\ldots,J_N(\beta)\}$. Multiplying the latter by $\phi_k$ and summing over $k$, we obtain \eqref{eq:Laplacebeta}.  Lastly, the converse statement \textit{(2)} is a direct consequence of Theorem \ref{th:converse}.
\end{proof}

\subsection{Steklov functionals}
Assume now that $\partial \Sigma \neq \emptyset$.  Recall that the trace operator on $\Sigma$ is a compact, surjective, bounded linear map $T : H^1(\Sigma) \to H^{1/2}(\partial \Sigma)$ which 
extends the restriction operator $\left. \cdot \right|_{\partial \Sigma} : \cC^\infty(\Sigma) \to \cC^\infty(\partial \Sigma)$. It defines an isomorphism between the quotient $H^1(\Sigma)/H_0^1(\Sigma)$ and $H^{1/2}(\partial \Sigma)$.  The harmonic extension operator is the bounded linear map $\widehat{\cdot} : H^{1/2}(\partial \Sigma) \to H^1(\Sigma)$ such that for any $u \in H^{1/2}(\partial \Sigma)$,  $\widehat{u}$ is the unique element of $H^1(\Sigma)$ such that
\[
\begin{cases}
\Delta_g \widehat{u} = 0 & \text{weakly on $\Sigma$,}\\
T(\widehat{u}) = u & \text{a.e.~on $\partial \Sigma$,}
\end{cases}
\]
for any $g \in \mathcal{R}(\Sigma)$. Note that $\widehat{u}$ does not depend on $g$ since $\Sigma$ is compact. 

For any $g \in \mathcal{R}^2(\Sigma)$, the associated Dirichlet-to-Neumann operator $\mathcal{D}_g$ maps $u \in \cC^\infty(\partial \Sigma)$ to the normal derivative $\partial_g^\nu \widehat{u} \in \cC^\infty(\partial \Sigma)$. This operator admits a discrete spectrum whose $k$-th element, $k \in \setN^*$, is given by the well-known min-max formula
\begin{align*}
\sigma_k(g) & = \min_{E \in \cG_k(H^{1/2}(\partial \Sigma))} \max_{u \in \underline{E}} \frac{\int_{\Sigma} |d\widehat{u}|^2_{g} \di v_{g}}{\int_{\partial \Sigma} u^2 \di v_{\partial g}}
\end{align*}
where $\partial g$ is the smooth Riemannian metric induced by $g$ on $\partial \Sigma$. The number $\sigma_k(g)$ is called the $k$-th Steklov eigenvalue of $(\Sigma,g)$. 

To let our theory take the boundary into full consideration, we work with the following modified Steklov eigenvalues: for any $(g,\tilde{g}) \in \mathcal{R}^2(\Sigma) \times \mathcal{R}^2(\partial \Sigma)$ and $k \in \setN^*$, we set
\begin{align*}
\sigma_k(g,\tilde{g}) & \df \min_{E \in \mathcal{G}_{k}( H^{1/2}(\partial\Sigma))} \max_{u \in \underline{E} }  \frac{\int_\Sigma \left\vert d \widehat{u} \right\vert_g^2 \di v_g}{\int_{\partial\Sigma} u^2  \di v_{\tilde{g}}} \,\cdot
 \end{align*}
Any associated eigenfunction $\varphi \in H^{1/2}(\partial \Sigma)$ satisfies
\[
\int_{\partial \Sigma}\phi \,\partial_g^\nu \widehat{\varphi}  \di v_{\partial g} = \sigma_k(g,\tilde{g}) \int_{\partial \Sigma} \phi \varphi \di v_{\tilde{g}}
\] 
for any $\phi \in \cC^{\infty}(\partial \Sigma)$. When $\tilde{g} = \beta \partial g$ for some $\beta \in \cC^m(\partial \Sigma)$ with high enough $m \in \setN$, elliptic regularity theory ensures that any associated eigenfunction belongs to $\cC^{2}(\partial \Sigma)$. We define the scale-invariant quantities
 \[
\overline{\sigma}_k(g,\tilde{g}) \df \scaling(g,\tilde{g})\sigma_k(g,\tilde{g}), \qquad k \in \setN^*,
\]
where 
\[
\scaling(g,\tilde{g}) = v_{\tilde{g}}(\partial \Sigma).
\]
We adapt the notation in a natural way: $\{\mu_i(g,\tilde{g})\}$, $\aleph_i(g,\tilde{g})$, $J_N(g,\tilde{g})$ , $i_N(g,\tilde{g})$,  $d_k(g,\tilde{g})$, $c(g,\tilde{g})$, and so on.  We also recall the notation $\mathcal{G}(\lambda,\eps,c)$ of \eqref{eq:def_quadric}

For any $\eps \in \{0,\pm 1\}^{J_N(g,\tilde{g})}$, we let $\mathcal{Q}(g,\tilde{g},\eps)$ be the quadric with parameters $(\sigma_1(g,\tilde{g}),\ldots,\sigma_{J_N(g,\tilde{g})}(g,\tilde{g}))$,  $\eps$ and $v_{\tilde{g}}(\partial \Sigma)$.

\begin{theorem} \label{theocriticalsteklov2} Let $\Sigma$ be a compact, connected, smooth surface with a non-empty boundary. 
\begin{enumerate}
\item For a high enough integer $m $, if $g \in \cR^m(\Sigma)$ is such that $(g,\partial g)$ is critical for
\[
\mathfrak{F}  : \mathcal{R}^m(\Sigma) \times \mathcal{R}^m(\partial \Sigma) \ni (g',\tilde{g}') \mapsto F\left(\overline\sigma_1(g',\tilde{g}'),\cdots, \overline\sigma_N(g',\tilde{g}')\right)\]
then there exists $\eps \in \{0,\pm 1\}^{J_N(g,\partial g)}$ such that the following hold. 
\begin{enumerate}
\item There exists a free boundary $\Delta_g$-harmonic mapping
\[
\Phi : (\Sigma,\partial \Sigma,g) \to \left(\setR^{J_{N}(g,\partial g)},\mathcal{Q}\left(\sigma(g,\partial g),\eps, \frac{c(g,\partial g)}{v_{\partial g}(\partial \Sigma)}\right),h_\eps \right)
\]
whose coordinates $\varphi_k$ restricted to $\partial \Sigma$ are $\|\cdot\|_{L^2(\partial g)}$-orthogonal one to another and such that 
\begin{equation}\label{eq:conformalSteklov}
\Phi^* h_{\eps} = \frac{\langle d \Phi, d \Phi \rangle_{g,h_{\eps}}}{2} \, g.
\end{equation}
In particular, if $\Phi \in \cM(\pm\eps)$,  then $\Phi$ is a free boundary minimal immersion.
\item There exists $\tilde{d} \in \prod_{i=1}^{i_N(g,\partial g)} \Mix(d_{\aleph_i(g,\partial g)})$ such that for any $1 \le k \le J_N(g,\partial g)$, the $k$-th coordinate $\phi_k$ of $\Phi$ satisfies
\begin{equation}\label{eq:normSteklov}
\int_{\partial \Sigma} \phi_k^2 \di v_{\partial g} =\begin{cases}
\,\,\, |\tilde{d}_k| & \text{if $S(g,\partial g)=0$,}\\
\displaystyle \frac{|\tilde{d}_k|}{|S(g,\partial g)|} & \text{otherwise.}
 \end{cases}
\end{equation}
\end{enumerate}
\item Conversely,  assume that there exist $\eps = (\eps_j)\in \{0,\pm 1\}^N$ and $\sigma:= (\sigma_j) \in [0,+\infty]^N$ and a free boundary conformal minimal immersion
\[
\Phi : (\Sigma,\partial \Sigma,g_0) \to (\R^N,\cQ(\sigma,\eps,1),h_\eps)
\]
where $g_0$ is a Riemannian metric such that the function 
$$f \df h_\eps\left( \Phi, \partial_{\nu_{\Phi^*h_\eps}} \Phi \right) $$ 
is positive on $\partial M$. Then for any extension $\tilde{f}$ of $f$ on $M$, the Riemannian metric
\[
g \df \tilde{f} \Phi^* h_\eps \text{ and } \partial g \df f \Phi^* h_\eps 
\]
is critical for some finite combination of Steklov eigenvalues.

\end{enumerate}
\end{theorem}

\begin{rem} \label{rem:stek} In \textit{(2)}, the assumption on $f$ holds if  $\eps_j = 1$ for any $j$ since $\mathcal{Q}$ is convex is that case. Indeed the harmonic function $\psi_y : x \mapsto h_\eps\left(\Phi(x),\Lambda \Phi(y)\right)$ realizes its maximum at $y$ for any $y\in \partial M$, and the Hopf lemma implies that $\partial_\nu \psi(y) >0$.
\end{rem}

The proof is similar to the one of Theorem \ref{theocriticallaplace2}, but we provide details to emphasize the differences between the two contexts.

\begin{proof}
Choose $m$ such that any eigenfunction associated to any $\sigma_k(g,\tilde{g})$ is $\mathcal{C}^2$. Define  $X \df \cS^m(\Sigma) \times \cS^m(\partial \Sigma)$, $\Omega  \df \mathcal{R}^m(\Sigma) \times \mathcal{R}^m(\partial \Sigma)$, $Y \df  L^{2}(\partial \Sigma)$ and $ H \df  H^{1/2}(\partial \Sigma)$. For any $(g',\tilde{g}') \in \Omega$, we consider the Hilbert norm on $H$ given by
\[
\|v\|_{H(g',\tilde{g}')} \df \left( \int_{\partial \Sigma} v^2  \di v_{\tilde{g}'} + \int_\Sigma |d \widehat{v}|^2_{g'} \di v_{g'} \right)^{1/2}
\]
for any $v \in H$, and we set:
\begin{itemize}
\item $Q(g',\tilde{g}',u) \df \int_{\partial\Sigma} u^2  \di v_{\tilde{g}'}$ for any $u \in Y$,
\item $G(g',\tilde{g}',v) \df \int_{\Sigma} |d\widehat{v}|_{g'}^2  \di v_{g'}$ for any $v \in H$,
\item $\cR(g',\tilde{g}',v) \df G(g',\tilde{g}',v)/Q(g',\tilde{g}',v)$ for any $v \in H$.
\end{itemize}
Then $N(g',\tilde{g}',\cdot) \df (G(g',\tilde{g}',\cdot) + Q(g',\tilde{g}',\cdot))^{1/2}$ coincides with $\| \cdot \|_{H(g',\tilde{g}')}$.  The fact that $(H,\|\cdot\|_{H(g',\tilde{g}')}) \hookrightarrow  (Y,Q^{1/2}(g',\tilde{g}', \cdot))$ is a compact embedding follows from the compactness of $T$ and the boundedness of the harmonic extension operator.  The compactness of $M$ also ensures that there exists a neighborhood $V \subset \Omega$ of $(g,\partial g)$ and $C\ge1$ such that $C^{-1} N(g,\partial g,\cdot) \le N(g',\tilde{g}',\cdot) \le CN(g,\partial g,\cdot)$ for any $(g',\tilde{g}') \in V$.  Therefore, Assumptions (A), (B), (C'), (D') are satisfied. Assumption (E) holds too as we can explicitely compute the Fréchet derivatives of $G$, $Q$ and $\scaling$:  
acting as in the proof of Theorem \ref{theocriticallaplace2}, we easily obtain that for any $u \in Y$ and $v \in H$,
 \begin{align*}
 Q_{(g,\partial g)}(g,\partial g,u) & = \left( 0 ,  \frac{1}{2}u^2 \partial g \right), \\
G_{(g,\partial g)}(g,\partial g,v) & = \left( - d \widehat{v} \otimes d \widehat{v} + \frac{1}{2} |d\widehat{v}|^2_g\,  g\, , \,  0 \right)\\
\scaling_{(g,\partial g)}(g,\partial g) & =   \left( 0, \frac{1}{2}\partial g\right).
\end{align*}
Assumptions (G) and (F) may be directly checked from these formulae. Hence we can apply Theorem \ref{th:main_cri} to get existence of $\Psi \df (\psi_k) \in \cU_{J_N(g,\partial g)}(g,\partial g)$,  $\tilde{d} \in \prod_{i=1}^{i_N(g,\partial g)} \Mix(d_{\aleph_i(g,\partial g)})$ and $\eps = (\eps_k) \in \{0,\pm 1\}^{J_N(g,\partial g)}$ such that \eqref{eq:normSteklov} holds and
\begin{equation}\label{eq:preuveSteklov1}
\sum_{k=1}^{J_N(g,\partial g)} \eps_k \, d \widehat{\psi}_k \otimes d \widehat{\psi}_k    = \frac{1}{2}  \left( \sum_{k=1}^{J_N(g,\partial g)} \eps_k \, |d \widehat{\psi}_k|^2 \right) g \qquad \qquad \text{on $\Sigma$,} 
\end{equation}
\begin{equation}\label{eq:preuveSteklov2}
\sum_{i=1}^{i_N(g,\partial g)} \mu_i(g,\partial g) \sum_{k \in \aleph_i(g,\partial g)} \eps_k \psi_k^2  = \frac{c(g,\partial g)}{v_{\partial g}(\partial \Sigma)} \qquad\quad\qquad \qquad  \text{on $\partial \Sigma$.}
\end{equation}
Set $\phi_k \df \widehat{\psi}_k$ for any $k \in \{1,\ldots,J_N(g,\partial g)\}$. Then \eqref{eq:preuveSteklov1} exactly means that the mapping $\Phi \df (\phi_k)_{1 \le k \le J_N(g,\partial g)}$ satisfies \eqref{eq:conformalSteklov}.  Moreover,  we obtain that $\Phi$ is free boundary $\Delta_g$-harmonic by observing that each $\phi_k$ is harmonic on $\Sigma$ and by taking the normal derivative of \eqref{eq:preuveSteklov2}. The converse statement is, again, a direct consequence of Theorem \ref{th:converse}.
\end{proof}

Let us now consider criticality in a conformal class $[g]$ where $g \in \mathcal{R}(\Sigma)$ is some reference Riemannian metric. The conformal invariance of the Dirichlet energy on surfaces implies that for any $\tilde{g} \in [g]$ and $k \in \setN^*$,  the $k$-th Steklov eigenvalue of $\tilde{g}$ writes as 
\begin{align*}
\sigma_k(\tilde{g}) & = \min_{E \in \cG_k(H^{1/2}(\partial M))} \max_{u \in \underline{E}} \frac{\int_{\Sigma} |d\widehat{u}|^2_{g} \di v_{g}}{\int_{\partial \Sigma} u^2 \di v_{\partial \tilde{g}}} \,\cdot
\end{align*}
We consider a variant of this formula by setting
\begin{align*}
\sigma_k(\beta) & \df \min_{E \in \cG_k(H^{1/2}(\partial M))} \max_{u \in \underline{E}} \frac{\int_{\Sigma} |d\widehat{u}|^2_{g} \di v_{g}}{\int_{\partial \Sigma} u^2 \beta \di v_{\partial g}}
\end{align*}
for any $\beta \in \cC_{>0}(\partial \Sigma)$ and we define
\[
\overline{\sigma}_k(\beta) \df \scaling(\beta)\sigma_k(\beta) \qquad \text{with } \scaling(\beta) \df \int_{\partial \Sigma} \beta \di v_{\partial g}.
\]
Consider the associated notations: $\{\mu_i(\beta)\}$, $\aleph_i(\beta)$, $J_N(\beta)$ , $i_N(\beta)$,  $d_k(\beta)$, $c(\beta)$, and so on.  For any $\eps \in \{0,\pm 1\}^{J_N(\beta)}$, let $\mathcal{Q}(\beta,\eps)$ be the quadric with parameters $(\sigma_1(\beta),\ldots,\sigma_{J_N(f)}(\beta))$,  $\eps$ and $c(\beta)/\|\beta\|_{L^1(\partial g)}$. 

The following theorem follows from a suitable adaptation of the previous proof where we freeze the $g$ variable and replaces $\tilde{g}$ by $\beta g$ in the quantities $Q(g,\tilde{g},u)$, $G(g,\tilde{g},v)$ and $\scaling(g,\tilde{g})$.   For the sake of brevity, the statement is shortened ans we do not include the proof.

\begin{theorem}\label{theoSteklov2conforme}
Let $(\Sigma,g)$ be a compact, connected, smooth Riemannian surface with a non-empty boundary. 
\begin{enumerate}
\item For a high enough integer $m$, if $\beta \in \cC_{>0}^m(\partial \Sigma)$ is critical for the spectral functional
\[
\fk : \cC_{>0}^m(\partial \Sigma) \ni  \beta' \mapsto F(\overline{\sigma}_1 (\beta'),\ldots,\overline{\sigma}_N (\beta'))
\]
and such that $c(\beta)\neq 0$, then there exists a free boundary $\Delta_{g}$-harmonic mapping $\Phi : (\Sigma,\partial \Sigma, g) \to (\setR^N,\cQ(\beta,\eps),h_\eps)$ whose coordinates $\varphi_k$ in restriction to $\partial \Sigma$ are $\|\cdot\|_{L^2(\partial g)}$-orthogonal one to another,  for some $\eps \in \{0,\pm 1\}^{J_N(\beta)}$. 

\item Conversely, the statement (2) of theorem \ref{theocriticalsteklov2}, holds replacing the assumption "free boundary conformal minimal immersion" by "free boundary harmonic map" and criticality of $\beta := h_\eps( \Phi,\partial_{\nu_{g_0}}\Phi)$ holds for a combination of Steklov eigenvalues in a conformal class, as soon as it is positive.
\end{enumerate}
\end{theorem}

\begin{rem} As noticed in Remark \eqref{rem:stek}, in \textit{(2)}, the assumption on $f$ holds if $\eps_j = 1 $ for any $j$.
\end{rem}

\section{Laplace and Steklov functionals in higher dimensions}   \label{sec:high}

In this section, we deal with functionals involving either the Laplace or Steklov eigenvalues of a compact, connected, smooth manifold $M$ of dimension $n \ge 3$.  Following \cite{KarpukhinMetras}, we work with suitable weighted versions of these eigenvalues. The main reason for that is that in dimension $n\ge 3$, the Dirichlet energy is no longer invariant under conformal change. Like in the previous section, we let a positive integer $N$ and a function $F \in \cC^1(\setR^N)$ be fixed throughout.

\subsection{Laplace functionals}\label{sec:LaplaceHigher} Assume that $\partial M = \emptyset$.  We begin with considering criticality in a conformal class $[g]$ where  $g \in \mathcal{R}(M)$ is some fixed reference Riemannian metric.  For any $(\alpha,\beta) \in \cC_{>0}(M)^2$ and $k \in \setN^*$,  set
\begin{align*}
\lambda_k(\alpha,\beta) & \df \inf_{E \in \mathcal{G}_{k}\left( H^1(M)\right)} \sup_{u \in \uE }  \frac{\int_M \left\vert d u \right\vert_g^2 \alpha \di v_g}{\int_M u^2 \beta \di v_g}\\
 \scaling(\alpha,\beta) & \df  \frac{\int_M \beta \di v_g }{\left(\int_M \alpha^{\frac{n}{n-2}}\di v_g\right)^{\frac{n-2}{n}}} \\
\overline{\lambda}_k(\alpha,\beta) & \df \lambda_k(\alpha,\beta)\,  \scaling(\alpha,\beta) .
\end{align*}
Any eigenvalue $\phi_k \in H^1(M)$ associated with $\lambda_k(\alpha,\beta)$ satisfies the weak equation
\begin{equation}\label{eq:EL_alpha_beta}
-\delta_g \left( \alpha d \phi_k \right)  = \lambda_k(\alpha,\beta) \phi_k \beta \qquad \text{on $M$.}
\end{equation}
Elliptic regularity theory ensures that if $\beta \in \cC^m(M)$ for some  high enough $m \in \setN$, then $\varphi \in \cC^{2}(M)$. Recall the notations adapted from the previous sections: $\{\mu_i(\alpha,\beta)\}$, $\aleph_i(\alpha,\beta)$, $J_N(\alpha,\beta)$ , $i_N(\alpha,\beta)$,  $d_k(\alpha,\beta)$, $c(\alpha,\beta)$, and so on. We also recall Definition \ref{eq:def_quadric} for $\mathcal{Q}(\lambda,\eps,c)$

\begin{theorem} \label{theocriticallaplacen} Let $(M,g)$ be a closed, connected, smooth Riemannian manifold of dimension $n \ge 3$.  
\begin{enumerate}
\item For a high enough integer $m$, if $(\alpha,\beta)\in \cC_{>0}^m(M)^2$ is critical for the spectral functional
\[ \mathfrak{F} : \cC_{>0}^m(M)^2 \ni (\alpha',\beta') \mapsto  F\left(\overline\lambda_1(\alpha',\beta') ,\cdots, \overline\lambda_N(\alpha',\beta') \right),\]
then there exists $\eps \in \{0,\pm 1\}^{J_N(\alpha,\beta)}$ such that the following holds. 

\begin{enumerate}
\item There exists a $\cC^2$ mapping $\Phi : M \to \R^{J_N(\alpha,\beta)}$ with $\|\cdot\|_{L^2(\beta)}$-orthogonal coordinates such that if $c(\alpha,\beta) \neq 0$, then 
\[
\Phi : (M,g) \to \left(\cQ\left(\Lambda(\alpha,\beta),\eps,\frac{c(\alpha,\beta)}{\| \beta \|_{L^1(g)}}\right),h_{\eps}\right)
\]
is $n$-harmonic 
and such that
\begin{equation}\label{eq:alpha}
\frac{\alpha}{\|\alpha\|^{\frac{2}{n}}_{L^{\frac{n}{n-2}}(g)}} = \langle d\Phi,d\Phi \rangle_{g,h_{c(\alpha,\beta)\eps}}^{\frac{n-2}{2}}
\end{equation}
\begin{equation}\label{eq:beta}
 \frac{\beta}{\|\alpha\|^{\frac{2}{n}}_{L^{\frac{n}{n-2}}(g)}} = \langle d\Phi,d\Phi \rangle_{g,h_{c(\alpha,\beta)\eps}}^{\frac{n-2}{2}} \frac{ \langle d\Phi,d\Lambda(\alpha,\beta)\Phi \rangle_{g,h_{c(\alpha,\beta)\eps}} }{\langle \Lambda(\alpha,\beta)\Phi,\Lambda(\alpha,\beta) \Phi \rangle_{h_{c(\alpha,\beta)\eps}} }
\end{equation}
while if $c(\alpha,\beta)=0$, then
\begin{equation}\label{eq:c=0}
\langle d\Phi,d\Phi \rangle_{g,h_{\eps}} = 0.
\end{equation}

\item There exists $\tilde{d} \in \prod_{i=1}^{i_N(x)} \Mix(d_{\aleph_i(\alpha,\beta)})$ such that for any $1 \le k \le J_N(\alpha,\beta)$, the $k$-th coordinate $\phi_k$ of $\Phi$ satisfies
\begin{equation}\label{eq:normLaplacehigher}
\int_\Sigma \phi_k^2 \di v_g = \begin{cases} \,\,  \, |\tilde{d}_k| & \text{if $S(\alpha,\beta)=0$},\\
\displaystyle\frac{|\tilde{d}_k|}{\left| S(\alpha,\beta) \right|}  & \text{otherwise}. \, 
\end{cases}
\end{equation}
\end{enumerate}

\item Conversely, for $\eps \df (\eps_j) \in \{0,\pm 1\}^N$ and $\Lambda\df (\lambda_j) \in [0,+\infty]^N$ any $n$-harmonic mapping 
$$ \Phi : (M,g) \to \left(\mathcal{Q}(\Lambda,\eps,1),h_\eps\right) $$
such that the functions
$$ \alpha \df \langle d\Phi,d\Phi \rangle_{g,h_{\eps}}^{\frac{n-2}{2}} $$
and 
$$ \beta \df \langle \Phi, -\delta_g\left( \alpha d\Phi \right) \rangle_{h_\eps}  = \alpha  \frac{ \langle d\Phi,d\Lambda(\alpha,\beta)\Phi \rangle_{g,h_{\eps}} }{\langle \Lambda(\alpha,\beta)\Phi,\Lambda(\alpha,\beta) \Phi \rangle_{h_{\eps}} }  $$
are well defined and positive, then $(\alpha,\beta)$ is a critical point of a combination of Laplace eigenvalues.
\end{enumerate}

\end{theorem}

\begin{rem}\label{rem:steklovn}
If we set $g_0 \df \langle d \Phi,d\Phi \rangle_{g,h_{c(\alpha,\beta)\eps}} g $, then the mapping $\Phi : (M,g_0) \to (\cQ,h_{\eps})$ is $p$-harmonic for any $p\geq 2$.
\end{rem}

\begin{proof}
Unless specified, all the integrals considered in this proof are implicitely taken over $M$ with respect to the volume measure $v_g$.  We set $X \df \cC^m(M)$ for $m$ high enough. For any $(\alpha',\beta')\in \Omega \df  \cC^m_{>0}(M)$,  $u \in Y\df L^2(M)$ and $v \in H\df H^1(M)$,  set
\[
Q(\alpha',\beta',u) \df  \int u^2 \beta', \qquad G(\alpha',\beta',v) \df \int \left\vert d v \right\vert_g^2 \alpha', 
\]
\[
\left\| v \right\|_{H\left(\alpha',\beta'\right)}^2 \df \int v^2 \beta' + \int \left\vert dv\right\vert_g^2 \alpha'
\]
and note that $N(\alpha',\beta',\cdot) \df \left( G(\alpha',\beta',\cdot) + Q(\alpha',\beta',\cdot)\right)^{1/2}$ coincides with the Hilbert norm $\|\cdot\|_{H(\alpha',\beta')}$.  A direct computation yields that
 \[
 Q_{(\alpha',\beta')}(\alpha',\beta',u) = \big( u^2 v_g, 0 \big), \qquad G_{(\alpha',\beta')}(\alpha',\beta',u) = (0,|du|^2_g v_g)
 \]
 \[
 \scaling_{(\alpha',\beta')}(\alpha',\beta') =  \scaling(\alpha',\beta') \left(  \frac{1}{\int \beta' }  \, v_g \,  , \,   - \frac{(\alpha')^{\frac{2}{n-2}} }{\int (\alpha')^{\frac{n}{n-2}}} \, v_g \right)
 \]
It is easily checked all the assumptions (A)--(G) are satisfied. Then Theorem \ref{th:main_cri} implies that there exist $\Phi \df (\phi_k) \in \cU_{J_N(\alpha,\beta)}(\alpha,\beta)$, $\tilde{d} \in \prod_{i=1}^{i_N(\alpha,\beta)} \Mix(d_{\aleph_i(\alpha,\beta)})$ and $\eps = (\eps_k) \in \{0,\pm 1\}^{J_N(\alpha,\beta)}$ such that \eqref{eq:normLaplacehigher} holds and
\begin{equation}\label{eq:identity1}
\sum_{i=1}^{i_N(\alpha,\beta)} \mu_i(\alpha,\beta) \sum_{k \in \aleph_i(\alpha,\beta)} \eps_k \phi_k^2 = \frac{c(\alpha,\beta)}{\int \beta}
\end{equation}
\begin{equation}\label{eq:identity2}
\sum_{k=1}^{J_N(\alpha,\beta)} \eps_k |d \phi_k|_g^2  =   \frac{\alpha^{\frac{2}{n-2}}}{\int \alpha^{\frac{n}{n-2}}} \, c(\alpha,\beta) \, \cdot 
\end{equation}
If $c(\alpha,\beta)=0$, then  \eqref{eq:identity2} rewrites exactly as \eqref{eq:c=0}.  From now on, assume that $c(\alpha,\beta)\neq 0$. Divide \eqref{eq:identity2} by $c(\alpha,\beta)$ to obtain \eqref{eq:alpha}. After that, replace $\alpha$ in \eqref{eq:EL_alpha_beta} by its value given by \eqref{eq:alpha} to obtain
\[
- \delta_g (\langle d\Phi,d\Phi \rangle_{g,h_{c(\alpha,\beta)\eps}}^{\frac{n-2}{2}}  d \Phi)  \paral \Lambda \Phi
\]
which is the equation for n-harmonic maps. Using \eqref{eq:identity1} ends the proof of (1). The converse is an immediate consequence of Theorem \ref{th:converse}. 
\end{proof}

\subsection{Steklov functionals} Assume now that $\partial M \neq \emptyset$ and consider a reference Riemannian metric $g \in \mathcal{R}(M)$. We begin with studying criticality in the conformal class $[g]$.  For any $(\alpha,\beta) \in  \mathcal{C}(M) \times \mathcal{C}(\partial M)$ and $k \in \setN^*$, set
\begin{align*}
 \sigma_k(\alpha,\beta)  & \df \inf_{E \in \mathcal{G}_k\left(H^{1/2}(\partial M) \right)} \sup_{u \in \uE}  \frac{\int_M \left\vert d \widehat{u} \right\vert_g^2 \alpha \di v_g}{\int_{\partial M} u^2 \beta \di v_{\partial g}}\, ,\\
\scaling(\alpha,\beta) & \df  \frac{\int_{\partial M} \beta \di v_{\partial g} }{\left(\int_{M} \alpha^{\frac{n}{n-2}} \di v_g\right)^{\frac{n-2}{n}}} \\
\overline{\sigma}_k(\alpha,\beta) & \df \sigma_k(\alpha,\beta)\scaling(\alpha,\beta).
\end{align*}
Any eigenfunction $\psi \in H^{1/2}(\partial M)$ associated with $\sigma_k(\alpha,\beta)$ satisfies the weak equation
\begin{equation}\label{eq:EL_alpha_beta_steklov}
\alpha \, \partial_g^\nu \hat{\psi} = \sigma_k(\alpha,\beta) \, \psi \,  \beta \qquad \text{on $\partial M$.}
\end{equation}
Elliptic regularity ensures that for $m$ high enough, if $\beta \in \cC^m(M)$ then $\phi \in \cC^2(M)$.

We consider the associated notations: $\{\mu_i(\alpha,\beta)\}$, $\aleph_i(\alpha,\beta)$, $J_N(\alpha,\beta)$ , $i_N(\alpha,\beta)$,  $d_k(\alpha,\beta)$, $c(\alpha,\beta)$, etc.  We recall definition \eqref{eq:def_quadric} of $\mathcal{Q}(\lambda,\eps,c)$.

\begin{theorem} \label{theocriticalsteklovn} 
Let $(M,g)$ be a compact, connected, smooth Riemannian manifold of dimension $n \ge 3$ with a non-empty boundary $\partial M$.  

\begin{enumerate}
\item For a high enough integer $m$, if $(\alpha,\beta) \in \cC_{>0}^m(M)\times \cC_{>0}^m(\partial  M)$ is critical for the spectral functional
\[ \mathfrak{F} : \cC_{>0}^m(M)\times \cC_{>0}^m(\partial  M) \ni (\alpha',\beta') \mapsto  F\left(\overline\sigma_1(\alpha',\beta') ,\cdots, \overline\sigma_N(\alpha',\beta') \right)\]
then there exists $\eps \in \{0,\pm 1\}^{J_N(\alpha,\beta)}$ such that the following hold.
\begin{enumerate}
\item There exists a $\cC^2$ mapping $\Phi : M \to \setR^{J_N(\alpha,\beta)}$ whose restriction $\Psi \df \left. \Phi \right|_{\partial M}$ has $\|\cdot\|_{L^2(\beta)}$-orthogonal coordinates and such that if $c(\alpha,\beta) \neq 0$, then 
 \[\Phi : (M,\partial M,g) \to \left(\setR^{J_N(\alpha,\beta)},\cQ\left(\sigma\left(\alpha,\beta\right),\eps, \frac{c(\alpha,\beta)}{\int_{\partial M}\beta dv_{\partial g}}\right),h_{c\eps}\right)\]
is free boundary $n$-harmonic and satisfies
$$
\frac{\alpha}{\|\alpha\|_{L^{\frac{n}{n-2}}(g)}^{\frac{2}{n}}}  = \langle d \Phi, d \Phi \rangle_{g,h_{c(\alpha,\beta) \eps}}^{\frac{n-2}{2}}  \qquad \text{on $M$}$$
$$\frac{\beta}{\|\alpha\|_{L^{\frac{n}{n-2}}(g)}^{\frac{2}{n}}} = \langle d \Phi, d \Phi \rangle_{g,h_{c(\alpha,\beta) \eps}}^{\frac{n-2}{2}} h_\eps(\Phi, \partial_g^\nu \Phi) \qquad \text{on $\partial M$}
$$
while if $c(\alpha,\beta) = 0$, then
$$
\langle d \Phi, d \Phi \rangle_{g,h_\eps} = 0  \qquad \text{on $M$}.
$$

\item There exists $\tilde{d} \in \prod_{i=1}^{i_N(\alpha,\beta)} \Mix(d_{\aleph_i(\alpha,\beta)})$ such that for any $1 \le k \le J_N(\alpha,\beta)$, the $k$-th coordinate $\psi_k$ of $\Psi$ satisfies
\begin{equation}\label{eq:normSteklovhigher}
\int_{\partial M} \psi_k^2 \di v_{\partial g} = \begin{cases} \,\,  \, |\tilde{d}_k| & \text{if $S(\alpha,\beta)=0$},\\
\displaystyle\frac{|\tilde{d}_k|}{\left| S(\alpha,\beta) \right|}  & \text{otherwise}. \, 
\end{cases}
\end{equation}

\end{enumerate}

\item Conversely, for $\eps \df (\eps_j) \in \{0,\pm 1\}^N$ and $\sigma\df (\sigma_j) \in [0,+\infty]^N$ any free boundary $n$-harmonic mapping 
$$ \Phi : (M,\partial M ,g) \to \left(\R^N,\mathcal{Q}(\sigma,\eps,1),h_\eps\right) $$
such that the function
$$ \alpha \df \langle d\Phi,d\Phi \rangle_{g,h_{\eps}}^{\frac{n-2}{2}} $$
is positive in $M$ and the function
$$ \beta \df \alpha h_\eps(\Phi,\partial_g^\nu \Phi)$$
is positive on $\partial M$, then $(\alpha,\beta)$ is a critical point of a combination of Steklov eigenvalues.

\end{enumerate}
\end{theorem}

\begin{proof} The proof is similar to the one of Theorem \ref{theocriticallaplacen} but we provide some details for convenience of the reader.
For the sake of brevity, we drop the notations $\di v_{\partial g}$ and $\di v_{g}$ in integrals over $\partial M$ and $M$ respectively. For any $(\alpha',\beta') \in \Omega \df \cC_{>0}^m(M)\times \cC_{>0}^m(\partial  M) \subset X \df \cC^m(M)\times \cC^m(\partial  M)$, $u \in Y\df L^2(\partial M)$ and $v \in H \df H^{1/2}(\partial M)$, set
\[
Q(\alpha',\beta',u) = \int_{\partial M} u^2 \beta' , \qquad G(\alpha',\beta',v) = \int_M |d\widehat{u}|_g^2 \alpha' ,
\]
\[
\|v\|^2_{H(\alpha',\beta')} = \int_{\partial M} v^2 \beta' + \int_M |d\widehat{v}|_g^2 \alpha' .
\]
Then
 \[
 Q_{(\alpha',\beta')}(\alpha',\beta',u) = \big( u^2 v_{\partial g}  , 0\big), \qquad G_{(\alpha',\beta')}(\alpha',\beta',u) = (0,  |d\widehat{u}|^2_g v_g)
 \]
 \[
 \scaling_{(\alpha',\beta')}(\alpha',\beta') =  \scaling(\alpha',\beta') \left( \frac{1}{\int_{\partial M} \beta' } \, v_{\partial g} \, , \,  - \frac{(\alpha')^{\frac{2}{n-2}} }{\int_M (\alpha')^{\frac{n}{n-2}}} \, v_g\right).
 \]
Hence Theorem \ref{th:main_cri} implies that there exist $\Psi \df (\psi_k) \in \cU_{J_N(\alpha,\beta)}(\alpha,\beta)$, $\tilde{d} \in \prod_{i=1}^{i_N(\alpha,\beta)} \Mix(d_{\aleph_{i}(\alpha,\beta)})$ and $(\eps_k) \in \{0,\pm 1\}^{J_N(\alpha,\beta)}$ such that
\begin{equation}\label{eq:identity1'}
\sum_{i=1}^{i_N(\alpha,\beta)} \mu_i(\alpha,\beta) \sum_{k \in \aleph_i(\alpha,\beta)} \eps_k \psi_k^2 = \frac{c(\alpha,\beta)}{\int_{\partial M} \beta} \qquad \text{on $\partial M$,}
\end{equation}
\begin{equation}\label{eq:identity2'}
\sum_{k=1}^{J_N(f,h)} \eps_k |d \widehat{\psi}_k|_g^2  =   \frac{\alpha^{\frac{2}{n-2}}}{\int_{M} \alpha^{\frac{n}{n-2}}} \, c(\alpha,\beta)  \qquad \text{on $M$.}
\end{equation}
Set
\[
\Phi \df (\widehat{\psi}_1, \ldots, \widehat{\psi}_{J_N(\alpha,\beta)}).
\]
Then the conclusion is obtained by combining in a natural way the final lines of the proofs of Theorem \ref{theoSteklov2conforme} and \ref{theocriticallaplacen} with the help of \eqref{eq:EL_free_boundary_modif} and \eqref{eq:EL_alpha_beta_steklov}. The converse is left to the reader.
\end{proof}

\begin{rem}
Remark \ref{rem:steklovn} also holds in the context of Theorem \ref{theocriticalsteklovn}.
\end{rem}

\section{Harmonic mappings through particular Rayleigh quotients} 
\label{sec:conf}

In this section, we apply the theory developed in Section 3 to the case where the quadratic form $G$ do not depend on the variation parameter while $Q$ depends on it only linearly.  This yields new results for the conformal criticality of eigenvalue functionals associated with conformally covariant operators. We let $N \in \setN^*$ and $F \in \cC^1(\setR^N)$ be kept fixed throughout. 

\subsection{General statement}

Let $(M,g)$ be a closed, connected, smooth Riemannian manifold.  We let $Y$ be the space of square-integrable (equivalent classes) of functions on $M$.  For any $\beta \in \cC_{>0}(M)$,  we define on $Y$ the norm
\[
\|u\|_{L^2(\beta)} \df \left( \int_M u^2 \beta\di v_g \right)^{1/2}.
\]
For brevity,  let us write $\|\cdot\|_{L^2}$ instead of $\|\cdot\|_{L^2(1)}$.  We point out that the norms $\{\|\cdot\|_{L^2(\beta)}\}_{\beta \in  \cC_{>0}(M)}$ are locally equivalent because $M$ is compact.  

Let $G$ be  a quadratic form on $Y$ with $\|\cdot\|_{L^2}$-dense domain $H$. For any $\beta \in  \cC_{>0}(M)$ and $u \in Y$, we set
\[
Q(\beta,u) = \int_M u^2  \beta \di v_{g}.
\]
Then the Rayleigh quotient
\[
\mathcal{R} :  \cC_{>0}(M) \times Y \ni (\beta,u) \mapsto \frac{G(u)}{Q(\beta,u)}
\]
satisfies Assumption (A).  We assume that  Assumptions (B), (C), (D) and (F) hold too. Assumptions (E) and (G) are trivially satisfied with the Fréchet derivatives of $G$ and $Q$ given by
\begin{equation}\label{eq:Fréch}
G_\beta(u) = 0 \qquad \text{and} \qquad Q_\beta(\beta,u) =  u^2 v_g
\end{equation}
for any $u \in H$ and $\beta \in  \cC_{>0}(M)$.  By Elementary Consequence (2),  for any $\beta \in  \cC_{>0}(M)$ there exists a self-adjoint operator $L(\beta)$ whose eigenvalues are given by
\[
\lambda_k (\beta) =  \min_{E \in \mathcal{G}_{k}(H)} \max_{u \in \uE} \frac{G(u)}{Q(\beta,u)}\,  ,\quad  k \in \setN^*.
\]

For brevity we write $L$ instead of $L(1)$.  We assume that $L$ is elliptic.  In particular, elliptic regularity theory ensures here again that if $\beta$ is $\cC^m$ for a high enough $m$, then the eigenfunctions of $L(\beta)$ are all $\cC^2$.

For any real number $p\ge 1$ and any $k \in \setN^*$ we set
\[
\overline{\lambda}_k (\beta) \df \scaling(\beta) \lambda_k (\beta) \qquad 
\text{with} \qquad
\scaling(\beta) \df \|\beta\|_{L^p} = \left( \int_M \beta^p \di v_g \right)^{1/p}.
\]
We point out that $\scaling$ is $\cC^1$-Fréchet differentiable at any $\beta \in \cC_{>0}(M)$ with Fréchet derivative given by 
\begin{equation}\label{eq:Fréchetsc}
\scaling_\beta(\beta) =  \frac{\beta^{p-1}}{\|\beta\|^{p-1}_{L^p(g)}}  \, v_g.
\end{equation}
We are then in a position to apply Theorem \ref{th:main_cri}.
Like in the previous sections, we use the usual associated notations: for any $\beta \in \cC_{>0}(M)$,  we set $d_k(\beta)\df \partial_k F(\overline{\lambda}_1(\beta),\ldots,(\overline{\lambda}_N(\beta))$, $c(\beta) \df  0$ if $\sum_{k=1}^N d_k(\beta) \lambda_k(\beta)=0$ otherwise $c(\beta)\df \sum_{k=1}^N d_k(\beta) \lambda_k(\beta)/ \left| \sum_{k=1}^N d_k(\beta) \lambda_k(\beta) \right|$,  and we recall the definition \eqref{eq:def_quadric} of $\mathcal{Q}(\lambda,\eps,c)$.

We define the energy
\[
E_{\eps}(\Psi) \df \frac{1}{2} \sum_{k=1}^{N}  \eps_k \int_M \psi_k [L \psi_k] \di v_g
\]
on the set of $\cC^2$ mappings $\Psi = (\psi_1,\ldots,\psi_N) : M \to \setR^N$ such that $\psi_k \in H$ for any $k$.  Since each $L$ is elliptic,  $L$-harmonic mappings from $(M,g)$ to $(\cQ(\Lambda,\eps,c),h_\eps)$ are defined like in \ref{subsubsec:L_harmonic_mappings}.

Our main result is the following.

\begin{theorem}\label{th:abstractconformal}
Let $(M,g)$ be a closed, connected, smooth Riemannian manifold. For a high enough integer $m$,  assume that $\beta \in \cC_{>0}^m(M)$ is critical for the spectral functional
\[
\mathfrak{F} : \cC_{>0}^m(M) \ni \beta \mapsto F(\overline{\lambda}_1 (\beta),\ldots,\overline{\lambda}_N (\beta)).
\]
Then there exist $\eps \in \{0,\pm 1\}^{J_N(\beta)}$ and a mapping
\[
\Phi  = (\phi_k) : M \to \setR^{J_N(\beta)}
\]
such that:
\begin{itemize}
\item[(i)] each $\phi_k$ belongs to $E_k(\beta)$,
\item[(ii)] the family $(\phi_k)$ is $\|\cdot\|_{L^2(\beta)}$-orthogonal,
\item[(iii)] the mapping $\Psi \df (\beta/\|\beta\|_{L^p(g)})^{-\frac{p-1}{2}} \Phi$ maps $M$ to $\cQ(\Lambda(\beta),\eps,\frac{c(\beta)}{\Vert \beta \Vert_{L^p(g)}})$.
\end{itemize}
If $p=1$ 
then $\Phi$ is  a $L$-harmonic mapping  from $(M,g)$ to $\cQ(\Lambda(\beta),\eps,\frac{c(\beta)}{\Vert \beta \Vert_{L^1(g)}})$. 
Conversely, for any $L$-harmonic mapping $\Phi :(M,g)\to \cQ(\Lambda,\eps,1)$ for $\Lambda \in \R^N$ and $\eps \in \{0,\pm 1 \}^N$ such that $\beta \df h_\eps\left(\Phi,L\Phi\right) $ is positive, $\beta$ is a critical point of a combination of eigenvalues with respect to $L$. 

 \end{theorem}

 \begin{proof}
Choose $m$ high enough to ensure that elliptic regularity applies.  Apply Theorem \ref{th:main_cri} knowing that  the Fréchet derivatives of $G$, $Q$ and $\scaling$ are given by \eqref{eq:Fréch} and \eqref{eq:Fréchetsc}. Then the rescaled Euler--Lagrange equation \eqref{eq:EL2} writes as
\begin{equation}\label{eq:EL3}
\sum_{i=1}^{i_N(\beta)} \mu_i(\beta) \sum_{k \in \aleph_i(\beta)} \eps_k \phi_k^2 = c(\beta)\frac{\beta^{p-1}}{\|\beta\|^p_{L^p(g)}}   \qquad \text{on $M$}
\end{equation}
where each $\phi_k$ belongs to $E_k(\beta)$ and $(\phi_k)$ is $\|\cdot\|_{L^2(\beta)}$-orthogonal.  Set $\Phi \df (\phi_k)$. Then \textit{(i)} and \textit{(ii)} holds. Set $\Psi \df (\beta/\|\beta\|_{L^p(g)})^{-\frac{p-1}{2}} \Phi$.
Then Equation \eqref{eq:EL3} yields
\begin{equation}\label{eq:EL5}
\sum_{i=1}^{i_N(\beta)} \mu_i(\beta) \sum_{k \in \aleph_i(\beta)} \eps_k \psi_k^2 = \frac{c(\beta)}{\|\beta\|_{L^p(g)}} \qquad \text{on $M$}
\end{equation}
which means that \textit{(iii)} is satisfied.  If $p=1$ 
equation \eqref{eq:EL3} reduces to 
\begin{equation}\label{eq:EL4}
\sum_{i=1}^{i_N(\beta)} \mu_i(\beta) \sum_{k \in \aleph_i(\beta)} \eps_k \phi_k^2 = \frac{c(\beta)}{\|\beta\|_{L^1(g)}} \qquad \text{on $M$}
\end{equation}
which implies that $\Phi$ is an $L$-harmonic from $(M,g)$ to $\cQ(\Lambda(\beta),\eps,\frac{c(\beta)}{\Vert \beta \Vert_{L^1(g)}})$, see \ref{subsubsec:L_harmonic_mappings}. The proof of the converse is left to the reader.  
\end{proof}

\subsection{Application to GJMS operators}

Let $(M,g)$ be a closed, connected, smooth Riemannian manifold of dimension $n$. For any integer $m$ belonging to
\[
I(n) \df \begin{cases}
\setN^* & \text{if $n$ is odd,}\\
\{1,\ldots, n/2\} & \text{if $n$ is even,}
\end{cases}
\]
the associated GJMS operator $J_{g,m}$ of order $m$ is a self-adjoint, elliptic, conformally covariant operator acting on $\cC^\infty(M)$ with leading order term $\Delta_{g}^m$ \cite{GJMS}.  This operator admits a discrete spectrum $\{\lambda_k(g)\}_{k \in \setN^*}$ given by the usual min-max formula: for any $k \in \setN^*$,
\begin{align*}
\lambda_k(g) & = \min_{E \in \cG_k(H^{m}(M))} \max_{u \in \underline{E}} \frac{\int_M u J_{g,m} u \di v_{g}}{\int_M u^2 \di v_{g}}  \, \cdot
\end{align*}
If $\tilde{g}=f g$ for some $f \in \cC_{>0}^\infty(M)$, then the conformal covariance rule of $J_{g,m}$ states that for any  $u \in \cC^\infty(M)$,
\[
J_{\tilde{g},m} (u) = f^{-\frac{n+2m}{4}} J_{g,m} (f^{\frac{n-2m}{4}}u).
\] This implies that for any $u \in H^m(M)$,
\begin{align*}
 \frac{\int_\Sigma u J_{\tilde{g},m} u \di v_{\tilde{g}}}{\int_\Sigma u^2  \di v_{\tilde{g}}}  =  \frac{\int_\Sigma v J_{g,m} v \di v_{g}}{\int_\Sigma v^2 f^{m} \di v_{g}} \, \cdot
\end{align*}
where $v=f^{\frac{n-2m}{4}} u$. In view of this, we define
\begin{align*}
\lambda_k(\beta) & \df \min_{E \subset \cG_k(H^{1}(M))} \max_{u \in \underline{E}} \frac{\int_M u J_{g,m}u \di v_{g}}{\int_M u^2 \beta \di v_{g}}\\
 \scaling(\beta) & \df \|\beta\|_{L^{n/(2m)}(g)}\\
\overline{\lambda}_k(\beta) &  \df \scaling(\beta)\lambda_k(\beta)
\end{align*}
for any $\beta \in C_{>0}(M)$.   Then Theorem \ref{th:abstractconformal} directly implies the following result, where  we use the usual notation.

\begin{theorem}\label{th:general}
Let $(M,g)$ be a closed, connected, smooth Riemannian manifold of dimension $n$, and $m$ an integer belonging to $I(n)$. For a high enough $\ell$,  assume that $\beta \in \cC_{>0}^\ell(M)$ is critical for the spectral functional
\[
\fk : \cC_{>0}^\ell(M) \ni  \beta' \mapsto F(\overline{\lambda}_1 (\beta'),\ldots,\overline{\lambda}_N (\beta')),
\]
then there exist $\eps \in \{0,\pm 1\}^{J_N(\beta)}$ and a mapping
\[
\Phi : M \to \setR^{J_N(\beta)}
\]
such that:
\begin{itemize}
\item[(i)] each $\phi_k$ belongs to $E_k(\beta)$,
\item[(ii)] the family $(\phi_k)$ is $\|\cdot\|_{L^2(\beta)}$-orthogonal,
\item[(iii)] the rescaled mapping $\Psi\df (\beta/ \|\beta\|_{L^{\frac{n}{2m}}(g)})^{-\frac{n-2m}{4m}} \Phi$ is a $J_{\tilde{g},m}$-harmonic map from $M$ to $\cQ\left(\Lambda(\beta),\eps,\frac{c(\beta)}{\Vert \beta \Vert_{L^{\frac{n}{2m}}(g)}}\right)$, where $\tilde{g} = \beta^{\frac{1}{m}} g$
\end{itemize}
Conversely, any  $J_{g,m}$-harmonic map  $ \Psi : M \to \cQ\left(\Lambda,\eps,1 \right)$ for $\Lambda \in \R^N$ and $\eps \in \{0,\pm 1 \}^N$ such that $\beta \df h_\eps\left(\Psi,J_{g,m}\Psi\right)$ is positive, $\beta$ is a critical point of a combination of eigenvalues with respect to $J_{g,m}$. In other words, the metric $\tilde{g} = \beta^{\frac{1}{m}}g$ is a critical metric in the conformal class of $g$.
\end{theorem}

In the sequel, we deduce two corollaries from Theorem \ref{th:general}.

\subsubsection*{Paneitz operator}

Let $(M,g)$ be a closed, connected, smooth Riemannian $4$-manifold.  For any $\tilde{g} \in [g]$, the associated Paneitz operator
\[
P_{\tilde{g}} \df \Delta_{\tilde{g}}^2 - \delta_{\tilde{g}} \left( \frac{2}{3} R_{\tilde{g}} g - 2 \Ric_{\tilde{g}} \right) \nabla^{\tilde{g}} 
\]
is a fourth-order, self-adjoint, elliptic operator acting on $H^2(M)$.    As well-known, 
\[
P_{\tilde{g}} = J_{\tilde{g},2}.
\]
For any $\beta \in \cC_{>0}(M)$, we set
\[
\lambda_k(\beta) \df \min_{E \in \cG_k(H^{2}(M))} \max_{u \in \underline{E}} \frac{\int_M u P_{g} u \di v_{g}}{\int_M u^2 \beta \di v_{g}}
\]
\[
\overline{\lambda}_k(\beta)  \df \scaling(\beta)\lambda_k(\beta) \qquad \text{with} \qquad \scaling(\beta) = \|\beta\|_{L^1(g)}.
\]
Then Theorem \ref{th:general} yields the following.

\begin{cor}
Let $(M,g)$ be a closed, connected, smooth, four-dimensional Riemannian manifold. For a high enough integer $\ell$, if $\beta \in \cC_{>0}^\ell(M)$ is critical for $\fk$,  then there exists a $P_{g}$-harmonic mapping $\Phi$ from $(M,g)$ to some pseudo-Euclidean quadric $(\cQ(\Lambda(\beta),\eps,\frac{c(\beta)}{\int_M \beta}),h_\eps)$, where $\eps \in \{0,\pm 1\}^{J_N(\beta)}$,  and the coordinates of $\Phi$ are $\|\cdot\|_{L^2(\beta)}$-orthogonal eigenfunctions of $P_g$.

Conversely, if $\Phi : M \to (\cQ(\Lambda,\eps,1),h_\eps) $ is a $P_g$-harmonic mapping $\Lambda \in \R^N$ and $\eps \in \{0,\pm 1 \}^N$ such that 
$$ \beta  \df h_\eps\left(\Phi,P_g \Phi\right)  $$
also expressed as
\begin{equation*}
\begin{split} h_\eps(\Lambda \Phi,\Lambda\Phi) \beta = -h_\eps(\Lambda \Delta\Phi,\Delta\Phi) + 2 \langle \nabla \Phi \nabla \Delta \Phi \rangle_{g,h_\eps} + \Delta_g\left( \langle \nabla \Phi, \nabla \Lambda\Phi \rangle_{g,h_\eps} \right) \\
 +  \frac{2}{3}R_g \langle \nabla \Phi,\nabla \Lambda \Phi \rangle_{g,h_\eps} + 2 \langle \nabla \Phi,\nabla \Lambda \Phi \rangle_{\Ric_g,h_\eps} 
\end{split} \end{equation*}
is positive, $\beta$ is a critical point of a combination of eigenvalues $\overline{\lambda}_k$ and the metric $\beta^{\frac{1}{2}} g$ is critical in the conformal class of $g$.
\end{cor}

\subsubsection*{Conformal Laplacian} Let $(M,g)$ be a closed, connected, smooth Riemannian manifold of dimension $n\ge 3$.  For any $\tilde{g} \in [g]$, the associated conformal Laplacian
\begin{equation}\label{eq:conformal_Laplacian}
L_{\tilde{g}} \df \Delta_{\tilde{g}} + \frac{n-2}{4(n-1)} R_{\tilde{g}}.
\end{equation}
is a second-order, self-adjoint, elliptic operator acting on $H^1(M)$.   As well-known, 
\[
L_{\tilde{g}} = J_{\tilde{g},1}.
\]
For any $\beta \in \cC_{>0}(M)$, we set
\[
\lambda_k(\beta) \df \min_{E \in \cG_k(H^{1}(M))} \max_{u \in \underline{E}} \frac{\int_M u L_{g} u \di v_{g}}{\int_M u^2 \beta \di v_{g}}
\]
\[
\overline{\lambda}_k(\beta)  \df \scaling(\beta)\lambda_k(\beta) \qquad \text{with} \qquad \scaling(\beta) = \|\beta\|_{L^1(g)}.
\]
Then Theorem \ref{th:general} yields the following.

\begin{cor}
Let $(M,g)$ be a closed, connected, smooth Riemannian manifold of dimension $n \ge 3$. For a high enough integer $\ell$,  if $\beta \in \cC^\ell_{>0}(M)$ is critical for $\fk$, then there exists a mapping $\Phi : M \to \setR^{J_N(g)}
$
whose coordinates are $\|\cdot\|_{L^2(\beta)}$-orthogonal eigenfunctions of $L_g$ and such that the rescaled mapping
\[
\Psi \df (\beta/\|\beta\|_{L^{\frac{n}{2}}(g)})^{-\frac{n-2}{4}} \Phi
\]
is a $L_{\tilde{g}}$-harmonic map from $M$ to $\cQ(\Lambda(\beta),\eps,\frac{c(g)}{\|\beta\|_{L^1(g)}})$, where $\tilde{g}\df \beta g$, $\eps \in \{0,\pm 1\}^{J_N(\beta)}$.
Conversely, if $\Phi : M \to (\cQ(\Lambda,\eps,1),h_\eps) $ is a $L_g$-harmonic mapping $\Lambda \in \R^N$ and $\eps \in \{0,\pm 1 \}^N$ such that 
$$\beta \df h_\eps\left(\Phi,L_g \Phi\right) = \frac{\langle\nabla \Lambda \Phi,\nabla \Phi\rangle_{g,h_\eps}+\frac{n-2}{4(n-1)}R_g}{h_\eps\left(\Lambda\Phi,\Lambda\Phi\right)}$$ 
is positive, $\beta$ is a critical point of a combination of eigenvalues $\overline{\lambda}_k$ and the metric $\beta g$ is critical in the conformal class of $g$.
\end{cor}

\begin{rem}
Notice that if criticality holds for a single eigenvalue, then the target manifold is a sphere and $L_g$-harmonic mappings are $\Delta_g$-harmonic mappings.
\end{rem}

\begin{rem}
We believe that Theorem \ref{th:abstractconformal} will also find natural applications to the fractional GJMS operators considered in \cite{CaseChang} and the boundary operators introduced in \cite{Case,CaseLuo}.
\end{rem}

\section{Conformal Laplacian}\label{sec:conformal}

In this short section, we focus on eigenvalue functionals associated with the conformal Laplacian on a closed, connected, smooth manifold $M$ of dimension $n\ge 3$.  The study of critical points in a given conformal class is covered by the previous section, but the concrete formula \eqref{eq:conformal_Laplacian} allows us to treat the case of critical points in the whole set $\cR(M)$. This is what we do in this section.

We fix a Riemannian metric $g \in \cR(M)$. For any $(g',\beta) \in \cR^2(M) \times \cC_{>0}(M)$, we set
\begin{align*}
\lambda_k^C(g',\beta) & \df \min_{E \in \cG_k(H^1(M))} \max_{u \in \uE} \frac{\int_M |du|_{g'}^2 + c_n R_{g'} u^2 \di v_{g'}}{\int_M u^2 \beta \di v_g}\\
\scaling(g',\beta) & \df \|\beta\|_{L^{\frac{n}{2}}(g)} v_{g'}(M)^{2/n-1} \\
\overline{\lambda}_k^C(g',\beta) & \df  \lambda_k^C(g',\beta)\scaling(g',\beta),
\end{align*}
where $k \in \setN^*$.   The conformal covariance rule of the conformal Laplacian writes as
\[
L_{\beta g} u = \beta^{-\frac{n+2}{4}} L_g (\beta^{\frac{n-2}{4}} u)
\]
for any $\beta \in \cC^\infty(M)$ and $u \in H^1(M)$ and implies that $\lambda_k^C(g,\beta)$ coincides with the $k$-th lowest eigenvalue of $L_{\beta g}$. As a consequence, we adapt the notation and writes $\beta g$ instead of $(g,\beta)$ whenever possible. We let $N \in \setN^*$ and $F \in \cC^1(\setR^N)$ be fixed, as usual, and we use the notations introduced in the previous sections. Then the following holds.

\begin{theorem}
Let $(M,g)$ be a closed, connected, smooth Riemannian manifold of dimension $n\ge 3$.  For a high enough integer $m$, if $\beta \in \cC_{>0}^m(M)$ is such that the couple $(g,\beta)$ is critical for the spectral functional
\[
\fk : \cR^m(M) \times \cC_{>0}^m(M) \ni (g',\beta') \mapsto F(\lambda_1^C(g',\beta'), \ldots,\lambda_N^C(g',\beta')),
\]
then there exists $\eps \in \{0,\pm 1\}^{J_{N}(\beta g)}$ such that the following hold.  
\begin{enumerate}
\item[(a)] There exists a $\cC^2$ mapping
\[
\Phi = (\phi_k) : M \to \setR^{J_N(\beta g)}
\]
such that:
\begin{itemize}
\item[(i)] each $\phi_k$ belongs to $E_k(\beta g)$,
\item[(ii)] the family $(\phi_k)$ is $\|\cdot\|_{L^2(\beta)}$-orthogonal,
\item[(iii)] $\Psi \df (\beta/\|\beta\|_{L^{\frac{n}{2}}(g)})^{-\frac{n-2}{4}} \, \Phi : M\to \cQ(\Lambda(\beta g), \eps, \frac{c(\beta g)}{\Vert \beta \Vert_{L^{\frac{n}{2}}}^{\frac{n}{2}}})$ is $L_{\beta g}$-harmonic
\item[(iv)] the following identities hold:
\begin{equation}\label{eq:Einstein1}
\Phi^*h_\eps = \frac{\langle d \Phi, d \Phi\rangle_{g,h_\eps}}{n}  g + \frac{c_n}{n} h_\eps(\Phi,\Phi) \mathring{\Ric}_g
\end{equation}
\begin{equation}\label{eq:Einstein2}
\langle d \Phi, d \Phi\rangle_{g,h_\eps} = \frac{n}{n-2} \frac{c(g,\beta)}{v_g(M)} - c_n R_g h_\eps(\Phi,\Phi).
\end{equation}
In particular, if $g$ is Einstein and $\langle d \Phi, d \Phi\rangle_{g,h_\eps}>0$ on $M$, then $\Phi$ is a conformal immersion.
\end{itemize}

\item[(b)] There exists $\tilde{d} \in \prod_{i=1}^{i_N(\beta g)} \Mix(d_{\aleph_i(\beta g)})$ such that for any $1 \le k \le J_N(\beta g)$, the $k$-th coordinate $\phi_k$ of $\Phi$ satisfies
\begin{equation}\label{eq:normconformalLap}
\int_M \phi_k^2 \di v_g =\begin{cases}
\,\,\, |\tilde{d}_k| & \text{if $S(\beta g)=0$,}\\
\displaystyle \frac{|\tilde{d}_k|}{|S(\beta g)|} & \text{otherwise.}
 \end{cases}
\end{equation}
\end{enumerate}
\end{theorem}

\begin{proof}
Set $X \df \cS^m(M) \times \cC^m(M)$ and $\Omega\df \cR^m(M) \times \cC_{>0}^m(M)$ for a high enough integer $m$. For any $(g',\beta') \in \Omega$ and $u \in L^2(M)$, $v \in H^1(M)$, define
\[
Q(g',\beta',u) \df \int_M u^2 \beta' \di v_{g} \quad \text{and} \quad G(g',\beta',u) \df \int_M |du|_{g'}^2 + c_n R_{g'} u^2 \di v_{g'}
\]
From previous arguments, we know that $Q$ and $\scaling$ are Fréchet differentiable at $(g,\beta)$, with
\[
Q_{(g,\beta)}(g,\beta,u) = \left( 0, u^2 v_g \right),
\]
\[
\frac{\scaling_{(g,\beta)}(g,\beta,u)}{\scaling(g,\beta,u)} = \left( \frac{n}{2 v_g(M)}\, g, \frac{\beta^{\frac{n}{2}-1}}{\|\beta\|_{L^{\frac{n}{2}}}^{\frac{n}{2}}(g)} v_g \right).
\]
In addition, since for any $h \in \cS^2(M)$,
\[
\left.  \frac{\di}{\di t}\right|_{t=0} \int_M R_{g+th} \di v_{g+th} = \int_M \left\langle -\Ric_g + \frac{R_g}{2} \, g, h \right\rangle_g \di v_g,
\]
(see e.g.~\cite[proof of Proposition 1.1]{Via}),  also $G$ is Fréchet differentiable at $(g,\beta)$, with
\[
G_{(g,\beta)}(g,\beta,u) = \left( - d u \otimes d u + \frac{1}{2} |du|^2_g g  - c_n \Ric_g u^2 + \frac{c_n}{2}R_gu^2 g,  \,\, 0 \right).
\]
One can easily show that Assumptions (A)--(F) holds. Then Theorem \ref{th:main_cri} applies and gives the existence of $\Phi \in \cU_{J_{N}(\beta g)}(\beta g)$, $\tilde{d} \in \prod_{i=1}^{i_N(\beta g)} \Mix(d_{\aleph_i(\beta g)})$ and $\eps \in \{0,\pm 1\}^{J_{N}(\beta g)}$ such that \eqref{eq:normconformalLap} holds and
\begin{align}\label{eq:premièreequation}
& \sum_{k=1}^{J_{N}(\beta g)} \eps_k \left(- d \phi_k \otimes d \phi_k + \frac{1}{2} |d\phi_k|^2_g g  - c_n \Ric_g \phi_k^2 + \frac{c_n}{2}R_g\phi_k^2 g \right)  =  \frac{n c(\beta g)}{2 v_g(M)}\,    g 
\end{align}
\begin{align}\label{eq:secondequation}
& \frac{n}{2} \sum_{i=1}^{i_N(\beta g)} \mu_i(\beta g) \sum_{k \in \aleph_i(\beta g)} \eps_k \phi_k^2  = \frac{c(\beta g)}{\|\beta\|^{\frac{n}{2}}_{L^{\frac{n}{2}}(g)}} \beta^{\frac{n}{2}-1}.
\end{align}
Note that \textit{(i)} and \textit{(ii)}  follow from the fact that $\Phi \in \cU_{J_{N}(\beta g)}(\beta g)$.  Moreover, \textit{(iii)} is obtained from mutiplying \eqref{eq:secondequation} by $(\beta/\|\beta\|_{L^{\frac{n}{2}(g)}})^{-\frac{n}{2}-1}$ and setting $\Psi$ as indicated. Lastly,  in \textit{(iv)},  \eqref{eq:Einstein1} and \eqref{eq:Einstein2}  are obtained by taking the traceless part and the trace of \eqref{eq:premièreequation}, respectively. If $g$ is Einstein then $\mathring{Ric}_g=0$ so that \eqref{eq:Einstein1} becomes $\Phi^*h_\eps = (1/n)\langle d \Phi, d \Phi\rangle_{g,h_\eps} \, g$.
\end{proof}

\section{Laplace spectral functionals with mixed boundary conditions}\label{sec:mixed}

In this section, we consider spectral functionals involving Laplace eigenvalues with different boundary conditions on a smooth, connected, compact surface $\Sigma$ with a non-empty boundary $\partial\Sigma$.  For a finite partition of $\partial \Sigma$ by the images $\{A_i\}$ of smooth connected curves:
\[
\partial \Sigma = \sqcup_{i=1}^L A_i
\]
we let $H_{i}^1(\Sigma)$ be the completion of $\cC_c^\infty(\mathring{\Sigma} \cup A_i)$ with respect to any $H^1$ norm on $\Sigma$ and we define, for any $k \in \setN^*$ and $g \in \cR^2(\Sigma)$,
\[
\lambda_k^{A_i}(g) \df \min_{E \in \cG_k(H_{i}^1(\Sigma))} \max_{u \in \uE} \frac{\int_M |du|^2_g \di v_g}{\int_M u^2 \di v_g} \cdot
\]
We also define the normalized version of these numbers
\[
\overline{\lambda}_k^{A_i}(g) \df \lambda_k^{A_i}(g) v_g(\Sigma)
\]
and the classical Dirichlet and Neumann quantities
\[
\overline{\lambda}_k^{\cD}(g)  \df  \overline{\lambda}_k^{\partial \Sigma}(g) \qquad \text{and} \qquad \overline{\lambda}_k^{\mathcal{N}}(g) \df \overline{\lambda}_k^\emptyset(g).
\]
We let $N \ge L +1$ be an integer and consider $F \in \cC^1(\setR^N)$.  For any $g \in \cR^2(\Sigma)$, we set
\[\fk(g) \df F(\lambda_1^{\cD}(g), \lambda_1^{A_1}(g),\ldots,\lambda_1^{A_L}(g),\lambda_1^{\mathcal{N}}(g),\ldots,\lambda_{N-L-1}^{\mathcal{N}}(g))
\]
and we define the numbers $d_k(g)$ in a similar way as in the previous sections.   We set
\begin{align*}
S(g)&  \df \left(d_1(g)\lambda_1^{\cD}(g) + d_2(g)\lambda_1^{A_1}(g)+\ldots+d_{L+1}(g)\lambda_1^{A_L}(g) \right)\\
& + \sum_{k=1}^{J_{N-L-1}(g)} d_{L+1+k}(g) \lambda_k^\mathcal{N}(g)
\end{align*}
and we define $c(g)$ as the sign of $S(g)$ (equal to $0$ when $S(g)=0$). We also introduce the set of indices
\[
I(g) \df \bigg\{ k \in \{1,\ldots,L+1\} \, : \,  d_k(g) \, \neq \, 0\bigg\},
\]
and for any $\eps \in \{0,\pm 1\}^{L+1+J_N^\mathcal{N}(g)}$,  we let $\cD\cQ(\lambda,\eps,c,I)$ be the positive piece of the dyadic subdivision of $\cQ(\lambda,\eps,c)$ (see definition \ref{eq:def_quadric}) associated with a set of indices $I$, that is to say,
\[
\mathcal{D}\cQ(\lambda,\eps,c,I) \df \left\{ x \in \cQ(\lambda,\eps,c) \, : \,  x_i > 0 \, \,   \text{if $i \in I$} \right\}
\]
The boundary of $\mathcal{D}\cQ(\lambda,\eps,c,I)$ in $\cQ(\lambda,\eps,c)$ is defined as
\[
\partial \mathcal{D}\cQ(\lambda,\eps,c,I) \df \left\{ x \in \cQ(\lambda,\eps,c) \, : \,  x_i = 0 \, \,   \text{if $i \in I$} \right\}.
\]

\begin{theorem}\label{th:main8}
Let $\Sigma$ be a smooth, connected, compact surface with a non-empty boundary $\partial\Sigma= \sqcup_{i=1}^L A_i$ where each $A_i$ is the image of a smooth connected curve.

\begin{enumerate}
\item For a high enough integer $m$, if $g \in \cR^m(\Sigma)$ is critical for the restriction of $\fk$ to $\cR^m(\Sigma)$ and such that $c \df c(g)\neq 0$, then then there exists a free boundary minimal immersion
\[
\Phi : (\Sigma,\partial \Sigma, g) \to \left(\mathcal{D}\cQ\left(\Lambda(g),\eps,\frac{c(g)}{v_g(\Sigma)}\right),\partial \mathcal{D}\cQ\left(\Lambda(g),\eps,\frac{c(g)}{v_g(\Sigma)}, I(g)\right),h_{c\eps}\right)
\]
where 
\begin{equation}\label{eq:eigenvalues}
\Lambda(g) \df \left(\lambda_1^{\cD}(g), \lambda_1^{A_1}(g),\ldots,\lambda_1^{A_L}(g),\lambda_1^{\mathcal{N}}(g),\ldots,\lambda_{N-L-1}^{\mathcal{N}}(g)\right)
\end{equation}
with $\|\cdot\|_{L^2(g)}$-orthogonal last $J_{N-L-1}(g)$ coordinates, for $\eps \in \{0,\pm 1\}^{L+1+J_N^\mathcal{N}(g)}$.

\item Conversely, for $\lambda \in \R^N$, $\eps \in \{0,\pm 1\}$, $I \subset \{1,\cdots,N\}$ and any conformal free boundary minimal immersion 
$$\Phi : (\Sigma,\partial \Sigma, g) \to \left(\mathcal{D}\cQ(\lambda,\eps,1),\partial \mathcal{D}\cQ(\lambda,\eps,1,I),h_{\eps}\right)$$
then the metric 
$$ h_\eps(\Phi,\Delta_{\Phi^* h_\eps} \Phi) \Phi^* h_\eps   $$
is critical for a combination of Laplace eigenvalues with mixed Dirichlet and Neumann boundary conditions.
\end{enumerate}
\end{theorem}

\begin{proof}
First of all, we point out that the eigenvalues $\lambda_1^{\cD}(g), \lambda_1^{A_1}(g),\ldots,\lambda_1^{A_L}(g)$ are all simple, and that the associated eigenvalues have a constant sign on $\mathring{\Sigma}$. Secondly, we choose $m$ high enough to ensure that the eigenfunctions corresponding to the eigenvalues \eqref{eq:eigenvalues} are all $\cC^2$. Set $X \df \cS^m(\Sigma)$,  $\Omega \df \cR^m(\Sigma)$,  $Y \df L^2(\Sigma)$ and,  for any $g \in \cR(M)$, $u \in L^2(M)$ and $v \in H^1(M)$, 
\[
Q(g,u) \df \int_M u^2 \di v_g\qquad
G(g,v)  \df \int_M |dv|_g^2 \di v_g \qquad\cR(g,v)  \df \frac{G(g,v)}{Q(g,v)}\, \cdot
\]
Adapting the proof of Proposition \ref{prop:diff} in a natural way, we obtain that
\begin{align*}
\partial_C^\pm \fk(g) & \subset \scaling(g) \sum_{\ell=1}^{L+1} d_\ell(g) \cR_g(g,u_\ell)\\
&  + \scaling_g(g) \left(d_1(g)\lambda_1^{\cD}(g) + d_2(g)\lambda_1^{B_1}(g)+\ldots+d_{L+1}(g)\lambda_1^{B_L}(g) \right)\\
& + \scaling(g) \conv \left\{  \sum_{k=1}^{J_{N-L-1}(g)} d_{L+1+k}(g) \mathcal{R}_g(g,u^\mathcal{N}_k) : (u^\mathcal{N}_k) \in \mathbf{U}_{N-L-1}^{\mathcal{N}}(g) \right\} \\
& + \scaling_g(g) \sum_{k=1}^{J_{N-L-1}(g)} d_{L+1+k}(g) \lambda_k^\mathcal{N}(g)
\end{align*}
where $u_1$ is an $\|\cdot\|_{L^2(g)}$-normalized eigenfunction associated with $\lambda_1^{\cD}(g)$,  $u_{1+\ell}$ is an $\|\cdot\|_{L^2(g)}$-normalized eigenfunction associated with $\lambda_1^{B_\ell}(g)$ for any $\ell \in \{1,\ldots,L\}$,  and $\mathbf{U}_{N-L-1}^{\mathcal{N}}(g)$ is the set of $\|\cdot\|_{L^2(g)}$-orthonormal $(N-L-1)$-uples $(u_k^\mathcal{N}) \subset H^1(\Sigma)$ such that $\Delta_g u_k^\mathcal{N} = \lambda_k^\mathcal{N}(g) u_k^\mathcal{N}$ for any $k\in \{1,\ldots,J_{N-L-1}(g)\}$. Then we apply the mixing Lemma \ref{lem:mixing} to rewrite the previous convex hull as 
\[
\left\{  \sum_{k=1}^{J_{N-L-1}(g)} \tilde{d}_{L+1+k} \mathcal{R}_g(g,u^\mathcal{N}_k) : \tilde{d} \in \prod_{i=1}^{i_{N-L-1}(g)} \Mix(d_{\aleph_i(g)}), (u^\mathcal{N}_k) \in \cU_{N-L-1}^{\mathcal{N}}(g) \right\} \, \cdot
\]
By criticality of $g$, we obtain that there exists
\[
\overline{\Phi} = (\overline{\phi}_1,\ldots,\overline{\phi}_{L+1},\overline{\phi}^\mathcal{N}_1,\ldots,\overline{\phi}^\mathcal{N}_{N-L-1})
\]
such that
\begin{align*}
0 & =  \sum_{\ell=1}^{L+1} d_\ell(g) \cR_g(g,\overline{\phi}_\ell)+ \sum_{k=1}^{J_{N-L-1}(g)} \tilde{d}_{L+1+k} \mathcal{R}_g(g,\overline{\phi}^\mathcal{N}_k) + \frac{\scaling_g(g)}{\scaling(g)} S(g) .
\end{align*}
We rescale each coordinate of $\overline{\Phi}$ as in the proof of Theorem \ref{th:main_cri} to obtain a new map $\Phi$ whose coordinates $\phi_\ell$, $\phi_k^\mathcal{N}$ satisfy
\begin{align}\label{eq:au-dessus}
& \phantom{=} \sum_{\ell=1}^{L+1} \eps_\ell G_g(g,\phi_\ell) + \sum_{k=1}^{J_{N-L-1}(g)} \eps_{k+L+1} G_g(g,\phi_k) +  \frac{\scaling_g(g)}{\scaling(g)} c(g) \nonumber \\
& = \sum_{\ell=1}^{L+1} \eps_\ell Q_g(g,\phi_\ell^\mathcal{N}) + \sum_{k=1}^{J_{N-L-1}(g)} \eps_{k+L+1} Q_g(g,\phi_k^\mathcal{N})
\end{align}
for some $\eps \in \{0,\pm 1\}^{L+J_{N-L-1}(g)}$ such that $\eps_\ell$ is the sign of $d_\ell(g)$ for any $\ell \in \{1,\ldots,L+1\}$. As already computed in the previous section, one has
\[
Q_g(g,\phi) =  \frac{1}{2} \phi^2 g \qquad G_g(g,\phi)  = - d\phi \otimes d\phi + \frac{1}{2}|d\phi|_g^2 g \qquad \scaling_g(g) = \frac{1}{2} \, g,
\]
for any $\phi \in H_1(\Sigma)$. By taking the trace and the traceless part of \eqref{eq:au-dessus} respectively, we obtain
\begin{equation}\label{eq:tracemixed}
\sum_{\ell=1}^{L+1} \eps_\ell \phi_\ell^2 + \sum_{k=1}^{J_{N-L-1}(g)} \eps_{k+L+1} (\phi_k^\mathcal{N})^2 = \frac{c}{v_g(\Sigma)}
\end{equation}
\[
\sum_{\ell=1}^{L+1} c\eps_\ell d \phi_\ell \otimes d \phi_\ell + \sum_{k=1}^{J_{N-L-1}(g)} c\eps_{k+L+1} d \phi_k^\mathcal{N} \otimes d \phi_k^\mathcal{N}  = \frac{1}{2}|d \Phi|_{g,h_c\eps}^2g
\]
where we have set
\[
|d \Phi|_{g,h_{c\eps}}^2 \df \left( \sum_{\ell=1}^{L+1} c\eps_\ell | d \phi_\ell|^2_g + \sum_{k=1}^{J_{N-L-1}(g)} c\eps_{k+L+1} |d \phi_k^\mathcal{N} |^2_g \right).
\]
With no loss of generality, we can assume that the functions $\overline{\phi}_\ell$ are all positive on $\mathring{\Sigma}$.  Since these functions vanish on $\partial \Sigma$,  we obtain from \eqref{eq:tracemixed} that $\Phi$ maps $M$ to $\cD\cQ(g,\eps)$ and that $\Phi(\partial M) \subset \partial\cD\cQ(g,\eps)$. Moreover, taking the normal derivative of \eqref{eq:tracemixed} on $\partial M$, we obtain that $\Phi(M)$ meets $\partial \cD\cQ(g,\eps)$ orthogonally. By \ref{subsec:mappings}, we conclude that $\Phi$ is free boundary and minimal, as desired.
\end{proof}

\bibliographystyle{alpha} 
\bibliography{Biblio.bib}

\end{document}